\documentclass[14pt]{article}
\usepackage{mathpazo}
\usepackage{amsmath}
\usepackage{amsfonts}
\usepackage{array,color}
\usepackage{abstract}
\usepackage[bottom]{footmisc}
\usepackage[left=1.15in, right=1.15in, bottom=0.60in, includefoot]{geometry}
\usepackage{graphicx}
\usepackage{indentfirst}
\usepackage{amsthm}
\usepackage{mathrsfs}
\usepackage{makeidx}
\usepackage{url}
\usepackage{dsfont}
\usepackage{latexsym}
\usepackage{amsmath}
\usepackage{amssymb}
\usepackage{color}
\usepackage{hyperref}
\usepackage{tikz}
\usepackage{MnSymbol}
\numberwithin{equation}{subsection}

\renewcommand{\P }{\mathcal P}

\newcommand{\D}{\mathcal D}
\newcommand{\M}{\mathcal M}

\newcommand{\A}{\mathcal A}
\newcommand{\B}{\mathcal B}

\newcommand{\sU}{\mathscr U}

\newcommand{\G}{\Gamma}
\newcommand{\g}{\gamma}
\newcommand{\sg}{\sigma}
\newcommand{\mc}{\mathbb{C}}
\newcommand{\ca}{\curvearrowright}

\newcommand{\emm}{\mathcal{M}}
\newcommand{\emmt}{\tilde{\mathcal{M}}}
\newcommand{\enn}{\mathcal N}
\newcommand{\euu}{\mathcal{U}}
\newcommand{\la}{\langle}
\newcommand{\ra}{\rangle}
\newcommand{\El}{\mathcal{L}}
\newcommand{\Sg}{\Sigma}
\newcommand{\rar}{\rightarrow}
\newcommand{\La}{\Lambda}
\newcommand{\lam}{\lambda}
\newcommand{\bten}{\bar\otimes}
\newcommand{\cM}{\mathcal M}
\newcommand{\cB}{\mathcal B}
\newcommand{\De}{\Delta}
\newcommand{\tp}{\bar{\otimes}}
\newcommand{\Q}{\mathcal Q}

\begin{document}
	\newtheorem{Lemma}{Lemma}
	\theoremstyle{plain}
	\newtheorem{theorem}{Theorem~}[section]
	\newtheorem{main}{Main Theorem~}
	\newtheorem{lemma}[theorem]{Lemma~}
	\newtheorem{assumption}[theorem]{Assumption~}
	\newtheorem{proposition}[theorem]{Proposition~}
	\newtheorem{corollary}[theorem]{Corollary~}
	\newtheorem{definition}[theorem]{Definition~}
	\newtheorem{defi}[theorem]{Definition~}
	\newtheorem{notation}[theorem]{Notation~}
	\newtheorem{example}[theorem]{Example~}
	\newtheorem*{remark}{Remark~}
	\newtheorem*{cor}{Corollary~}
	\newtheorem*{question}{Open Problem}
	\newtheorem{claim}{Claim}
	\newtheorem*{conjecture}{Conjecture~}
	\newtheorem*{fact}{Fact~}
	\newtheorem*{thma}{Theorem A}
	\newtheorem*{corb}{Corollary B}
	\newtheorem*{thmc}{Theorem C}
	\renewcommand{\proofname}{\bf Proof}
	\newcommand{\email}{Email: }

\title{Some Applications of Group Theoretic Rips Constructions to the Classification of von Neumann Algebras}
\author{Ionut Chifan\thanks{I.C. has been supported in part by the NSF grants DMS-1600688 and FRG DMS-1854194}, Sayan Das and Krishnendu Khan\thanks{K.K. has been supported in part by the NSF grants FRG DMS-1853989 and FRG DMS-1854194 }}
\date{}
\maketitle
\begin{abstract}

In this paper we study various von Neumann algebraic rigidity aspects for the  property (T) groups that arise via the Rips construction developed by Belegradek and Osin in geometric group theory \cite{BO06}. Specifically, developing a new interplay between Popa's deformation/rigidity theory \cite{Po07} and geometric group theory methods we show that several algebraic features of these groups are completely recognizable from the von Neumann algebraic structure. In particular, we obtain new infinite families of pairwise non-isomorphic property (T) group factors thereby providing positive evidence towards Connes' Rigidity Conjecture.

In addition, we use the Rips construction to build examples of property (T) II$_1$ factors which possess maximal von Neumann subalgebras without property (T) which  answers a question raised in an earlier version of \cite{JS19} by Y. Jiang and A. Skalski. 
\end{abstract}	
\thispagestyle{empty}

	\section{Introduction }
	The von Neumann algebra $\mathcal L(G)$ associated to a countable discrete group $G$ is called \emph{the group von Neumann algebra} and it is defined as the bicommutant of the left regular representation of $G$ computed inside the algebra of all bounded linear operators on the Hilbert space of the square summable functions on $G$. $\El(G)$ is a II$_1$ factor (has trivial center) precisely when all nontrivial conjugacy classes of $G$ are infinite (icc), this being the most interesting for study \cite{MvN43}. The classification of group factors is a central research theme revolving around the following fundamental question: \emph{What aspects of the group $G$ are remembered by $\El(G)$?} This is a difficult topic as algebraic group properties usually do not survive after passage to the von Neumann algebra regime. Perhaps the best illustration of this phenomenon is Connes' celebrated result asserting that all amenable icc groups give isomorphic factors, \cite{Co76}. Hence genuinely different groups such as the group of all finite permutations of the positive integers, the lamplighter group, or the wreath product of the integers with itself give rise to isomorphic factors. Ergo, basic algebraic group constructions such as direct products, semidirect products, extensions, inductive limits or classical algebraic invariants such as torsion, rank, or generators and relations in general cannot be recognized from the von Neumann algebraic structure. In this case the only information on $G$ retained by the von Neumann algebra is amenability.
	\vskip 0.05in
	When $G$ is non-amenable the situation is far more complex and an unprecedented progress has been achieved through the emergence of Popa's deformation/rigidity theory \cite{Po07,Va10icm,Io12,Io18icm}. Using this completely new conceptual framework it was shown that various algebraic/analytic properties of groups and their representations can be completely recovered from their von Neumann algebras, \cite{OP03,OP07,IPV10,BV12,CdSS15,DHI16,CI17,CU18}.  In this direction an impressive milestone was Ioana, Popa and Vaes's discovery of the first examples of groups $G$ that can be completely reconstructed from $\El(G)$, i.e.\ \emph{$W^*$-superrigid groups}\footnote[1]{If $H$ is any group such that $\El(G)\cong \El(H)$ then $H\cong G$.} \cite{IPV10}. Additional examples were found subsequently in \cite{BV12,B13,CI17}.  It is worth noting that the general strategies used in establishing these results share a common essential ingredient---the ability to first reconstruct from $\El(G)$ specific given algebraic features of $G$. For instance, in the examples covered in \cite{IPV10,BV12,B13}, the first step was to show that whenever $\El(G)\cong \El(H)$ then the mystery group $H$ admits a generalized wreath product decomposition exactly as $G$ does; also in the case of \cite[Theorem A]{CI17} again the main step was to show that $H$ admits an amalgamated free product splitting exactly as $G$. These aspects motivate a fairly broad and independent study on this topic---the quest of identifying a comprehensive list of algebraic features of groups which completely pass to the von Neumann algebraic structure. While a couple of works have already appeared in this direction \cite{CdSS15,CI17,CU18} we are still far away from  having a satisfactory overview of these properties and a great deal of work remains to be done. 
	
	\vskip 0.05in A striking conjecture of Connes predicts that all icc property (T) groups are $W^*$-superrigid. Despite the fact that this conjecture motivated to great effect a significant portion of the main developments in Popa's deformation/rigidity theory \cite{Po03,Po04,Io11,IPV10}, no example of a property (T)  $W^*$-superrigid group is currently known. The first hard evidence towards Connes' conjecture was found by Cowling and Haagerup in \cite{CH89}, where it was shown that uniform lattices in $Sp(n,1)$ give rise to non-isomorphic factors for different values of $n \geq 2$. Moreover, for any collection $\{G_k\}_{k}$ of uniform lattices in $Sp(n_k,1), n_k \geq 2$, the group algebras $\{\mathcal L(\times^n_{i=1}G_i)\}_n$ are pairwise non-isomorphic.
	 Later on, using a completely different approach,  Ozawa and Popa \cite{OP03} obtained a far-reaching generalization of this result by showing  that for any collection $\{G_n\}_{n}$ of hyperbolic, property (T) groups (e.g.\ uniform lattices in $Sp(n,1), n \geq 2$, \cite{CH89}) the group algebras $\{\mathcal L(\times^n_{i=1}G_i)\}_n$ are pairwise non-isomorphic. However, little is known beyond these two classes of examples. 
	Moreover, the current literature offers an extremely limited account on what algebraic features that occur in a property (T) group are completely recognizable at the von Neumann algebraic level. For instance, besides the preservation of the Cowling-Haagerup constant \cite{CH89}, the amenability of normalizers of infinite amenable subgroups in hyperbolic property (T) groups from \cite[Theorem 1]{OP07}, and the direct product rigidity for hyperbolic property (T) groups from \cite[Theorem A]{CdSS15}, \cite[Theorem A]{CU18} very little is known. Therefore in order to successfully construct property (T) $W^*$-superrigid groups via a strategy similar to the ones used in \cite{IPV10,BV12,B13,CI17} we believe it is imperative to identify new algebraic features of property (T) groups that survive the passage to the von Neumann algebraic regime. Any success in this direction will potentially hint at what group theoretic methods to pursue in order to address Connes' conjecture.   
	\vskip 0.05in
	In this paper we make new progress on this study by showing that many algebraic aspects of the Rips constructions developed in geometric group theory by Belegradek and Osin \cite{BO06} are entirely recoverable from the von Neumann algebraic structure. To properly introduce the result we briefly describe their construction. Using the prior Dehn filling results from \cite{Os06}, Belegradek and Osin showed in \cite[Theorem]{BO06} that for every finitely generated group $Q$ one can find a property (T) group $N$ such that $Q$ can be realized as a finite index subgroup of $Out(N)$. This canonically gives rise to an action $Q \curvearrowright^\sigma N$ by automorphisms such that the corresponding semidirect product group $N\rtimes_\sigma Q$ is hyperbolic relative to $\{Q\}$. Throughout the document the semidirect products $N\rtimes_\sigma Q$ will be termed Belegradek-Osin's group Rips constructions. When $Q$ is torsion free then one can pick $N$ to be torsion free as well and hence both $N$ and $N\rtimes_\sigma Q$ are icc groups. Also when $Q$ has property (T) then $N\rtimes_\sigma Q$ has property (T). Under all these assumptions we will denote by $\mathcal Rip _{ T} (Q)$ the class of these Rips construction groups $N\rtimes_\sigma Q$. 
	\vskip 0.05in
	
	The first main result of our paper concerns a fairly large class of canonical fiber products of groups in $\mathcal Rip_T(Q)$. Specifically, consider any two groups $N_1 \rtimes_{\sigma_1} Q, N_2\rtimes_{\sigma_2} Q \in \mathcal Rip_T (Q)$ and form the canonical fiber product $G=(N_1\times N_2) \rtimes_\sigma Q$ where $\sigma=(\sigma_1,\sigma_2)$ is the diagonal action. Notice that since property (T) is closed under extensions \cite[Section 1.7]{BdlHV00} it follows that $G$ has property (T). Developing new interplay between geometric group theoretic methods  \cite{Rip82, DGO11, Os06, BO06} and deformation/rigidity methods \cite{Io11,IPV10,CdSS15,CdSS17,CI17,CU18}, for a fairly large family of groups $Q$, we show that the semidirect product feature of $G$ is an algebraic property completely recoverable from the von Neumann algebraic regime. In addition, we also have a complete reconstruction of the acting group $Q$. The precise statement is the following

	
	\begin{thma}[Theorem \ref{semidirectprodreconstruction}]\label{thma}
		Let $Q=Q_1\times Q_2$, where $Q_i$ are icc, biexact, weakly amenable, property (T),  torsion free, residually finite groups. For $i=1,2$ let $N_i\rtimes_{\sigma_i} Q\in \mathcal Rip_T(Q)$ and denote by $\G=(N_1\times N_2)\rtimes_{\sigma} Q$ the semidirect product associated with the diagonal action $\sigma=\sigma_1 \times \sigma_2 :Q\ca N_1\times N_2 $. Denote by $\emm=\El(\G)$ be the corresponding II$_1$ factor. Assume that $\Lambda$ is any arbitrary group and $\Theta: \El(\Gamma)\rar \El(\Lambda)$ is any $\ast$-isomorphism. Then there exist group actions by automorphisms $H\ca^{\tau_i} K_i$  such that $\La= (K_1\times K_2) \rtimes_{\tau} H$ where $\tau=\tau_1\times \tau_2: H\ca K_1\times K_2$ is the diagonal action. Moreover one can find a multiplicative character $\eta:Q\rar \mathbb T$, a group isomorphism $\delta: Q\rar H$ and unitary $w\in \El(\Lambda)$ and $\ast$-isomorphisms $\Theta_i: \El(N_i)\rar \El(K_i)$ such that for all $x_i \in L(N_i)$ and $g\in Q$ we have 
		\begin{equation}
		\Theta((x_1\otimes x_2) u_g)= \eta(g) w ((\Theta_1(x_1)\otimes \Theta(x_2)) v_{\delta(g)})w^*.
		\end{equation}      
		Here $\{u_g \,:\,g\in Q\}$ and $\{v_h,\,:\, h\in H\}$ are the canonical unitaries implementing the actions of $Q \ca \El(N_1)\bar\otimes \El(N_2)$ and $H\ca \El(K_1)\bar\otimes \El(K_2)$, respectively.
	\end{thma}
There are countably infinitely many groups that are residually finite, torsion free, hyperbolic, and have property (T). A concrete such family is $\{\Lambda_k \,|\, k\geq 2\}$ where $\Lambda_k<Sp(k,1)$ is a uniform lattice. It is well known the $\Lambda_k$'s are residually finite \cite{Mal40}, (virtually) torsion free \cite{Sel60}, hyperbolic \cite[Example B]{Gr87}, have property (T) (see for instance,\cite[Theorem 1.5.3]{BdlHV00})  and are pairwise non-isomorphic \cite{CH89}. However there are infinitely many pairwise non-isomorphic such lattices even in the same Lie group. To see this fix $k\geq 2$ together with a torsion free, uniform lattice $\Gamma<Sp(k,1)$.  Since $\Gamma $ is residually finite there is a sequence of normal, finite index, proper subgroups $\cdots \lhd \Gamma_{n+1} \lhd \Gamma _n \lhd \cdots \lhd \Gamma_1\lhd \Gamma$ such that $\cap_n \Gamma_n=1$. Being subgroups, $\Gamma_n$ are clearly residually finite and torsion free. Moreover, the finite index condition implies that all $\Gamma_n$'s are  hyperbolic and have property (T). As the $\Gamma_n$'s are co-hopfian \cite{Pr76} and for every $n< m$ we have $\Gamma_n< \Gamma_m$ then $\Gamma_n\ncong \Gamma_m$. Therefore the class $\{\Gamma_n\,|\, n\in \mathbb N \}$ satisfies our conditions. 
Finally we note that, since every hyperbolic group is finitely presented and there are only countably many such groups, one cannot built examples of larger cardinality than the ones presented above.    
	
In conclusion, Theorem A provides explicit examples of infinitely many pairwise non-isomorphic group $\rm II_1$ factors with property (T). Moreover these groups are quite from the previously classes \cite{CH89,OP03} as they give rise to factors that are non-solid ($\El(\G)$ contains two commuting nonamenable subfactors $\El(N_1)$ and $\El(N_2)$), are tensor indecomposable (\cite[Lemma 2.3]{D20}) and do not admit Cartan subalgebras (Corollary~\ref{cor:nocartan}). Moreover, using Margulis normal subgroup theorem the factors covered by Theorem A are non-isomorphic to any factor arising from any irreducible lattices in a higher-rank semisimple Lie group (see remarks after the proof of Theorem ~\ref{semidirectprodreconstruction}). We also mention that Theorem A or its strong rigidity version Theorem \ref{nonisomsemidirect} (see also Corollary \ref{inffindext}) provides examples of infinite families of finite index subgroups $\Gamma_n \leqslant \G$ in a given icc property (T) group $\Gamma$ such that the corresponding group factors $\El(\G_n)$ and $\El(\G_m)$ are nonisomorphic for $n \neq m$. As $\Gamma_n$'s are measure equivalent this provides new counterexamples to D. Shlyakhtenko's question, asking whether measure equivalence of icc groups implies isomorphism of the corresponding group factors (see \cite[Page 18]{Po09}), which are very different in nature from the ones obtained in \cite{CI, CdSS15}. We summarize this discussion in the next corollary.
\color{black}
	\begin{corb} [Corollary~\ref{inffindext} below] Assume the same notations as in Theorem A.
		\begin{enumerate}
			\item[1)] Let $Q_1,Q_2$ be uniform lattices in $Sp(n,1)$ with $n\geq 2$ and let $Q:=Q_1 \times Q_2$. Also let $\cdots \leqslant Q_1^s\leqslant \cdots \leqslant Q_1^2\leqslant Q^1_1\leqslant Q_1$ be an infinite family of finite index subgroups and denote by $Q_s:= Q_1^s\times Q_2\leqslant Q$. Then consider $N_1\rtimes_{\sigma_1} Q,N_2\rtimes_{\sigma_2} Q\in {\mathcal Rip}_T(Q)$ and let $\G= (N_1\times N_2)\rtimes _{\sigma_1\times \sigma_2}  Q$. Inside $\G$ consider the finite index subgroups $\G_s:=(N_1\times N_2)\rtimes _{\sigma_1\times \sigma_2}  Q_s$. Then the family $\{ \mathcal L(\G_s) \,|\, s\in I \}$ consists of pairwise non-isomorphic finite index subfactors of $\mathcal L(\G)$. 
			\item[2)] Let $\G,\G_n$ be as above. Then $\G_n$ is measure equivalent to $\G$ for all $n \in \mathbb N$, but $\El(\G_n)$ is not isomorphic to $\El(\G_m)$ for $n \neq m$. 
		\end{enumerate}
		
	\end{corb}     
	From a different perspective our theorem can be also seen as a von Neumann algebraic superrigidity result regarding conjugacy of actions on noncommutative von Neumann algebras. Notice that very little is known in this direction as most of the known superrigidity results concern algebras arising from actions of groups on probability spaces.
	
	\vskip 0.05in
	
	In certain ways one can view Theorem A as a first step towards providing an example of a property (T) superrigid group. While the acting group $Q$ can be completely recovered, as well as certain aspects of the action $Q\ca N_1\times N_2$ (e.g.\ trivial stabilizers) only the product feature of the "core" $\El (N_1 \times N_2)$ can be reconstructed at this point. While the reconstruction of $N_1$ and $N_2$ seems to be out of reach momentarily, we believe that a deeper understanding of the Rips construction, along with new von Neumann algebraic techniques are necessary to tackle this problem. We also remark that in a subsequent article it was shown that the group factors as in Theorem A have trivial fundamental group (see \cite[Theorem B]{CDHK20}).
	
	\vskip 0.05in
	Besides the aforementioned rigidity results we also investigate applications of group Rips constructions to the study of maximal von Neumann algebras. If $\mathcal M$ is a von Neumann algebra then a von Neumann subalgebra $\mathcal N\subset \mathcal M$ is called \emph{maximal} if there is no intermediate von Neumann subalgebra $\mathcal P$ so that $\mathcal N\subsetneq \mathcal P \subsetneq \mathcal M$. Understanding the structure of maximal subalgebras of a given von Neumann algebra is a rather difficult problem that is intimately related with the very classification of these objects. Despite a series of remarkable earlier successes on the study of maximal \emph{ amenable } subalgebras initiated by Popa \cite{Po83} and continued more recently \cite{Sh06,CFRW08,Ho14,BC14,BC15,Su18,CD19,JS19}, significantly less is known for the arbitrary maximal ones. For instance Ge's question \cite[Section 3, Question 2]{Ge03} on the existence of non-amenable factors that posses maximal factors that are hyperfinite was settled in the affirmative only very recently by Y. Jiang and A. Skalski in \cite{JS19}. In fact in their work Jiang-Skalski proposed a more systematic approach towards the study of maximal von Neumann subalgebras within various categories such as the von Neumann algebras with Haagerup's property or with property (T) of Kazhdan. Their investigation also naturally led to several interesting open problems, \cite[Section 5]{JS19}.   
	\vskip 0.05in
	
	In this paper we explain how in a setting similar with \cite{JS19} the Belegradek-Osin's group Rips construction techniques and Ol'shanski's type monster groups can be used in conjunction with Galois correspondence results for II$_1$ factors $\grave{a}$ $la$ Choda \cite{Ch78} to produce many maximal von Neumann subalgebras arising from group/subgroup situation. In particular, through this mix of results we are able to construct many examples of II${_1}$ factors with property (T) that have maximal von Neumann subalgebras without property (T), thereby answering Problem 5.5 in the first version of the paper \cite{JS19} (see Theorem \ref{maximalvN}).  More specifically, using the Ol'shanskii's small cancellation techniques in the setting of lacunary hyperbolic groups \cite{OOS07} we explain how one can construct a property (T) monster group $Q$ whose maximal subgroups are all isomorphic to a given rank one group $Q_m$\footnote{any group that is isomorphic to a subgroup of $(\mathbb Q,+)$ is called rank one} (see Section \ref{smallcancellation}). Then if one considers the Belegradek-Osin Rips construction $N\rtimes Q$ corresponding to $Q$ then using a Galois correspondence (Theorem \ref{trivrel}) one can show the following

	\begin{thmc} (Theorem \ref{maximalvN}) For every maximal rank one subgroup $Q_m< Q$ consider the subgroup $N\rtimes Q_m <N\rtimes Q$. Then $\El(N\rtimes Q_m)\subset \El(N\rtimes Q)$ is a maximal von Neumann subalgebra.
	\end{thmc}
	Note that since $N$ and $Q$ have property (T) then so does $N\rtimes Q$ and therefore the corresponding II$_1$ factor $\mathcal L(N\rtimes Q)$ has property (T) by \cite{CJ}. However since $N\rtimes Q_m$ surjects onto the infinite abelian group $Q_m$ then it does not have property (T) and hence $\mathcal L(N\rtimes Q_m)$ does not have property (T) either.
	Another solution to the problem of finding maximal subalgebras without property (T) inside factors with property (T) was also obtained independently by Y. Jiang and A. Skalski in the most recent version of their paper. Their beautiful solution has a different flavor from ours; even though the Galois correspondence theorem $\grave a$ $la$ Choda is a common ingredient in both of the proofs. Hence we refer the reader to \cite[Theorem 4.8]{JS19b} for another solution to the aforementioned problem.
	\color{black}
	
	
	\section{Preliminaries}
	
	\subsection{Notations and Terminology}

	We denote by $\mathbb N$ and $\mathbb Z$ the set of natural numbers and the integers, respectively. For any $k\in \mathbb N$ we denote by $\overline{1,k}$ the integers $\{1,2,...,k\}$. 
\vskip 0.05in	
	All von Neumann algebras in this document will be denoted by calligraphic letters e.g.\ $\mathcal A$, $\mathcal B$, $\mathcal M$, $\mathcal N$, etc. Given a von Neumann algebra $\mathcal M$ we will denote by $\mathscr U(\mathcal M)$ its unitary group, by $\mathscr P(\mathcal M)$ the set of all its nonzero projections, and by $\mathscr Z(\mathcal M)$ its center. We also denote by $(\mathcal M)_1$ its unit ball. All algebras inclusions $\mathcal N\subseteq \mathcal M$ are assumed unital unless otherwise specified. Given an inclusion $\mathcal N\subseteq \mathcal M$ of von Neumann algebras we denote by $\mathcal N'\cap \mathcal M$ the relative commutant of $\mathcal N$ in $\mathcal M$, i.e.\ the subalgebra of all $x\in \mathcal M$ such that $xy=yx$ for all $y\in \mathcal N$. We also consider the one-sided quasinormalizer $\mathscr {QN}^{(1)}_{\mathcal M}(\mathcal N)$ (the semigroup of all $x\in\mathcal M$ for which there exist $x_1,x_2,...,x_n \in \mathcal M$ such that $\mathcal N x\subseteq \sum_i x_i \mathcal N$) and the quasinormalizer $\mathscr {QN}_{\mathcal M}(\mathcal N)$ (the set of all $x\in\mathcal M$ for which there exist $x_1,x_2,...,x_n \in \mathcal M$ such that $\mathcal N x\subseteq \sum_i x_i \mathcal N$ and $x\mathcal N \subseteq \sum_i  \mathcal N x_i$) and we notice that $\mathcal N\subseteq \mathscr{N}_{\mathcal M}(\mathcal N)\subseteq \mathscr {QN}^{}_{\mathcal M}(\mathcal N)\subseteq \mathscr {QN}^{(1)}_{\mathcal M}(\mathcal N)$.
\vskip 0.05in
All von Neumann algebras $\emm$ considered in this article will be tracial, i.e.\ endowed with a unital, faithful, normal linear functional $\tau:M\rightarrow \mathbb C$  satisfying $\tau(xy)=\tau(yx)$ for all $x,y\in \emm$. This induces a norm on $\emm$ by the formula $\|x\|_2=\tau(x^*x)^{1/2}$ for all $x\in \emm$. The $\|\cdot\|_2$-completion of $\emm$ will be denoted by $L^2(\emm)$.  For any von Neumann subalgebra $\mathcal N\subseteq \mathcal M$ we denote by $E_{\mathcal N}:\mathcal M\rightarrow \mathcal N$ the $\tau$-preserving conditional expectation onto $\mathcal N$. 
\vskip 0.05in
For a countable group $G$ we denote by $\{ u_g | g\in G \} \in \mathscr U(\ell^2G)$ its left regular representation given by $u_g(\delta_h ) = \delta_{gh}$, where $\delta_h:G\rightarrow \mathbb C$ is the Dirac function at $\{h\}$. The weak operatorial closure of the linear span of $\{ u_g | g\in G \}$ in $\mathscr B(\ell^2 G)$ is the so called group von Neumann algebra and will be denoted by $\El(G)$. $\El(G)$ is a II$_1$ factor precisely when $G$ has infinite non-trivial conjugacy classes (icc). If $\mathcal M$ is a tracial von Neumann algebra and $G \ca^\sigma \mathcal M$ is a trace preserving action we denote by $\mathcal M \rtimes_\sigma G$ the corresponding cross product von Neumann algebra \cite{MvN37}. For any subset $K\subseteq G$ we denote by $P_{\mathcal M K}$  the orthogonal projection from the Hilbert space $L^2(\mathcal M \rtimes G)$ onto the closed linear span of $\{x u_g \,|\, x\in \mathcal M, g\in K\}$. When $\mathcal M$ is trivial we will denote this simply by $P_K$.  
\vskip 0.05in
Given a subgroup  $H \leqslant G$ we denote by $C_G(H)$ the centralizer of $H$ in $G$ and by $N_G(H)$ the normalizer of $H$ in $G$. Also we will denote by $QN^{(1)}_G(H)$ the  one-sided quasinormalizer of $H$ in $G$; this is the semigroup of all $g\in G$ for which there exist a finite set $F \subseteq G$ such that $Hg\subseteq F H$. Similarly we denote by $QN_G(H)$ the  quasinormalizer (or commensurator) of $H$ in $G$, i.e.\ the subgroup of all $g\in G$ for which there is a finite set $F \subseteq G$ such that $Hg\subseteq FH$ and $gH\subseteq HF$. We canonically have $H C_G(H)\leqslant N_G(H)\leqslant QN_G(H)\subseteq QN^{(1)}_G(H)$. We often consider the virtual centralizer of $H$ in $G$, i.e. $vC_G(H)=\{g\in G \,|\, |g^{H}|<\infty\} $. Notice $vC_G(H)$ is a subgroup of $G$ that is normalized by $H$. When $H=G$, the virtual centralizer is nothing else but the FC-radical of $G$. Also one can easily see from definitions that $HvC_G(H)\leqslant QN_G(H)$. For a subgroup $H\leqslant G$ we denote by $\llangle H\rrangle$ the normal closure of $H$ in $G$.

Finally, for any groups $G$ and $N$ and an action $G\curvearrowright^{\sg} N$ we denote by $N\rtimes_\sigma G$ the corresponding semidirect product group.

	\subsection{Popa's Intertwining Techniques} Over more than fifteen years ago, Sorin Popa has introduced  in \cite [Theorem 2.1 and Corollary 2.3]{Po03} a powerful analytic criterion for identifying intertwiners between arbitrary subalgebras of tracial von Neumann algebras. Now this is known in the literature  as \emph{Popa's intertwining-by-bimodules technique} and has played a key role in the classification of von Neumann algebras program via Popa's deformation/rigidity theory.

\begin {theorem}\cite{Po03} \label{corner} Let $(\mathcal M,\tau)$ be a separable tracial von Neumann algebra and let $\mathcal P, \mathcal Q\subseteq \mathcal M$ be (not necessarily unital) von Neumann subalgebras. 
Then the following are equivalent:
\begin{enumerate}
\item There exist $ p\in  \mathscr P(\mathcal P), q\in  \mathscr P(\mathcal Q)$, a $\ast$-homomorphism $\theta:p \mathcal P p\rightarrow q\mathcal Q q$  and a partial isometry $0\neq v\in q \mathcal M p$ such that $\theta(x)v=vx$, for all $x\in p \mathcal P p$.
\item For any group $\mathcal G\subset \mathscr U(\mathcal P)$ such that $\mathcal G''= \mathcal P$ there is no sequence $(u_n)_n\subset \mathcal G$ satisfying $\|E_{ \mathcal Q}(xu_ny)\|_2\rightarrow 0$, for all $x,y\in \mathcal  M$.
\item There exist finitely many $x_i, y_i \in \mathcal M$ and $C>0$ such that  $\sum_i\|E_{ \mathcal Q}(x_i u y_i)\|^2_2\geq C$ for all $u\in \mathcal U(\mathcal P)$.
\end{enumerate}
\end{theorem} 
\vskip 0.02in
\noindent If one of the three equivalent conditions from Theorem \ref{corner} holds then we say that \emph{ a corner of $\mathcal P$ embeds into $\mathcal Q$ inside $\mathcal M$}, and write $\mathcal P\prec_{\mathcal M}\mathcal Q$. If we moreover have that $\mathcal P p'\prec_{\mathcal M}\mathcal Q$, for any projection  $0\neq p'\in \mathcal P'\cap 1_{\mathcal P} \mathcal M 1_{\mathcal P}$ (equivalently, for any projection $0\neq p'\in\mathscr Z(\mathcal P'\cap 1_{\mathcal P}  \mathcal M 1_{P})$), then we write $\mathcal P\prec_{\mathcal M}^{s}\mathcal Q$.
\vskip 0.02in
	
For further use we record the following result which controls the intertwiners in algebras arising form malnormal subgroups. Its proof is essentially contained in \cite[Theorem 3.1]{Po03} so it will be left to the reader. 	  
\begin{lemma}[Popa \cite{Po03}]\label{malnormalcontrol}Assume that $H\leqslant G$ is an almost malnormal subgroup and let $G \ca \mathcal N$ be a trace preserving action on a tracial von Neumann algebra $\mathcal N $. Let $\mathcal P \subseteq \mathcal N \rtimes H$ be a von Neumann algebra such that $\mathcal P\nprec_{\mathcal N\rtimes H}  \mathcal N$. Then for all elements $x,x_1,x_2,...,x_l \in \mathcal N\rtimes G$ satisfying $\mathcal P x\subseteq \sum^l_{i=1} x_i \mathcal P$ we must have that $x\in \mathcal N\rtimes H$. 
\end{lemma}

We continue with the following intertwining result for group algebras which is a generalization of some previous results obtained under normality assumptions \cite{DHI16}. For reader's convenience we also include a brief proof.
\begin{lemma}\label{intertwiningintersection} Assume that $H_1,H_2 \leqslant G$ are groups, let $G \ca \mathcal N$ be a trace preserving action on a tracial von Neumann algebra $\mathcal N$ and denote by $\mathcal M =\mathcal N \rtimes G$ the corresponding crossed product. Also assume that $\mathcal A \prec^s \mathcal N \rtimes H_1$ is a von Neumann algebra such that $\mathcal A\prec_{\mathcal M} \mathcal N \rtimes H_2$. Then one can find $h\in G$ such that $\mathcal A\prec_{\mathcal M} \mathcal N \rtimes (H_1 \cap h H_2 h^{-1})$. 
\end{lemma}
\begin{proof} Since $\mathcal A \prec^s \mathcal N \rtimes H_1$ then by \cite[Lemma 2.6]{Va10} for every $\varepsilon>0$ there exists a finite subset $S\subset G$ such that $\|P_{S H_1 S}(x)-x\|_2\leqslant \varepsilon$ for all $x\in (\mathcal A)_1$. Here for every $K\subseteq G$ we denote by $P_K$ the orthogonal projection from $L^2(\mathcal M)$ onto the closure of the linear span of $\mathcal N u_g$ with $g\in K$. Also since $\mathcal A\prec_{\mathcal M} \mathcal N \rtimes H_2$ then by Popa's intertwining techniques there exist a scalar $0<\delta<1$ and a finite subset $T\subset G$ so that $\|P_{T H_2 T}(x)\|_2\geqslant \delta$, for all $x\in (\mathcal A)_1$. Thus, using this in combination with the previous inequality, for every $x\in \mathscr U(\mathcal A)$ and every $\varepsilon>0$, there are finite subsets $S, T \subset G$ so that $\|P_{T H_2 T}\circ P_{S H_1 S} (x)\|_2\geqslant \delta-\varepsilon$. Since there exist finite subsets $R, U\subset G$ such that $TH_2 T\cap SH_1 S\subseteq U (\cup_{r\in R} H_2\cap r H_1 r^{-1})) U$ we further get that $\| P_{U (\cup_{r\in R} H_2\cap r H_1 r^{-1})) U } (x)\|_2\geqslant \delta-\varepsilon$. Then choosing $\varepsilon>0$ sufficiently small and using Popa's intertwining techniques together with a diagonalization argument (see proof of \cite[Theorem 4.3]{IPP05}) one can find $r \in R$ so that $\mathcal A \prec \mathcal N \rtimes (H_2\cap rH_1r^{-1})$, as desired.\end{proof}

In the sequel we need the following three intertwining lemmas, which establish that under certain conditions, intertwining in a larger algebra implies that the intertwining happens in a "smaller subalgebra".

\begin{lemma}\label{intertwininglower1} Let $\mathcal A, \mathcal B\subseteq \mathcal N\subseteq \mathcal M$ be von Neumann algebras so that $\mathscr N_{\mathcal M}(\mathcal A)''=\mathcal M$. If $\mathcal B \prec_{\mathcal M} \mathcal A$ then $\mathcal B \prec_{\mathcal N} \mathcal A$.   
\end{lemma}

\begin{proof} Since $\mathcal B \prec_{\mathcal M} \mathcal A$ then by Theorem \ref{corner} one can find $x_1,x_2...x_n ,y_1,y_2,...,y_n\in \mathcal M$ and $c>0$ such that $\sum_{i=1}^n\|E_{\mathcal A}(x_i by_i)\|_2^2\geq c, \text{ for all } b\in \mathcal U(\mathcal B)$. Since $\mathscr N_{\mathcal M}(\mathcal A)''=\mathcal M$ then using basic $\|\cdot \|_2$-approximation for $x_i$ and $y_i$ and shrinking $c>0$ if necessary one can find $g_1,g_2,... ,g_l, h_1,h_2,..., h_l\in \mathcal N_{\mathcal M}(\mathcal A)$ and $c'>0$ such that for all $b\in \mathcal U(\mathcal B)$ we have \begin{equation}\label{intineq}
	\sum_{i=1}^n\|E_{\mathcal A}(g_i bh_i)\|_2^2\geq c'>0.
	\end{equation} 
	Using normalization we see that $E_{\mathcal A}(g_i b h_i)=E_{g_i\mathcal Ag_i^*}(g_i bh_i)= g_iE_{\mathcal A}( bh_ig_i)g_i^*$. This combined with \eqref{intineq} and $\mathcal A \subseteq \mathcal N$ give $0<c'\leq \sum_{i=1}^l\|E_{\mathcal A}( bh_i g_i)\|_2^2= \sum_{i=1}^l\|E_{\mathcal A}\circ E_{\mathcal N}( b h_ig_i)\|_2^2=\sum_{i=1}^l\|E_{\mathcal A}( b E_{\mathcal N}(h_ig_i))\|_2^2$ for all $b\in \mathcal U(\mathcal B)$. Since $E_{\mathcal N}(h_ig_i)\in \mathcal N$ then using Theorem \ref{corner} this clearly shows that $\mathcal B \prec_{\mathcal N} \mathcal A$.  \end{proof}

\begin{lemma} \label{intertwininglower2}Let $Q$ be a group and denote by ${\rm d}(Q)=\{(q,q)\,|\, q\in Q\}$. Let $\mathcal A$ be a tracial von Neumann algebra and assume   $(Q \times Q) \ca^{\sg} \mathcal A$ is a trace preserving action. Let $\mathcal B \subseteq \mathcal A$ be a regular von Neumann subalgebra which is invariant under the action $\sg$. Let $\mathcal D \subseteq \mathcal A \rtimes_{\sg} {\rm d}(Q)$ be a subalgebra such that $\mathcal D \prec_{\mathcal A \rtimes_{\sg} (Q \times Q)} \mathcal B \rtimes_{\sg} {\rm d}(Q)$. Then  $\mathcal D \prec_{\mathcal A \rtimes_{\sg} {\rm d}(Q)} \mathcal B \rtimes_{\sg} {\rm d}(Q)$.	
	
\end{lemma}
\begin{proof}
Denote by $\M:= \mathcal A \rtimes_\sg (Q\times Q)$, $\mathcal N:=\mathcal A \rtimes_{\sg} {\rm d}(Q)$, and $\P:=\mathcal B \rtimes_{\sg} {\rm d}(Q)$. Thus $\P\subset \mathcal N\subset \mathcal M$ and with these notations we establish the following 

\begin{claim}\label{convdescend} Let  $(v_n)_n \subset \mathscr U(\mathcal N)$ be a sequence such that 
	$\lim_{n\rightarrow\infty}\|E_{\P}(av_n b)\|_2 = 0$  for all $a,b\in \mathcal N$. Then  \begin{equation}\label{conv1}
	\lim_{n\rightarrow\infty}\|E_{\P}(xv_n y)\|_2= 0\text{ for all }x,y\in \mathcal M. 
	\end{equation}
	
\end{claim}

\emph{Proof of Claim \ref{convdescend}}. Notice that $(Q \times Q) = (Q \times 1) \rtimes_{\rho} {\rm d}(Q)$, where  ${\rm d}(Q)\ca^\rho(Q \times 1)$ is the action by conjugation. Therefore using basic $\|\cdot\|_2$-approximations and the $\P$-bimodularity of the conditional expectation $E_\P$ it suffices to show \eqref{conv1} only for $x=(u_g\otimes 1) c$ and $y=d(u_h\otimes 1)$ for all $g,h\in Q$ and $c,d\in \A$. Under these assumptions we see that \begin{equation}\label{expectestimate}\begin{split}E_{\P}((u_g\otimes 1) c v_n d(u_h\otimes 1)) &=E_{\P}\circ P_{(u_g\otimes 1) \M (u_h\otimes 1) }((u_g\otimes 1) c v_n d(u_h\otimes 1))\\
&=P_{\B({\rm d}(Q)\cap (g,1){\rm d}(Q)(h,1))}((u_g\otimes 1) c v_n d(u_h\otimes 1)).\end{split}
\end{equation}
Here, and throughout the proof for every set $S\subseteq Q\times Q$ we denoted by $P_{\mathcal B S}$ the orthogonal projection onto the closed subspace $\overline{span}\{ \mathcal Bu_g\,|\, g \in S \}$.

To this end observe there exists an element $s\in Q$ such that  ${\rm d}(Q)\cap (g,1){\rm d}(Q)(h,1) \subseteq  [{\rm d}(Q)\cap (g,1){\rm d}(Q)(g^{-1},1)]{\rm d}(s)$. Moreover, a basic computation shows that ${\rm d}(Q)\cap (g,1){\rm d}(Q)(g^{-1},1)= {\rm d}(C_Q(g))$, where $C_Q(g)$ is the centralizer of $g$ in $Q$. Hence altogether we have that ${\rm d}(Q)\cap (g,1){\rm d}(Q)(h,1) \subseteq  {\rm d}(C_Q(g)) {\rm d}(s)$. Combining this with \eqref{expectestimate} and using the fact that $u_g\otimes 1$ normalizes $\B\rtimes {\rm d}(C_Q(g))$ we see that 

\begin{equation}\label{subordinationineq}
\begin{split}\|E_{\P}((u_g\otimes 1) c v_n d(u_h\otimes 1))\|_2&\leq \|P_{\B({\rm d}(C_Q(g)) {\rm d}(s))}((u_g\otimes 1) c v_n d(u_h\otimes 1))\|_2\\
&=\|E_{\B\rtimes {\rm d}(C_Q(g))}((u_g\otimes 1) c v_n d(u_{hs^{-1}}\otimes u_{s^{-1}}))\|_2\\
&=\|E_{\B\rtimes {\rm d}(C_Q(g))}(c v_n d(u_{hs^{-1}g}\otimes u_{s^{-1}}))\|_2\\
&=\|E_{\B\rtimes {\rm d}(C_Q(g))}(c v_n d E_{\mathcal N}(u_{hs^{-1}g}\otimes u_{s^{-1}}))\|_2\\
&= \delta_{hs^{-1}g, s^{-1}} \|E_{\B\rtimes {\rm d}(C_Q(g))}(c v_n d )\|_2\\
&\leq \|E_{\P}(c v_n d )\|_2.\end{split}
\end{equation}
Letting $n\rightarrow \infty$ in \eqref{subordinationineq} and using the hypothesis assumption, the claim obtains. $\hfill\blacksquare$

To show our lemma assume by contradiction that $\mathcal D \nprec_{\mathcal N} \mathcal P $. By Theorem \ref{corner} there is a sequence $(v_n)_n\subset \D\subset \mathcal N$ so that $\lim_{n\rightarrow\infty}\|E_{\P}(av_n b)\|_2 = 0$  for all $a,b\in \mathcal N$. Using Claim \ref{convdescend} we get $\lim_{n\rightarrow\infty}\|E_{\P}(xv_n y)\|_2= 0$ for all $x,y\in \mathcal M$ which by Theorem \ref{corner} again implies that $\mathcal D \nprec_{\mathcal M } \mathcal P$, a contradiction.
\end{proof}
\begin{lemma}\label{intlowertensor} Let $\mathcal C\subseteq \mathcal B$ and $\mathcal N\subseteq \mathcal M$ be inclusions of tracial von Neumann algebras. 
If $\mathcal A \subseteq \mathcal N \bten \mathcal B$  is a von Neumann subalgebra such that $\mathcal A\prec_{\mathcal M\bten \mathcal B} \mathcal M \bten \mathcal C$ then $\mathcal A\prec_{\mathcal N\bten \mathcal B} \mathcal N \bten \mathcal C$.
\end{lemma}
\begin{proof} By Theorem \ref{corner} one can find $x_i,y_i\in \mathcal M\bten \mathcal B$, $i=\overline{1,k}$ and a scalar $c>0$ such that \begin{equation}\label{egn1'}\sum\limits_{i=1}^n \| E_{\mathcal M \bten \mathcal C}(x_i a y_i)\|^2 \geq c \text { for all } d \in \euu(\mathcal A).\end{equation} Using $\|\cdot\|_2$-approximations of $x_i$ and $y_i$ by finite linear combinations of elements in $\mathcal M\bten_{\rm alg} \mathcal B$ together with the $\mathcal M\otimes 1$-bimodularity of $E_{\mathcal M\bten \mathcal C}$, after increasing $k$ and shrinking $c>0$ if necessary, in \eqref{egn1'} we can assume wlog that $x_i,y_i\in 1\otimes \mathcal B$. However, since $\mathcal A\subseteq \mathcal N\bten \mathcal B$ then in this situation we have $E_{\mathcal M \bten \mathcal C}(x_i a y_i)= E_{\mathcal M \bten \mathcal C}\circ E_{\mathcal N \bten \mathcal B}(x_i a y_i)= E_{\mathcal N \bten \mathcal C}(x_i a y_i)$. Thus \eqref{egn1'} combined with Theorem \ref{corner} give $\mathcal A\prec_{\mathcal N\bten \mathcal B} \mathcal N \bten \mathcal C$, as desired.   \end{proof}

In the sequel we need the following (minimal) technical variation of \cite[Lemma 2.6]{CI17}. The proof is essentially the same with the one presented in \cite{CI17} and we leave the details to the reader. 

\begin{lemma}[Lemma 2.6 in \cite{CI17}]\label{maxcorner2} Let $\mathcal P,\mathcal Q\subseteq \mathcal M$ be inclusions of tracial von Neumann algebras. Assume that $\mathscr {QN}^{(1)}_{\mathcal M}(\mathcal P)=P$ and $\mathcal Q$ is a II$_1$ factor. Suppose there is a projection $z\in \mathscr Z(\mathcal P)$ such that $\mathcal P z\prec^s \mathcal Q$ and a projection $p\in \mathcal P z$ such that $p\mathcal Pp=p\mathcal Qp$. Then one can find a unitary $u\in \mathcal M$ such that $u\mathcal Pz u^*=r\mathcal Qr$ where $r=uzu^*\in \mathscr P(\mathcal Q)$. 
\end{lemma}

The next lemma is a mild generalization of \cite[Proposition 7.1]{IPV10}, using the same techniques (see also the proof of \cite[Lemma 2.3]{KV15}).

\begin{lemma} \label{comultsubgp}
	Let $\La$ be an icc group, and let $\emm=\El(\La)$. Consider the comultiplication map $\Delta: \emm \rightarrow \emm \tp \emm$ given by $\Delta(v_{\lam})= v_{\lam} \otimes v_{\lam}$ for all $\lam \in \La$.
	Let $\mathcal A, \mathcal B \subseteq \emm$ be a (unital) $\ast$-subalgebras such that $\Delta(\mathcal A) \subseteq \emm \tp \mathcal B$. Then there exists a subgroup $\Sg < \La$ such that $\mathcal A \subseteq \El(\Sg) \subseteq \mathcal B$. In particular, if $\mathcal A= \mathcal B$, then $\mathcal A= \El(\Sg)$.
\end{lemma}

\begin{proof} Let $\Sg=\{s \in \La: v_s \in \mathcal B\}$. Since $\mathcal B$ is a unital $\ast$-subalgebra, $\Sg$ is a subgroup, and clearly $\El(\Sg) \subseteq \mathcal B$. We argue that $\mathcal A \subseteq \El(\Sg)$.
	
Fix $a \in \mathcal A$, and let $a = \sum_{\lam}a_{\lam}v_{\lam}$ be its Fourier decomposition. Let $I= \{ s\in \La: a_s \neq 0\}$. Fix $s \in I$, and consider the normal linear functional $\omega$ on $\emm$ given by $\omega(x)= \bar{a_s} \tau(xv_s^{\ast})$. Note that $(\omega \otimes 1)(a)= |a_s|^2 \otimes v_s$
Since $\Delta(\mathcal A) \subseteq \emm \tp \mathcal B$, we have that $(\omega \otimes 1)\Delta(\mathcal A) \subseteq \mathbb C \tp \mathcal B$. Thus, $v_s \in \mathcal B \Rightarrow s \in \Sg$. Since this holds for all $s \in I$, we get that $a \in \El(\Sg)$, and hence we are done.
\end{proof}

We end this section with the following elementary result. We are grateful to the  referee for suggesting a (much) shorter proof than the one we originally had, which used \cite[Proposition 2.3]{CD18}. 

\begin{lemma}\label{equalcorners} Let $\mathcal M$ be a tracial von Neumann algebra and let $\mathcal N$ be a type II$_1$ factor, with $\mathcal N \subseteq \mathcal M$ a unital inclusion. If there is $p\in \mathscr P(\mathcal N)$ so that $p\mathcal Np=p\mathcal Mp$ then $\mathcal N=\mathcal M$.
	\end{lemma}
	
\begin{proof}
	Shrinking $p$ if necessary we can assume $\tau(p)=1/n$. Let $v_1,...,v_n \in \mathcal N$ be partial isometries such that $v_iv_i^* = p$, for all $i$, and $\sum^n_{i=1} v_i^*v_i = 1$. Fix  $x \in \M$. Since for every $1\leq i,j\leq n$ we have $v_ixv_j^* \in p\M p = p\mathcal N p$ we get $x = \sum_{i,j=1}^n v_i^*(v_ixv_j^*)v_j \in \mathcal N$, as desired. \end{proof}


\subsection{Small Cancellation Techniques}\label{smallcancellation}
In this section, we recollect some geometric group theoretic preliminaries that will be used throughout this paper. We refer the reader to the book \cite{Ol91} and the papers \cite{OL93, OOS07} for more details related to the small cancellation techniques. We also refer the reader to the book \cite{LS77} for details concerning van Kampen diagrams.
\subsubsection{van Kampen Diagrams} \label{secvankampen}
Given a word $W$ over the alphabet set $\mathcal{S}$, we denote its length by $\|W\|$. We also write $W\equiv V$ to express the letter-for-letter equality for words $W,V$. \par 
Let $G$ be a group generated by a set of alphabets $\mathcal{S}$. A van Kampen diagram $\triangle$ over a presentation 
\begin{align} \label{vankampen}
G=\langle S|\mathcal{R}\rangle
\end{align} 
is a finite, oriented, connected, planar 2-complex endowed with a labeling function $Lab:E(\triangle)\rightarrow \mathcal{S}^{\pm 1}$, where $E(\triangle)$ denotes the set of oriented edges of $\triangle$, such that $Lab(e^{-1})\equiv (Lab(e))^{-1}$. Given a cell $\Pi$ of $\triangle$, $\partial \Pi$ denotes its boundary. Similarly $\partial \triangle$ denotes the boundary of $\triangle$. The labels of $\partial \triangle$ and $\partial \Pi$ are defined up to cyclic permutations. We also stipulate that the label for any cell $\Pi$ of $\triangle$ is equal to (up to a cyclic permutation) $R^{\pm 1}$, where $R\in\ \mathcal{R}$.\par 
Using the van Kampen lemma (\cite[Chapter 5, Theorem 1.1]{LS77}), a word $W$ over the alphabet set $\mathcal{S}$ represents the identity element in the group given by the presentation \eqref{vankampen} if and only if there exists a connected,  simply-connected planar diagram $\triangle$ over \eqref{vankampen} satisfying $Lab(\partial\triangle)\equiv W$.
\subsubsection{Small Cancellation over Hyperbolic Groups}
Let $G=\langle X\rangle$ be a finitely generated group and $X$ be a finite generating set for G. Recall that the Cayley graph $\G(G,X)$ of a group G with respect to the set of generators X is an oriented labeled 1–complex with the vertex set $V(\G(G,X))=G$ and the edge set $E(\G(G,X))=G\times X^{\pm 1}$.  An edge $e=(g,a)$ goes from the vertex g to the vertex ga and has label a. Given a combinatorial path $p$ in the Cayley graph $\G(G,X)$, the length $|p|$ is the number of edges in p. The word length $|g|$ of an element $g\in G$ with respect to the generating set $X$ is defined to be the length of a shortest word in $X$ representing $g$ in the group $G$ ie, $|g|:=\min_{h=_Gg}\|h\|$ . The formula $d(f,g)=|g^{-1}f|$ defines a metric on the group $G$. The metric on the Cayley graph $\Gamma(G,X)$ is the natural extension of this metric. A word $W$ is called a $(\lam,c)$-quasi geodesic in $\Gamma(G,X)$ for some $\lam>0,c\geq 0$ if $\lam\|W\|-c\leq |W|\leq \lam\|W\|+c$. A word $W$ is called a geodesic if it is a $(1,0)$-quasi geodesic. A word $W$ in the alphabet $X^{\pm 1}$ is called $(\lambda,c)$-quasi geodesic (respectively geodesic) in $G$ if any path in the Cayley graph $\Gamma(G,X)$ labeled by $W$ is $(\lambda,c)$-quasi geodesic (respectively geodesic). 
Throughout this section, $\mathcal{R}$ denotes a symmetric set of words (i.e. it is closed under taking cyclic shifts and inverses of words; and all the words are cyclically reduced) from $X^*$
, the set of words on the alphabet $X$. A common initial sub-word of any two distinct words in $\mathcal{R}$ is called a piece. We say that $\mathcal R$ satisfies the $C'(\mu)$ condition if any piece contained (as a sub-word) in a word $R\in\mathcal{R}$ has length smaller than $\mu\|R\|$.  

\begin{definition}\cite[Section 4]{OL93} \label{eppiece}
	A subword $U$ of a word $R\in \mathcal{R}$ is called an $\epsilon$-piece of the word $R$, for $\epsilon\geq 0$, if there exists a word $R'\in \mathcal{R}$ satisfying the following conditions: 
	\begin{enumerate} 
	\item[(1)] $R\equiv UV$ and $R^{\prime}\equiv U^{\prime}V^{\prime}$ for some $U^{\prime},V^{\prime}\in \mathcal{R}$;
	\item[(2)] $U^{\prime}=_G YUZ$ for some $Y,Z\in X^*$ where $\| Y\|,\| Z\| \leq \epsilon$;
	\item[(3)] $YRY^{-1}{\neq}_G R^{\prime}$.
\end{enumerate}
	We say that the system $\mathcal R$ satisfies the $C(\lambda,c,\epsilon,\mu,\rho)$-condition for some $\lambda\geq 1,c\geq 0,\epsilon\geq0,\mu>0,\rho>0$ if:
	\begin{enumerate} 
	\item[(a)] $\| R\| \geq \rho$ for any $R\in \mathcal{R}$;
	\item[(b)] Any word $R\in \mathcal{R}$ is a $(\lambda,c)$-quasi geodesic;
   \item[(c)] For any $\epsilon$-piece  $U$ of any word $R\in \mathcal{R}$, the inequalities $\| U\|,\| U^{\prime } \| <\mu \| R\|$ hold.
   \end{enumerate}
\end{definition}
In practice, we will need some slight modifications of the above definition \cite[Section 4]{OL93}. 
\begin{definition} \label{epprpc}
	A subword $U$ of a word $R\in \mathcal{R}$ is called an $\epsilon'$-piece of the word $R$, for $\epsilon\geq 0$, if:
	\begin{enumerate} 
	\item[(1)]  $R\equiv UVU^{\prime} V^{\prime}$ for some $V,U^{\prime} ,V^{\prime}\in X^*$,
	\item[(2)] $U^{\prime}=_GYU^{\pm}Z$ for some words $Y,Z\in X^*$ where $\| Y\| ,\| Z\| \leq \epsilon$.
\end{enumerate}
	We say that the system $\mathcal R$ satisfies the $C'(\lambda,c,\epsilon,\mu,\rho)$-condition  for some $\lambda\geq 1,c\geq 0,\epsilon \geq0,\mu>0,\rho>0$ if :
	\begin{enumerate} 
	\item[(d)] $\mathcal R$ satisfies the $C(\lambda,c,\epsilon,\mu,\rho)$ condition, and 
	\item[(e)] Every $\epsilon '$-piece $U$ of $R$ satisfies $||U'|| < \mu ||R||$, where $U'$ is as above.
\end{enumerate}
\end{definition}

Let $G$ be a group defined by 
\begin{align}\label{hyppresent}
G=\langle X|\mathcal{O}\rangle,
\end{align}
where $\mathcal{O}$ is the set of \textit{all} relators (not just the defining relations) of $G$. Given a symmetrized set of words $\mathcal{R}$ in the alphabet set $X$, we consider the quotient group ,
\begin{align}\label{representation}
H=\langle G|\mathcal{R}\rangle =\langle G|\mathcal{O}\cup \mathcal{R}\rangle.
\end{align} 
A cell over a van Kampen diagram over \eqref{representation} is called an $\mathcal{R}$-cell (respectively, an $\mathcal{O}$-cell) if its boundary label is a word from $\mathcal{R}$ (respectively, $\mathcal{O}$). We always consider a van Kampen diagram over \eqref{representation} up to some elementary transformations. For example we do not distinguish diagrams if one can be obtained from other by joining two distinct $\mathcal{O}$-cells having a common edge or by inverse transformations (\cite[Section 5]{OL93}).

\section{Some Examples of Ol'shanskii's Monster Groups in the Context of Lacunary Hyperbolic Groups}
In this section, we collect some group theoretic results needed for our main theorems in Sections \ref{sec4} and \ref{superrigidRIPS}. Readers who are mainly interested in the results in Section \ref{superrigidRIPS} may skip ahead to Subsection \ref{Rip}. The results in Subsections \ref{subsec3.1} and \ref{subsec3.2} shall be required for our main results in Section \ref{sec4}. \\
In order to derive our main results on the study of maximal von Neumann algebras (i.e.\ Theorem \ref{maximalvN}) we need to construct a new monster-like group in the same spirit with Ol'shanskii's famous examples from \cite{OL80}. Specifically, generalizing the geometric methods from \cite{OL93} to the context of lacunary hyperbolic groups \cite{OOS07} and using techniques developed by the third author from \cite{K19} we construct a group $G$ such that every maximal subgroup of $G$ is isomorphic to a subgroup of $\mathbb Q$, the group of rational numbers. While in our approach we explain in detail how these results are used, the main emphasis will be on the new aspects of these techniques. Therefore we recommend the interested reader to consult beforehand the aforementioned results \cite{OL93,K19}.  
\subsection{Elementary subgroups} \label{subsec3.1}

In this section, using methods developed in \cite{OL93}, we construct a group $Q$ whose maximal (proper) subgroups are rank $1$ abelian groups, see Theorem ~\ref{monster}. More specifically, we study "special limits" of hyperbolic groups, called lacunary hyperbolic groups, as introduced in \cite{OOS07}. 
\begin{definition}
	Let $\alpha: G\rightarrow H$ be a group homomorphism and $G=\langle A\rangle ,H=\langle B \rangle$. The injectivity radius $r_A(\alpha)$ is the radius of largest ball centered at identity of $G$ in the Cayley graph of $G$ with respect to $A$ on which the restriction of $\alpha$ is injective.
\end{definition} 
\begin{definition} \cite[Theorem 1.2]{OOS07}
	A finitely generated group $G$ is called lacunary hyperbolic group if $G$ is the direct limit of a sequence of  hyperbolic groups and epimorphisms;
	\begin{align}\label{lacunarylim} 
	G_1\overset{\eta_1}{\longrightarrow} G_{2}\overset{\eta_{2}}{\longrightarrow}\cdots G_i\overset{\eta_i}{\longrightarrow} G_{i+1}\overset{\eta_{i+1}}{\longrightarrow}G_{i+2}\overset{\eta_{i+2}}{\longrightarrow}\cdots 
	\end{align}
	where $G_i$ is generated by a finite set $ S_i $ and $\eta_i(S_i)=S_{i+1}$. Also $G_i$'s are $\delta_i$-hyperbolic where $\delta_i$=$o(r_{S_i}(\eta_i))$ , where $r_{S_i}(\eta_i)$=injective radius of $\eta_i$ w.r.t. $S_i$ .
\end{definition}
Fix $\omega$ a nonprincipal ultrafilter.  An asymptotic cone Cone$^{\omega}(X,e,d)$ of a metric space $(X,dist)$ where $e=\{e_i\}_i$, $e_i\in X$ for all $i$ and $d=\{d_i\}_i$ is an unbounded sequence of nondecreasing positive real numbers, is the $\omega$-limit of the spaces $(X,\frac{dist}{d_i})$. The sequence $d=\{d_i\}$ is called a scaling sequence. Following \cite[Theorem~3.3]{OOS07}, $G$ being lacunary hyperbolic group, is equivalent to the existence of a scaling sequence $d=\{d_i\}$ such that the asymptotic cone Cone$^{\omega}(\G(G,X),e,d)$ associated with the Cayley graph $\G(G,X)$ for a finite generating set $X$ of $G$ with $e=\{identity\}$ is an $\mathbb R$-tree for any nonprincipal ultrafilter $\omega$. For more details on asymptotic cones and their connection with lacunary hyperbolic groups we refer the reader to \cite[Section~2.3, Section~3.1]{OOS07}.

Our construction relies heavily on the notion of \textit{elementary} subgroups. For the readers' convenience, we collect below some preliminaries regarding elementary subgroups.
\begin{definition} \label{elsubhyp}
	A group $E$ is called elementary if it is virtually cyclic.	Let $G$ be a hyperbolic group and $g\in G$ be an infinite order element. Then the elementary subgroup containing $g$ is defined as
	\begin{center}
		$E(g):=\{x\in G| \ x^{-1}g^nx=g^{\pm n} \ for\ some \ n=n(x)\in \mathbb{N} \}$.
	\end{center}
\end{definition}

For further use we need the following result describing in depth the structure of elementary subgroups.
\begin{lemma}\label{elstruc}
	
	$1)$ \cite{Ol91}
	If $E$ is a torsion free elementary group then $E$ is cyclic. \\		
	$2)$ \cite[Lemma 1.16]{OL93} \label{elstruc2}
	Let $E$ be an infinite elementary group. Then $E$ contains normal subgroups $T\lhd E^{+}\lhd E$ such that $[E:E^{+}]\leq 2$ , $T$ is finite and $E^{+}/T\simeq \mathbb{Z}$. If $E\neq E^+$ then $E/T\simeq D_{\infty}$(infinite dihedral group). For a hyperbolic group $G$, $E(g)$ is unique maximal elementary subgroup of $G$ containing the infinite order element $g\in G$.
	
\end{lemma}

In the context of lacunary hyperbolic groups we need to introduce the following definition which generalizes Definition \ref{elsubhyp}.

\begin{definition}\label{lelementary}
	Let $G$ be a lacunary hyperbolic group and let $g\in G$ be an infinite order element. We define $E^{\mathcal{L}}(g):=\{x\in G |xg^nx^{-1}=g^{\pm n}$, for some $n=n(x)\in \mathbb{N} \}$.
\end{definition}

For future reference we now recall the following structural result regrading torsion elements in a $\delta$-hyperbolic group.
\begin{theorem}\cite[2.2.B]{Gr87}\label{torsionbound}
	Let $g\in G$ be a torsion element in a $\delta$-hyperbolic group $G$. Then $g$ is conjugate to an element $h$ in $G$ such that $|h|_G\leq 4\delta +1$.
\end{theorem}

The following elementary lemma will be used in the proof of Theorem~\ref{elementarygroup}. For convenience we include a short proof. 
\begin{lemma}\label{torsionfreelacunary}
	If $G$ is a torsion free lacunary hyperbolic group, then one can choose $G_i$ to be torsion free such that $G=\underset{\rightarrow}{\lim}G_i$. 
\end{lemma} 
\begin{proof}
	Fix a presentation $G=\langle S\ |\mathcal R \rangle$. By \cite[Theorem~3.3]{OOS07}, one can choose $G_i:=\langle  S\ |\mathcal R_{c(i)} \rangle $, where $\{c(n)\}_n$ is a strictly increasing sequence such that $\mathcal R_{c(i)}$ consists of labels of all cycles in the ball of radius $d_i$ (corresponding to the scaling sequence $\{d_i\}_i$ of the lacunary hyperbolic group) around the identity in $\G(G,S)$. Let $r_i$ be the injectivity radius of the quotient map $\phi_i:G_i\rightarrow G_{i+1}$. The lacunary hyperbolic condition implies that $\lim_{i\rightarrow\infty} \frac{\delta_i}{r_i}=0$, where $\delta_i$ is the hyperbolic constant for the group $G_i$ for all $i$. Choose $i_0$ such that for all $j\geq i_0$ we have $r_j>9\delta_j$. We will show the $G_j$'s are torsion free for all $j\geq i_0$, which proves the lemma. 
\vskip 0.1in	
	\noindent Fix any $j\geq i_0$. Assume by contradiction that $g\in G_j\setminus \{1\}$ is a torsion element. By Theorem~\ref{torsionbound} there is an element $h\in G_j\setminus \{1\}$ such that $h$ is conjugate to $g$ and $|h|_{G_j}\leq 4\delta_j+1$. Thus $h$ is a torsion element of $G_j$. Since  $|h|_{G_j}\leq 4\delta_j+1<r_i$ then $h$ is a non trivial element of $G_k$ for all $k\geq j$. Thus $h$ is a non trivial torsion element in the limit group $G$,  which is a contradiction!   
\end{proof}
The next result generalizes Lemma \ref{elstruc}, and provides a complete description of the structure of elementary subgroups of a torsion free lacunary hyperbolic group. This result can be deduced from the main theorem of \cite{K19}. For readers' convenience, we include a short proof.

\begin{theorem}\label{elementarygroup}
	Let $G$ be a torsion free lacunary hyperbolic group and let $g\in G$ be an infinite order element. Then $E^{\mathcal{L}}(g)$ is an abelian group of rank $1$ (i.e.\ $E^{\mathcal{L}}(g) \ embeds \ in \ (\mathbb{Q},+)$).
\end{theorem}	
\begin{proof}
	From the definition (\ref{lacunarylim}) of lacunary hyperbolic group it follows that $E^{\El}(g)=\underset{\rightarrow}{\lim}E_i(g)$ for every $e\neq g\in G$, where $E_i(g)$ is the elementary subgroup containing the element $g$ in the hyperbolic group $G_i$ when viewing $g\in G_i$. Since $G$ is torsion free one can choose $G_i$ to be torsion free by Lemma~\ref{torsionfreelacunary}. By Lemma~\ref{elstruc} part 1) we get that $E_i(g)$ is cyclic for all $i$. Observe that every surjective homomorphism between hyperbolic groups takes elementary subgroups into elementary subgroups, in particular $E_i(g)$ maps into $E_{i+1}(g)$. We now get the group $E^{\El}(g)=\underset{\rightarrow}{\lim}E_i(g)$ as an inductive limit of cyclic groups, which proves the theorem.
\end{proof}
\begin{remark}
	Let $G$ be a torsion free lacunary hyperbolic group and let $e\neq g\in G$. Note that $C_G(g)\leqslant E^{\El}(g)$, where $C_G(g)$ is the centralizer of $g$ in $G$.
\end{remark}

\subsection{Maximal Subgroups} \label{subsec3.2}
Let $G_0=\la X\ra$ be a torsion free $\delta$-hyperbolic group with respect to $X$, where $X=\{x_1,x_2,\ldots,x_n \}$ is a finite generating set. Without loss of generality we assume that $E(x_i)\cap E(x_j)=\{e\}$ for $i\neq j$. We define a linear order on $X$ by $x_i^{-1}<x_j^{-1}<x_i<x_j$, whenever $i<j$. Let $F'(X)$ denote the set of all non empty reduced words on $X$. Note that the order on $X$ induces the lexicographic order on $F'(X)$.
Let $F'(X)=\{w_1,w_2,\ldots \}$ be an enumeration with $w_i<w_j$ for $i<j$. Observe that $w_1=x_1$ and $w_2=x_2$. We now consider the set $\mathcal S:= F'(X)\times F'(X)\setminus \{ (w,w) |w\in F'(X) \}$ and enumerate the elements of $\mathcal S$ as
$\mathcal S=\{(u_1,v_1),(u_2,v_2),\ldots  \}$. \par 
Our next goal is to construct the following chain;
\begin{align}
G_0\overset{\beta_0}{\hookrightarrow} K_1\overset{\alpha_1}{\twoheadrightarrow}G'_1\overset{\gamma_1}{\twoheadrightarrow} G_1\overset{\beta_1}{\hookrightarrow} K_2 \overset{\alpha_2}{\twoheadrightarrow}G'_2\overset{\gamma_1}{\twoheadrightarrow} G_2  \cdots
\end{align}
where $K_i,G_i,G'_i$ are hyperbolic for all $i$ and $\eta_i:=\gamma_i\circ \alpha_{i}\circ \beta_{i-1}$ , $i\geq 1$, satisfies the conditions in \eqref{lacunarylim}.  \par 
Let $L$ be a rank $1$ abelian group. Then $L$ can be written as $L=\cup_{i=0}^{\infty}L_i$, where $L_i=\langle g_i\rangle_{\infty}$ and $g_i=g_{i+1}^{m_{i+1}}$ for some $m_{i+1}\in \mathbb{N}$. Here $\langle g_i\rangle_{\infty}$ denotes the infinite cyclic group generated by the infinite order element $g_i.$ \par 
Since $G_0$ is non-elementary, there exists a smallest index $j_i\geq i$ such that $v_{j_i}\notin E(u_{j_i})$. For $m\in \mathbb N$, define 
\begin{align}\label{amalgum} 
H^{k}_{i+1}:=H^{k-1}_{i+1}\underset{u_k=g_{(k,i+1)}^{m_{i+1}}}{*}\la g_{(k,i+1)} \ra_{\infty} \text{ where } H^0_{i+1}=G_i \text{ and } g_{(k,i+1)}=g_{i+1} \text{ for }k=1,2,\ldots ,j_i.
\end{align}
For $i \geq 0$ let $K_{i+1}$ be $H^{j_i}_{i+1}$. Note that $K_{i+1}$ is hyperbolic as $H^{k}_{i+1}$ is hyperbolic for all $k$ by \cite[Theorem 3]{MO98}. Choose $c_i,c_i'\in G_i$ such that $c_i,c_i'\notin E(u_k)$ for all $1 \leq k\leq j_i$ and $c_i,c_i'\notin E(v_{j_i}) $. One can find such $c_i$ and $c_i'$ since there are infinitely many elements in a non elementary hyperbolic group which are pairwise non commensurable, \cite[Lemma 3.8]{OL93}. Let
$Y_i:=\{g_{(k,i+1)}|1\leq k\leq j_i \}$.
Denote by 
\begin{align}\label{Word}
R_k:=g_{(k,i+1)}c_i^{n_{1,k}}c_i'c_i^{n_{2,k}}c'_i\cdots c_i^{n_{s_k,k}}c_i',
\end{align}
where $n_{s,k}$, for $ 1\leq k\leq j_i$ are defined as follows:
$$n_{1,k}=2^{k-1}n_{1,1}, \ s_k=n_{1,k-1} \ and \ n_{s,k}=n_{1,k}+(s-1).$$
We also denote by $\mathcal{R}_i$ be the set of all cyclic shifts of  $\{R_k^{\pm 1} : 1 \leq k \leq j_i\}$.

\begin{lemma}\cite[Lemma 5.1]{DAR17}\label{arman5.1}
	There exists a constant $K$ such that the set of words $\mathcal{R}$ defined above by \eqref{Word} are $(\lambda,c)$ quasi geodesic in $\Gamma(G,X)$, provided $n_{1,1}\geq K$, $c\notin E(g_{(k,i+1)})$, and $c'\notin E(g_{(k,i+1)}) $.
\end{lemma}
We now denote by $\tilde{\mathcal{R}}_{i+1}$ to be the set of words $\mathcal R_i$, defined as above, with $n_{1,k} \geq K$.

\begin{lemma}\cite[Lemma 5.2]{DAR17}\label{arman5.2}
	For any given constant $\epsilon_i\geq 0,\mu_i >0,\rho_i>0$, the system of words $\tilde{\mathcal{R}}_{i+1}$ (defined above) satisfies the $C'(\lambda_i,c_i,\epsilon_i,\mu_i,\rho_i)$ condition over $K_{i+1}$.
\end{lemma}
By construction there is a natural embedding $\beta_i:G_i\hookrightarrow K_{i+1}$. Let $G'_{i+1}:=\la K_{i+1}|\tilde{\mathcal{R}}_{i+1}\ra$ (where we are using the notations as in Subsection \ref{secvankampen} ). The factor group $G'_{i+1}$ is hyperbolic by \cite[Lemma 7.2]{OL93}. Now consider the natural quotient map $\alpha_{i+1}:K_{i+1} \twoheadrightarrow G'_{i+1}$.  Since $\alpha_{i+1}\circ \beta_{i}$ takes generators of $G_i$ to generators of $G_{i+1}'$, the map $\alpha_{i+1} \circ \beta_i$ is surjective.\par  
Consider the following set $$Z_i:=\{x\in X|x\notin E(u_{j_i})  \}.$$ 
Let $G_{i+1}:= G'_{i+1}/\llangle\mathcal R (Z_i,u_{j_i},v_{j_i},\lambda_i,c_i,\epsilon_i,\mu_i,\rho_i)\rrangle$ and let $\gamma_{i+1}:G'_{i+1}\twoheadrightarrow G_{i+1}$ be the quotient map. Here $\mathcal R (Z_i,u_{j_i},v_{j_i},\lambda_i,c_i,\epsilon_i,\mu_i,\rho_i)$ is the set of all conjugates and the cyclic shifts of some relations where we identify the elements of $Z_i$ with words of the form \eqref{Word} generated by $u_{j_i}$ and $v_{j_i}$. Since the relators $\mathcal R_i$ are generic we have added all the parameters to indicate these relations satisfy the small cancellation conditions with the parameters and their dependency to the specific set of words. One can choose the powers of $u_{j_i}$ and $v_{j_i}$ such that the small cancellation condition is satisfied by Lemma~\ref{arman5.1} and Lemma~\ref{arman5.2}. For more details on how to choose these words we refer the reader to \cite[Section~5]{OL93}, \cite[Section~5.4]{DAR17}. Thus it follows that the group $G_{i+1}$ is hyperbolic by \cite[Lemma 7.2]{OL93} as one can choose parameters $\lambda_i,c_i,\epsilon_i,\mu_i,\rho_i$ such that $\mathcal R (Z_i,u_{j_i},v_{j_i},\lambda_i,c_i,\epsilon_i,\mu_i,\rho_i)$ satisfies the $C'(\lambda_i,c_i,\epsilon_i,\mu_i,\rho_i)$ small cancellation condition \ref{epprpc} and the map $\gamma_{i+1}$ takes generating set to generating set. In particular, $\eta_{i+1}:=\gamma_{i+1}\circ \alpha_{i+1}\circ \beta_{i}$ is a surjective homomorphism which takes the generating set of $G_i$ to the generating set of $G_{i+1}$. Let $G^L:=\underset{\rightarrow}{\lim}G_i$. From its definition, it follows that  $G_{i+1}$ is the group generated by $u_{j_i}$ and $v_{j_i}$. \\

We summarize the above discussion in the following statement.
\begin{lemma}\label{limitmax}
	The above construction satisfies the following properties:
	\begin{enumerate}
		\item $G_i$ is non elementary hyperbolic group for all $i$;
		\item Either $u_i\in E(v_i)$ or the group generated by $\{u_i,v_i \}$ in $G_{i+1}$ is equal to all of $G_{i+1}$;
		\item For each element $x\in X$, $E(x)=\la y\ra$ in $G_i$, where $x=y^{m_1m_2\cdots m_i}$. The exponent $m_i$'s are described as follows: Being a rank 1 abelian group $L$ can be written as $L=\cup_{i=0}^{\infty}L_i$, where $L_i=\langle g_i\rangle_{\infty}$ and $g_i=g_{i+1}^{m_{i+1}}$ for some $m_{i+1}\in \mathbb{N}$;
	 \item $G^L:=\underset{\rightarrow}{\lim}G_i$ may be chosen to have property (T).
	\end{enumerate}
\end{lemma}
\begin{proof} Part 1. follows from \cite[Lemma 7.2]{OL93}. To see part 2. notice that by definition if $j_i>i$ then $v_i\in E(u_i)$ in $G_i$. Otherwise if $j_i=i$ then $v_i\notin E(u_i)$ in $G_i$ and $G_i$ is the group generated by $\{u_i,v_i\}$. Part 3. follows immediately from the fact that $x$ is not a proper power in $G_0$.  Finally, for part 4. notice that we may start the above construction with $G_0$ being a property (T) group. Then $G_1$ has property (T), as $G_0$ surjects onto $G_1$. By induction, each of the groups $G_i$ in the above construction have property (T). Hence $G^L$ has property (T).  \end{proof}

We are now ready to prove the main theorem of this section.
\begin{theorem}\label{monster}
	For any subgroup $Q_m$ of $(\mathbb Q,+)$ there exists a non elementary torsion free  lacunary hyperbolic group $G$ such that all maximal subgroups of $G$ are isomorphic to $Q_m$. Moreover, we may choose $G$ to have property (T).
\end{theorem}
\begin{proof}
	In the above construction let $L=Q_m$, $G=G^{Q_m}$ and take $d=m_1m_2\cdots m_i$ in \ref{amalgum}, where $L_i=\langle g_i\rangle_{\infty}$ and $g_i=g_{i+1}^{m_{i+1}}$ for some $m_{i+1}\in \mathbb{N}$ and $Q_m=\cup_{i=1}^{\infty}L_i$. One can choose sparse enough parameters to satisfy the injectivity radius condition in Definition \ref{lacunarylim} which in turn will ensure that $G$ is lacunary hyperbolic. The above construction also guarantees that $E^{\El}(g)=Q_m$ for all $g\in G\setminus \{1\}$. Suppose $P\nleqslant G$ is a maximal subgroup of $G$. As $P$ is a proper subgroup, $P$ is abelian by part $2.$ of Lemma \ref{limitmax}. Now let $e\neq h\in G$. Note that As P is abelian, P is contained in the centralizer of h. Now from Definition \ref{lelementary} it follows that $g\in P\leqslant E^{\El}(g)(\cong Q_m)\nleqslant G$. By maximality of $P$ we get that $P\cong Q_m$. Thus, all maximal subgroups of $G$ are isomorphic to $Q_m$ and hence any proper subgroup of $G$ is isomorphic to a subgroup of $Q_m$. \\
	The moreover part follows from part 4) of Theorem ~\ref{limitmax}.
\end{proof}
We end this section with the following well known counterexamples to von Neumann's conjecture.
\begin{corollary}[\cite{OL93},\cite{OL80}]
	For every non-cyclic torsion free hyperbolic group $\G$, there exists a non abelian torsion free quotient $\overline{\G}$ such that all proper subgroups of $\overline{\G}$ are infinite cyclic.
\end{corollary}
\begin{proof}Take $Q_m=\mathbb{Z}$ in Theorem \ref{monster}. \end{proof}


\subsection{Belegradek-Osin's Rips Construction in Group Theory} \label{Rip}

Rips constructions emerged in geometric group theory with the work of Rips from \cite{Rip82} and represent a rich source of examples for various pathological properties in group theory. This type of construction was used effectively to study automorphisms of property (T) groups. In this direction Ollivier-Wise \cite{OW} were able to construct property (T) groups whose automorphism group contain any given countable group.  This answered an important older question of P. de la Harpe and A. Valette about finiteness of outer automorphism groups of property (T) groups. Using the small cancellation methods developed in \cite{Os06} and \cite{AMO},  Belegradek and Osin discovered the following version of the Rips construction in the context of relatively hyperbolic groups:
\begin{theorem}\cite{BO06}\label{beleosinripsconstr}
	Let $H$ be a non-elementary hyperbolic group, $Q$ be a finitely generated group and $S$ a subgroup of $Q$. Suppose $Q$ is finitely presented with respect to $S$. Then there exists a short exact sequence
	$$1\rightarrow N\rightarrow G\overset{\epsilon}{\rightarrow} Q\rightarrow 1,$$
	and an embedding $\iota:Q\rightarrow G$ such that
	\begin{enumerate}
		\item $N$ is isomorphic to a quotient of $H$. 
		\item $G$ is hyperbolic relative to the proper subgroup $\iota(S)$.
		\item $\iota\circ \epsilon=Id$.
		\item If $H$ and $Q$ are torsion free then so is $G$.
		\item The canonical map $\phi :Q\hookrightarrow Out(N)$ is injective and $[Out(N):\phi(Q)]<\infty$.
	\end{enumerate}
\end{theorem}
This construction is extremely important for our work. We are particularly interested in the case when $H$ is torsion free and has property (T) and $Q=S$ and it is torsion free. In this situation Theorem \ref{beleosinripsconstr} implies that $G$ is admits a semidirect product decomposition  $G= N\rtimes Q$ and it is hyperbolic relative to $\{Q\}$. Notice that the finite conjugacy radical $FC(N)$ of $N$ is invariant under the action of $Q$ and hence $FC(N)$ is an amenable normal subgroup $G$. Since $G$ is relative hyperbolic it follows that $FC(N)$ is finite and hence it is trivial as $G$ is torsion free; in particular $N$ is an icc group. Since $G$ is hyperbolic relative to $Q$ it follows that the stabilizer of any $n\in N$ in $Q$ under the action $Q \ca^\sigma N$ is trivial.

We now introduce the following classes of groups that shall play an extremely important role throughout the rest of the paper.
\begin{definition}
	We denote by $\mathcal{R}ip (Q)$ the class of all semidirect product $G= N\rtimes Q$ satisfying the properties of Theorem~\ref{beleosinripsconstr}, where $Q=S$, $Q$ and $H$ are torsion free and $H$ has property (T). 
	
	Moreover, when $Q$ has property (T), we denote the class $\mathcal {R}ip (Q)$ by $\mathcal {R}ip_{T} (Q)$.
\end{definition}

 Since property (T) is closed under extensions it follows  that all groups in $ \mathcal {R}ip_{T} (Q)$ have property (T).  Our rigidity results in Section \ref{superrigidRIPS} concern this class of groups.

\vskip 0.05in

In the second part of this section we recall a powerful method from geometric group theory, termed Dehn filling. We are interested specifically in the group theoretic Dehn filling constructions developed by D. Osin and his collaborators in \cite{Os06,DGO11}. The next result, which is due to Osin, is a technical variation of \cite[Theorem 1.1]{Os06} and \cite[Theorem 7.9]{DGO11} and plays a key role to derive some of our main rigidity theorems in Section \ref{superrigidRIPS} (see Theorems \ref{controlprodpropt1} and \ref{controlprodpropt2}). For its proof the reader may consult \cite[Corollary 5.1]{CIK13}.

\begin{theorem}[Osin] \label{dehnfilling} Let $H\leqslant G$ be infinite groups where $H$ is finitely generated and residually finite. Suppose that $G$ is hyperbolic relative to $\{H\}$. Then there exist a non-elementary hyperbolic group $K$ and an epimorphism $\delta : G \rightarrow K$ such that $R = ker(\delta)$ is isomorphic to a non-trivial (possible infinite) free product $R = \ast_ {g \in T} R^{g}_0$ , where $T \subset G$ is a subset and $R^g_0= gR_0g^{-1}$ for a finite index normal subgroup $R_0\lhd H$. 
\end{theorem}

We end this section with an application of Theorem~\ref{dehnfilling}. The result describes the structure of the normal subgroups $N$ of $N\rtimes Q\in  \mathcal {R}ip_{T} (Q)$. Namely, combining Theorems \ref{dehnfilling} and \ref{beleosinripsconstr} we show that these groups are free-by-hyperbolic. This result will be essential to the proof of Theorem \ref{semidirectprodreconstruction}. \\
\begin{proposition}\label{freebyhyperbolic} Let $G=N\rtimes Q \in  \mathcal {R}ip_{\mathcal T}(Q)$ and assume that $Q$ is an infinite residually finite group. Then $N$ is a $\mathbb F_{n+1}$-by-(non elementary, hyperbolic property (T)) group where $n \in \mathbb N \cup \{\infty\}$.
\end{proposition}

\begin{proof} Since $G$ is hyperbolic relative to $\{Q\}$ and $Q$ is residually finite then by Theorem \ref{dehnfilling} there is a non-elementary hyperbolic group $K$ and an epimorphism $\delta : G \rightarrow K$ such that $L = ker(\delta)$ is isomorphic to a non-trivial free product $L = \ast_ {g \in T} Q^{g}_0$ , where $T \subset G$ is a subset and $Q_0\lhd Q$ is finite index, normal subgroup. Since $G=N\rtimes Q$ and $Q_0$ is normal in $Q$ one can assume without any loss of generality that $T\subset N$. Next we show that $N\cap L$ infinite. If it would be finite, as $G$ is icc, it follows that $N\cap L=1$. As $N$ and $L$ are normal in $G$ then the commutator satisfies $[N,L]\leqslant N\cap L=1$ and hence $L\leqslant C_G(N)$. To describe this  centralizer fix $g= nq\in C_G(N)$ where $n\in N$, $q\in Q$. Thus for all $m\in N$ we have  $nq m=mnq $ and hence $n \sigma_q(m) = m n$, where $\sg_q(x)=q^{-1}xq$ for all $x\in N$. Therefore $\sigma_q=ad (n)$ and by part 5.\ in Theorem \ref{beleosinripsconstr} we must have that $q=1$. This further implies that $m\in Z(N)=1$ and hence $C_G(N)=1$; in particular, $L=1$ which is a contradiction. In conclusion $N\cap L\lhd N$ is an infinite normal subgroup. Using the isomorphism theorem we see that $N/ (N\cap L)\cong (NL)/L$. Also from the free product description of $L$ we see that $N\rtimes Q_0\leqslant NL$ and hence $[G:NL]<\infty$.  In particular $(NL)/L$ is a finite index subgroup of $G/L=K$ and hence $(NL)/L$ is a (non-elementary) hyperbolic, property (T) group. To finish our proof we only need to argue that $N\cap L$ is a free group with at least two generators.  Since $L=\ast_ {g\in T} Q^{g}_0$, by Kurosh theorem there exist a set $X\subset L$, a collection of subgroups $Q_i \leqslant Q_0$ together with elements $g_i\in L$ such that $N\cap L = F(X)\ast (\ast_{i\in I} Q_i^{g_i}) $; here $F(X)$ is a free group with free basis $X$. In particular, for every $i\in I$ the previous relation implies that $Q_i^{g_i}\leqslant N$ and writing $g_i=n_iq_i$ for some $n_i\in N$, $q_i\in Q$ we see that $Q_i^{q_i}\leqslant N$. As $Q_i^{q_i}\leqslant Q$ we conclude that $Q_i^{q_i}\leqslant N\cap Q=1$ and hence $Q_i=1$. Thus $N\cap L= F(X)$ and since $G$ is icc and $N\cap L$ is normal in $G$ we see that $|X|\geq 2$, which finishes the proof.\end{proof}


\section{Maximal von Neumann Subalgebras Arising from Groups Rips Construction} \label{sec4}

If $\mathcal M$ is a von Neumann algebra then a von Neumann subalgebra $\mathcal N\subset \mathcal M$ is called \emph{maximal} if there is no intermediate von Neumann subalgebra $\mathcal P$ so that $\mathcal N\subsetneq \mathcal P \subsetneq \mathcal M$. Understanding the structure of maximal subalgebras of a given von Neumann algebra is a rather difficult problem that plays a key role in the very classification of these objects. Despite a series of earlier remarkable successes on the study of maximal \emph{ amenable } subalgebras initiated by Popa \cite{Po83} and continued more recently \cite{Sh06,CFRW08,Ho14,BC14,BC15,Su18,CD19,JS19}, much less is known for the maximal ones. For instance Ge's question \cite[Section 3, Question 2]{Ge03} on the existence of non-amenable factors that possess maximal factors which are amenable was settled in the affirmative only very recently in the work of Y. Jiang and A. Skalski, \cite{JS19}. We also remark that the study of maximal (or by duality minimal) intermediate subfactors has recently led to the discovery of a rigidity phenomenon for the intermediate subfactor lattice in the case of irreducible finite index subfactors \cite{BDLR18}.
\vskip 0.05
in 

In this section we make new progress in this direction by describing several concrete collections of maximal subalgebras in the von Neumann algebras arising from the groups $\mathcal{R}ip (Q)$ introduced in the previous subsection (see Theorem \ref{maximalvN} below). In particular, these examples allow construction of property (T) von Neumann algebras which have maximal von Neumann subalgebras without property (T). This answers a question raised by Y. Jiang and A. Skalski in \cite[Problem 5.5]{JS19} in the first version of their paper. Our arguments rely on the usage of Galois correspondence results for von Neumann algebras $\grave a$ $la$ Choda \cite{Ch78} and the classification of maximal subgroups in the monster-type groups provided in Theorem \ref{monster}. We remark that in the second version of their paper \cite[Theorem 4.8]{JS19}, Y. Jiang and A. Skalski independently obtained a different solution, using different techniques.    

\vskip 0.04in
First we need a couple of basic lemmas concerning automorphisms of groups. For the reader's convenience we include short proofs.

\begin{lemma} \label{liftouter}
	Let $N$ be a group, let ${\rm Id}\neq \alpha\in {\rm Aut}(N)$ and denote by $N_1 =\{ n \in N \,|\, \alpha(n)=n \}$ its fixed point subgroup. Then the following hold: 
	\begin{enumerate}\item Either $[N:N_1]=\infty$ or there is a subgroup $N_0\leqslant N_1\leqslant N$ that is normal in $N$ with $[N:N_0]<\infty$ and such that the induced automorphism $\tilde\alpha \in {\rm Aut}(N/C_N(N_0))$ given by $\tilde \alpha(n C_N(N_0))=\alpha(n) C_N(N_0)$ is the identity map; in particular, when $N$ is icc we always have $[N:N_1]=\infty$.   
		\item Either $[N:N_1]=\infty$, or $\alpha$ has finite order in ${\rm Aut}(N)$, or there is a $k\in \mathbb N$ and a subgroup $N_0\leqslant N_1\leqslant N$ that is normal in $N$ with $[N:N_0]<\infty$ and such that the induced automorphism $\tilde\alpha \in {\rm Aut}(N/Z(N_0))$ given by $\tilde \alpha(n Z(N_0))=\alpha(n) Z(N_0)$ has order $k$; in particular, when all finite index subgroups of $N$ have trivial center we either have $[N:N_1]=\infty$, or, $\tilde \alpha$ has finite order.   
	\end{enumerate}	
\end{lemma}

\begin{proof} 1. Assume that $2\leq [N:N_1]<\infty$. Then $N_0:= \cap_{h \in N}hN_1 h^{-1}\leqslant N_1$ is a finite index normal subgroup of $N$. Notice that the centralizer $C_N(N_0)$ is also normal in $N$.  Let $n \in N$ and $n_0 \in N_0$. As $N_0$ is normal we have  $n n_0n^{-1} \in N_0\leqslant N_1$ and hence $ n n_0 n^{-1}=\alpha(n n_0 n^{-1})=\alpha(n) n_0 \alpha(n^{-1})$. This implies that  $n_0^{-1}n^{-1} \alpha(n)n_0= n^{-1}\alpha(n)$ and hence $n^{-1}\alpha(n)\in C_N(N_0)$. Since $\alpha$ acts identically on $N_0$ one can see that $\alpha(C_N(N_0))=C_N(N_0)$. Thus one can define an automorphism  $\tilde \alpha: N/C_N(N_0)\rightarrow N/C_N(N_0)$  by letting $\tilde \alpha(n C_N(N_0))=\alpha(n)C_N(N_0)$. However the previous relations show that $\tilde \alpha$ is the identity map, as desired. For the remaining part of the statement, we notice that if $ [N:N_1]<\infty$ and $N$ is icc then the centralizer $C_N(N_0)$ is trivial and hence $\alpha=Id$, which is a contradiction.  
	
	\vskip 0.07in 
	
	2. Assume $[N:N_1]<\infty$ and $\alpha$ has infinite order in ${\rm Aut}(N)$. Also for each  $i\geq 2$ denote by $N_i=\{n\in N \,|\, \alpha^i(n)=n \}$ and notice that $N_1 \leqslant N_i \leqslant N_{i+1}\leqslant N$. Since $[N :N_1]<\infty$ there is $s \in \mathbb N$ so that, either $N_s=N_{l}$ for all $l\geq s$, or $N_s=N$. If $N_s=N$ then $\alpha^s=Id$, contradicting the infinite order assumption on $\alpha$. Now assume that $N_s=N_{s+1}$. For every $n \in N_{s+1}$ we have $\alpha^s(n)=\alpha^{s+1}(n)$ and thus $\alpha(n)=n$ which is equivalent to $n \in N_1$. This shows that $N_1=N_{s+1}$ and combining with the above we conclude  that  $N_1= N_{i}$ for all $i$.
	\vskip 0.05in	
	As $[N:N_1]<\infty$ then $N_0:= \cap_{h \in N}hN_1 h^{-1}\leqslant N_1$ is a finite index normal subgroup of $N$. $\alpha$ induces an automorphism $\tilde{\alpha}$ on the quotient group $N/N_0$ by $\tilde{\alpha}( n N_0)= \alpha(n)N_0$ for all $n\in N$. Since $[N:N_0]<\infty$ there is $k \in \mathbb N$ such that $\tilde\alpha^k=Id$ on $N/N_0$. Thus for every $n \in N$ we have $n^{-1}\alpha^k(n) \in N_0$.
	
	\vskip 0.05in 
	Let $n \in N$ and $n_0 \in N_0$. By normality we have  $n n_0n^{-1} \in N_0\leqslant N_1$ and hence $ n n_0 n^{-1}=\alpha^k(n n_0 n^{-1})=\alpha^k(n) n_0 \alpha^k(n^{-1})$. This implies that  $n_0^{-1}n^{-1} \alpha^k(n)n_0= n^{-1}\alpha^k(n)$ and hence $n^{-1}\alpha^k(n)\in Z(N_0)$. Since $N_0$ is normal in $N$, so is $Z(N_0)$. Since $\alpha$ leaves $Z(N_0)$ invariant, the map $\tilde\alpha : N/Z(N_0))\rightarrow N/Z(N_0)$ given by $\tilde \alpha(n Z(N_0))=\alpha(n) Z(N_0)$ is an automorphism. The previous relations show that it has order $k$. \end{proof}

Using this we will see that, in the case of icc groups, outer group actions $Q\ca N$ by automorphisms lift to outer actions \ $Q\ca \mathcal L(N)$ at the von Neumann algebra level. More precisely we have the following 
\begin{lemma} \label{trivrel}
	Let $N$ be an icc group and let $Q$ be a group together with an outer action $Q \ca^{\sg} N$. Then $\mathcal L(N)' \cap \mathcal L(N \rtimes_{\sg}Q)= \mc$.
\end{lemma}

\begin{proof} To get $\mathcal L(N)' \cap \mathcal L(N \rtimes_{\sg}Q)= \mc$ it suffices to show that for all $g \in (N \rtimes_{\sg}Q) \setminus \{e\}$ the $N$-conjugacy orbit $\mathcal O_N(g)=\{ngn^{-1}: n \in N\}$ is infinite. Suppose by contradiction there is $h=n_0q_0 \in (N\rtimes Q) \setminus \{e\}$ with $n_0\in N$ and $q_0\in Q$ such that $|\mathcal O_N(h)|<\infty$. Hence there exists a finite index subgroup $N_1\leqslant N$ such that $n h n^{-1}=h$ for all $n\in N_1$. This entails that  $n n_0 q_0 n^{-1}= n_0q_0$ and thus $n= n_0q_0 n q_0^{-1} n_0^{-1} = {\rm ad} (n_0) \circ \sigma_{q_0}(n)$ for all $n\in N_1$. Also, since $N$ is icc, we have that $q_0\neq e$. Let $\alpha={\rm ad}(n_0) \circ \sigma_{q_0}$.
	Since $Q\ca N$ is outer it follows that  ${\rm Id}\neq \alpha \in {\rm Aut}(N)$. Since $N$ is icc and $[N:N_1]<\infty$ then the first part in Lemma \ref{liftouter} leads to a contradiction.\end{proof}

With these results at hand we are now ready to deduce the main result of the section.

\begin{notation}\label{ripsvn} Fix any rank one group $Q_m$. Consider the lacunary hyperbolic groups $Q$ from Theorem \ref{monster} where the maximal rank one subgroups of $ Q$ are isomorphic to $Q_m$. Also let $N\rtimes Q\in \mathcal Rip (Q)$ be the semidirect product obtained via the Rips construction together with the subgroups $N\rtimes Q_m<N\rtimes Q$. Throughout this section we will consider the corresponding von Neumann algebras $\mathcal M_m :=\mathcal L(N\rtimes Q_m)\subset \mathcal L(N\rtimes Q):=\mathcal M$.  
\end{notation}

Assuming Notation \ref{ripsvn}, we now show the following

\begin{theorem}\label{maximalvN}  $\mathcal M_m$ is a maximal von Neumann algebra of $\mathcal M$. In particular, when $N\rtimes Q\in \mathcal Rip_{\mathcal T}(Q)$  then $\mathcal M_m$ is a non-property (T) maximal von Neumann subalgebra of a property (T) von Neumann algebra $\mathcal M$.   
\end{theorem}

\begin{proof} Fix $\mathcal P$ be any intermediate subalgebra $\mathcal M_m\subseteq \mathcal P\subseteq \mathcal M$.  Since $\mathcal M_m \subset \mathcal M$ is spatially isomorphic to the crossed product inclusion $\mathcal L(N)\rtimes Q_m \subset \mathcal L(N)\rtimes Q$ we have $\mathcal L(N)\rtimes Q_m \subseteq \mathcal P \subseteq \mathcal L(N)\rtimes Q$. By Lemma \ref{trivrel} we have that $(\mathcal L(N)\rtimes Q_m)' \cap (\mathcal L(N)\rtimes Q ) \subseteq \mathcal L(N)' \cap ( \mathcal L (N)\rtimes Q)=\mc$. In particular, $\mathcal P$ is a factor. Moreover, by the Galois correspondence theorem \cite{Ch78} (see also \cite[Corollary 3.8]{CD19}) there is a subgroup $Q_m\leqslant K\leqslant Q$ so that $\mathcal P= \mathcal L(N)\rtimes K$. Since by construction, $Q_m$ is a maximal subgroup of $Q$, we must have that $K=Q_m$ or $Q$. Thus we get that $\mathcal P = \mathcal M_m$ or $\mathcal M$ and the conclusion follows. \\
	For the remaining part note that $\emm$ has property (T) by \cite{CJ}. Also, since $N \rtimes Q_m$ surjects onto an infinite abelian group then it does not have property (T). Thus by \cite{CJ} again $\mathcal M_m =\mathcal L(N\rtimes Q_m)$ does not have property (T) either.
\end{proof}

As pointed out at the beginning of the section, the above theorem provides a positive answer to \cite[Problem 5.5]{JS19}. Another solution to the problem of finding maximal subalgebras without property (T) inside factors with property (T) was also obtained independently by Y. Jiang and A. Skalski in the most recent version of their paper. Their beautiful solution has a different flavor from ours; even though the Galois correspondence theorem $\grave a$ $la$ Choda is a common ingredient in both of the proofs. Hence we refer the reader to \cite[Theorem 4.8]{JS19b} for another solution to the aforementioned problem.
Also note that while the algebras $\mathcal M_m$ do not have property (T) they are also non-amenable. In connection with this it would be very interesting if one could find an example of a property (T) II$_1$ factor which have  maximal hyperfinite subfactors. This is essentially Ge's question but for property (T) factors.

\vskip 0.07in

In the final part of the section we show that whenever $Q_\iota$ is not isomorphic to $Q_\kappa$ then the resulting maximal von Neumann subalgebras $\mathcal M_m$ and $\mathcal M_n$ are non-isomorphic. In fact we have the following more precise statement

\begin{theorem} Assume that $Q_\iota, Q_\kappa <(\mathbb Q,+)$ and let $\Theta: \mathcal  M_\iota \rightarrow \mathcal M_\kappa$ be a $\ast$-isomorphism. Then there exists a unitary $u\in \mathcal U(\mathcal M_\kappa)$ such  that ${\rm ad}(u)\circ \Theta: \mathcal L(N_1) \rightarrow \mathcal L(N_2)$ is a $\ast$-isomorphism. Moreover there exist a group isomorphism $\delta: Q_\iota \rightarrow Q_\kappa$ and a $1-$cocycle $r: Q_\kappa \rightarrow \mathcal U(\mathcal L(N_2))$ such that for all $a\in \mathcal L(N_1)$ and $g\in Q_\iota$ we have ${\rm ad }(u)\circ \Theta (a u_g)= {\rm ad }(u) \circ \Theta(a ) v_{\delta(g)} r_{\delta(g)}$. In particular, we have ${\rm ad}(u)\circ \Theta\circ\alpha_g= {\rm ad}(r_{\delta(g)})\circ \beta_{\delta(g)}\circ {\rm ad}(u)\circ \Theta$.
	
\end{theorem} 
\begin{proof} Identify $\mathcal M_\iota= \mathcal L(N_1)\rtimes Q_\iota$ and $\mathcal M_\kappa= \mathcal L(N_2)\rtimes Q_\kappa$ and let $\Theta: \mathcal L(N_1)\rtimes Q_\iota\rightarrow \mathcal L(N_2)\rtimes Q_\kappa$ be the $\ast$-isomorphism. Notice that since $\Theta(\mathcal L(N_1))$ has property (T) and $Q_\kappa$ is amenable then by \cite{Po01}
	 we have that $\Theta(\mathcal L(N_1))\prec_{\mathcal M_\kappa} \mathcal L(N_2)$. Also by Lemma \ref{trivrel} we note that $\Theta ( \mathcal L(N))$ is a regular irreducible subfactor of $\mathcal M_\kappa$, i.e.\ $\Theta(\mathcal L(N_1))'\cap \mathcal M_\kappa= \Theta ( \mathcal L(N_1)'\cap \mathcal M_\iota)= \mathbb C 1$. Similarly, $\mathcal L(N_2)$ is a regular irreducible subfactor of $\mathcal M_\kappa$ satisfying $\mathcal L(N_2)\prec_{\mathcal M_\kappa} \Theta(\mathcal L(N_1))$. Thus by the proof of \cite[Lemma 8.4]{IPP05}, since $Q_\iota$'s are torsion free, one can find a unitary $u\in \mathcal M_\kappa$ such that ${\rm ad}(u)\circ \Theta(\mathcal L(N_1))=\mathcal L(N_2)$. So replacing $\Theta$ with ${\rm ad}(u)\circ \Theta$ we can assume that $\Theta(\mathcal L(N_1))=\mathcal L(N_2)$. Hence for every $g\in Q_\iota$ we have that $\Theta(\alpha_{g}(x))\Theta(u_g)=\Theta(u_g)\Theta(x)$ for all $x\in \mathcal L(N_1)$. Consider the Fourier decomposition of $\Theta(u_g)=\sum_{h\in Q_\kappa} n_h v_h$ where $n_h\in \mathcal L(N_2)$. Using the previous relation we get that $\Theta(\alpha_{g}(x)) n_h = n_h \beta_h\Theta(x)$ for all $h\in Q_\kappa$ and $x\in \mathcal L(N_2)$. Thus $n_h {n_h}^*\in \mathcal L(N_2)'\cap \mathcal M_\kappa= \mathbb C 1$ and hence there exist unitary $t_h\in\mathcal L(N_2)$ and scalar $s_h\in \mathbb C$ so that $n_h = s_h t_h$. Assume there exist $h_1\neq h_2\in Q_\kappa$   so that $s_{h_1},s_{h_2}\neq 0$. This implies that $\Theta(\alpha_{g}(x))  = t_{h_1} \beta_{h_1}\Theta(x) t_{h_1}^*= t_{h_2} \beta_{h_2}\Theta(x) t_{h_2}^*$ for all $x\in \mathcal L(N_2)$. Thus $\beta_{h_1}(t_{h_1}^*t_{h_2}) v_{{h{_1}}^{-1}h_2}= v_{h_1}^* t_{h_1}^* t_{h_2} v_{h_2}\in \mathcal L(N_2)'\cap \mathcal M_\kappa= \mathbb C 1$. Thus $h_1^{-1}h_2=1$ and $h_1=h_2$ which is a contradiction. In particular there exists a unique $\delta(g)\in Q_\kappa$ so that $s_k=0$ for all $k\in Q_\kappa\setminus \{ \delta(g)\}$. Altogether these show that there is a well-defined map $\delta:Q_\iota \rightarrow Q_\kappa$ so that $\Theta(u_g)= n_{\delta(g)} v_{\delta(g)}$ for all $g\in Q_\iota$. It is easy to see that $\delta$ is a group isomorphism and the map $r: Q_\kappa\rightarrow \mathcal U(\mathcal L(N_2))$ given by $r(h)=\beta_{
		h}(n_{h})$ is a 1-cocycle i.e. $r(h k)= c_h \beta_h( c_k)$. \end{proof}
	\noindent{\bf Final Remarks}. We notice that our strategy from the proof of Theorem \ref{maximalvN} can also be used to produce other examples  of non-property (T) subalgebras in property (T) factors. Indeed for $Q$ in the Rips construction one can take in fact \emph{any} torsion free, property (T) monster group $Q$ in the sense of Ol'shanskii. If one picks any maximal subgroup $Q_0<Q$ then, as before, the group von Neumann algebra $\mathcal L(N\rtimes Q_0)$ will obviously be maximal in $\mathcal L(N\rtimes Q)$. Notice that since $Q_0<Q$ is maximal then $Q_0$ is infinite index in $Q$. To see this note that if $Q_0$ is finite index in $Q$, then $Q_0$ has property (T) and hence is finitely generated. Therefore $Q_0$ would be abelian and hence trivial, which is a contradiction. Therefore $Q_0$ must have infinite index in $Q$. In this case it is either finitely generated, in which case is abelian or it is infinitely generated.  However in both scenarios $Q_0$ does not have (T) and hence neither does $N\rtimes Q_0$. Thus by \cite{CJ}, $\mathcal L(N\rtimes Q_0)$ does not have property (T).  
	

\section{Von Neumann Algebraic Rigidity Aspects for Groups Arising via Rips Constructions}\label{superrigidRIPS}

An impressive milestone in the classification of von Neumann algebras was the emergence over the past decade of the first examples of groups $G$ that can be completely reconstructed from their von Neumann algebras $\El(G)$, i.e.\ \emph{$W^*$-superrigid groups} \cite{IPV10,BV12,CI17}. The strategies used in establishing these results share a common key ingredient, namely, the ability to first reconstruct from $\El(G)$ various algebraic feature of $G$ such as its (generalized) wreath product decomposition in \cite{IPV10,BV12}, and respectively, its amalgam splitting in \cite[Theorem A]{CI17}. This naturally leads to a broad and independent study, specifically identifying canonical group algebraic features of a group that pass to its von Neumann algebra. While several works have emerged recently in this direction \cite{CdSS15,CI17,CU18} the surface has been only scratched and still a great deal of work remains to be done. 

A difficult conjecture of Connes predicts that all icc property (T) groups are $W^*$-superrigid. Unfortunately, not a single example of such group is known at this time. Moreover, in the current literature there is an almost complete lack of examples of algebraic features occurring in a property (T) group that are recognizable at the von Neumann algebraic level.  In this section we make progress on this problem for property (T) groups that appear as certain fiber products of Belegradek-Osin Rips type constructions.  Specifically, we have the following result:

\begin{theorem}\label{semidirectprodreconstruction} Let $Q=Q_1\times Q_2$, where $Q_i$ are icc, torsion free, biexact, property (T), weakly amenable, residually finite groups. For $i=1,2$ let $N_i\rtimes_{\sigma_i} Q\in \mathcal Rip_T(Q)$ and denote by $\G=(N_1\times N_2)\rtimes_{\sigma} Q$ the semidirect product associated with the diagonal action $\sigma=\sigma_1 \times \sigma_2 :Q\ca N_1\times N_2 $. Denote by $\emm=\El(\G)$ be the corresponding II$_1$ factor. Assume that $\Lambda$ is any arbitrary group and $\Theta: \El(\Gamma)\rar \El(\Lambda)$ is any $\ast$-isomorphism. Then there exist group actions by automorphisms $H\ca^{\tau_i} K_i$  such that $\La= (K_1\times K_2) \rtimes_{\tau} H$ where $\tau=\tau_1\times \tau_2: H\ca K_1\times K_2$ is the diagonal action. Moreover one can find a multiplicative character $\eta:Q\rar \mathbb T$, a group isomorphism $\delta: Q\rar H$, a unitary $w\in \El(\Lambda)$, and $\ast$-isomorphisms $\Theta_i: \El(N_i)\rar \El(K_i)$ such that for all $x_i \in \mathcal L(N_i)$ and $g\in Q$ we have 
	\begin{equation}
	\Theta((x_1\otimes x_2) u_g)= \eta(g) w ((\Theta_1(x_1)\otimes \Theta(x_2)) v_{\delta(g)})w^*.
	\end{equation}      
	Here $\{u_g \,|\,g\in Q\}$ and $\{v_h\,|\, h\in H\}$ are the canonical unitaries implementing the actions of $Q \ca \El(N_1)\bar\otimes \El(N_2)$ and $H\ca \El(K_1)\bar\otimes \El(K_2)$, respectively.
\end{theorem}

From a different perspective our theorem can be also seen as a von Neumann algebraic superrigidity result regarding conjugacy of actions on noncommutative von Neumann algebras. Notice that very little is known in this direction as well, as most of the known superrigidity results concern algebras arising from actions of groups on probability spaces.    
\vskip 0.07in
We continue with a series of preliminary results that are essential to derive the proof of Theorem \ref{semidirectprodreconstruction} at the end of the section. First we present a location result for commuting diffuse property (T) subalgebras inside a von Neumann algebra arising from products of relative hyperbolic groups.

\begin{theorem}\label{controlprodpropt1} For $i=\overline{1, n}$ let  $H_i<G_i$ be an inclusion of infinite groups such that $H_i$ is residually finite and $G_i$ is hyperbolic relative to $H_i$. Denote by $H=H_1\times ...\times H_n < G_1\times ...\times G_n=G$ the corresponding direct product inclusion. Let $\enn_1,\enn_2 \subseteq \El(G)$ be two commuting von Neumann subalgebras with property (T). Then for every $i\in\overline{1,n}$ there exists $k\in \overline{1,2}$ such that $\enn_k \prec \El(\hat G_i\times H_i)$, where $\hat G_i :=\times _{j\neq i} G_j$.
	

\end{theorem}     

\begin{proof} 
Fix $i\in\overline{1,n}$. Since $H_i$ is residually finite then using Theorem \ref{dehnfilling} there is a short exact sequence $$1\rar ker(\pi_i)\hookrightarrow G_i \xrightarrow{\pi_i} F_i\rar 1,$$ 
where $F_i$ is a non-elementary hyperbolic group and $\ker(\pi_i)=\langle H^0_i\rangle= \ast_{t\in T_i} (H_i^0)^t$, for some subset $T\subset G_i$ and a finite index normal subgroup $H_i^0\lhd H_i$. 

Following \cite[Notation 3.3]{CIK13} we now consider the von Neumann algebraic embedding corresponding to $\pi_i$, i.e.\ $\Pi_i: \El(G)\rar \El(G)\bar\otimes \El(F_i)$ given by $\Pi_i(u_g) = u_g \otimes v_{\pi_i(g_i)}$ for all $g=(g_j)\in G$; here $u_g$'s are the canonical unitaries of $\El(G)$ and $v_h$ are the canonical unitaries of $\El(F_i)$.  From the hypothesis we have that $\Pi_i(\enn_1),\Pi_i(\enn_2)\subset \El(G)\bar\otimes \El( F_i)=:\tilde \emm_i$ are commuting property (T) subalgebras. Fix $\mathcal A\subset \Pi_i(\enn_1)$ any diffuse amenable von Neumann subalgebra. Using \cite[Theorem 1.4]{PV12} we have either: a) $\mathcal A\prec_{\tilde \emm_i} \El(G)\bar\otimes 1$, or 
b) $\Pi_i(\enn_2)$ is amenable relative to $\El(G)\bar\otimes 1$ inside $\tilde \emm_i$.

Since the $\enn_k$'s have property (T) then so do the $\Pi_i(\enn_k)$'s. Thus using part b) above we get that $\Pi_i(\enn_2)\prec_{\tilde \emm}\El(G)\bar\otimes 1$. On the other hand, if case a) above were to hold for all $\mathcal A$'s then by  \cite[Corollary F.14]{BO08} we would get that $\Pi_i(\enn_1)\prec_{\tilde \emm_i}\El(G)\bar\otimes 1$. Therefore we can always assume that  $\Pi_i(\enn_k) \prec_{\tilde \emm_i} \El(G)\bar\otimes 1$ for $k=1 \text{ or } 2$.  

Due to symmetry we only treat $k=1$. Using \cite[Proposition 3.4]{CIK13} we get that $\enn_1\prec \El(\ker(\Pi_i))=\El(\hat G_i\times \ker(\pi_i))$. Thus there exist nonzero projections $p\in \enn_1$, $q\in \El(\hat G_i \times \ker(\pi_i))$, nonzero partial isometry $v\in \emm$ and a $\ast$-isomorphism $\phi: p\enn_1p \rar \mathcal B:=\phi(p\enn_1p)\subset q\El(\hat G_i\times \ker(\pi_i))q$  on the image such that
\begin{equation}
\phi(x)v=vx\text{ for all } x\in p\enn_1p.
\end{equation}
Also notice that since $\enn_1$ has property (T) then so does $p\enn_1p$ and therefore  $\mathcal B\subseteq q\El(\hat G_i \times \ker(\pi_i))q$ is a property (T) subalgebra. Since $\ker(\pi_i)=\ast_{t\in T} (H^0_i)^t $ then by further conjugating $q$ in the factor $\El(\hat G_i\times \ker(\pi_i))$ we can assume that there exists a unitary $u \in \El(\hat G_i\times \ker(\pi_i))$ and a projection $q_0\in \El(\hat G_i)$ such that  $\mathcal B\subseteq u(q_0\El(\hat G_i)q_0) \bar \otimes \El(\ker(\pi_i) ) u^*$. Using property (T) of $\mathcal B$ and \cite[Theorem]{IPP05} we further conclude that there is $t_0\in T$ such that $\mathcal B\prec_{u (q_0\El(\hat G_i)q_0 \bar\otimes \El(\ker (\pi_i)))u^*} u(q_0\El(\hat G_i)q_0\otimes \El((H_i^0)^{t_0}))u^*$. Composing this intertwining with $\phi$ we finally conclude that $\enn_1\prec_\emm \El(\hat G_i\times H_i^0)$, as desired.      
\end{proof}

\begin{theorem}\label{controlprodpropt2} Under the same assumptions as in Theorem \ref{controlprodpropt1} for every $k\in \overline{1,n}$ one of the following must hold \begin{enumerate}
\item [1)] there exists $i\in 1,2$ such that $\enn_i \prec_\emm \El(\hat G_k)$; 
\item [2)] $\enn_1\vee \enn_2 \prec_{\emm} \El(\hat G_k \times H_k)$.
\end{enumerate}

\end{theorem}
\begin{proof} From Theorem ~\ref{controlprodpropt1} there exists $i\in \overline{1,2}$ such that $\enn_i\prec \El(\hat G_k\times H_k)$. For convenience assume that $i=1$. Thus there exist nonzero projections $p\in \enn_1$, $q\in \El(\hat G_k \times H_k)$, nonzero partial isometry $v\in \emm$ and a $\ast$-isomorphism $\phi: p\enn_1p \rar \mathcal B:=\phi(p\enn_1p)\subset q\El(\hat G_k\times H_k)q$  on the image such that
\begin{equation}\label{int1}
\phi(x)v=vx\text{ for all } x\in p\enn_1p.
\end{equation}
Notice that $q\geq vv^*\in \mathcal B'\cap q\emm q$ and  $p\geq v^*v \in p\enn_ip '\cap p\emm p$. Also we can pick $v$ such that $s(E_{\El(\hat G_k \times H_k)}(vv^*))=q$.
Next we assume that $\mathcal B\prec_{L(\hat G_k \times H_k)} L(\hat G_k)$.  Thus there exist nonzero projections $p'\in \mathcal B$, $q'\in \El(\hat G_k)$, nonzero partial isometry $w\in q'\El(\hat G_k\times H_k)p'$ and a $\ast$-isomorphism $\psi: p'\mathcal B p' \rar q'\El(\hat G_k)q'$  on the image such that
\begin{equation}\label{int2}
\psi(x)w=wx\text{ for all } x\in p'\mathcal Bp'.
\end{equation}  

Notice that $q\geq p'  \geq ww^*\in (p'\mathcal Bp')' \cap p'\emm p'$ and $q' \geq w^*w\in \psi(p'\mathcal Bp')'\cap q'\emm q'$. Using \eqref{int1} and \eqref{int2} we see that \begin{equation}\label{int3}\psi(\phi(x))wv=w\phi(x)v=wvx\text{ for all }x\in p_0\enn_ip_0,\end{equation} where $p_0 \in \enn_i$ is a projection picked so that $\phi(p_0)=p'$. Also we note that if $0=wv$ then $0= wvv^*$ and hence $0=E_{\El(\hat G_k \times H_k)}(wvv^*)=w E_{\El(\hat G_k \times H_k)}(vv^*)$. This further implies that $0= w s( E_{\El(\hat G_k \times H_k)}(vv^*)) = wq =w$ which is a contradiction. Thus $wv\neq 0$ and taking the polar decomposition of $wv$ we see that  \eqref{int3} gives 1).

Next we assume that $\mathcal B\nprec_{\El(\hat G_k \times H_k)} \El(\hat G_k)$. Since $G_k$ is hyperbolic relative to $H_k$ then by Lemma \ref{malnormalcontrol}  we have that for all $x, x_1 x_2,..., x_l \in M$ such that $\mathcal Bx\subseteq \sum^l_{i=1} x_i \mathcal B$ we must have that $x\in \El(\hat G_k \times H_k)$. Hence in particular we have that $vv^*\in \mathcal B'\cap q\emm q\subseteq \El(\hat G_k\times H_k)$ and thus relation \eqref{int1} implies that $\mathcal Bvv^*=v\enn_i v^*\subseteq \El(\hat G_k\times H_k)$. Also for every $c\in \enn_{i+1}$ we can see that 
\begin{equation}\begin{split}
\mathcal Bvcv^*&= \mathcal B vv^* v c v^*= v\enn_i v^*v cv^*= vv^*v c \enn_i v^*\\&=vc \enn_i v^*= vc \enn_i v^*vv^*= vcv^*v \enn_i v^* = vcv^* \mathcal Bvv^*= vcv^* \mathcal B. \end{split}
\end{equation}
Therefore by Lemma \ref{malnormalcontrol} again we have that $vcv^*\in \El(\hat G_k \times H_k)$ and hence $v\enn_{i+1}v^*\subseteq \El(\hat G_k \times H_k)$. Thus $v \enn_i \enn_{i+1} v^* = vv^*v \enn_i \enn_{i+1} v^*= v\enn_i v^* v \enn_{i+1} v^*\subseteq \El(\hat G_k \times H_k)$, which by Popa's intertwining techniques implies that $\enn_1\vee \enn_2 \prec \El(\hat G_k \times H_k)$, i.e.\ 2) holds. \end{proof}

We now proceed towards proving the main result of this chapter. To simplify the exposition we first introduce a notation that will be used throughout the section.  
\begin{notation}\label{semidirectt} Denote by $Q= Q_1\times Q_2$, where $Q_i$ are infinite, residually finite, biexact, property (T), icc groups. Then consider $\G_i = N_i \rtimes Q\in \mathcal Rip_T(Q)$ and consider the semidirect product $\G= (N_1\times N_2)\rtimes_\sigma Q$ arising from the diagonal action  $\sigma=\sigma_1\times \sigma_2: Q\rar Aut(N_1\times N_2)$, i.e. $\sigma_g (n_1,n_2)=( (\sigma_1)_g(n_1), (\sigma_2)_g(n_2))$ for all $(n_1,n_2)\in N_1\times N_2$. For further use we observe that $\G$ is the fiber product $\G=\G_1\times_Q \G_2$ and thus embeds into $\G_1\times \G_2$ where $Q$ embeds diagonally into $Q\times Q$. Over the next proofs when we refer to this copy we will often denote it by ${\rm d}(Q)$. Also notice that $\G$ is a property (T) group as it arises from an extension of property (T) groups. Furthermore, $\G_1,\G_2\in\mathcal Rip_T(Q)$ easily implies that $\G$ is an icc group.

 For future use, use also recall the notion of the comultiplication studied in \cite{IPV10,Io11}. Let $\G$ be a group as above, and assume that $\Lambda$ is a group such that $\El(\Gamma)=\El(\Lambda)=\emm$. Then the ``comultiplication along $\Lambda$'' $\Delta: \emm\rar \emm\bar \otimes \emm$ is defined by $\Delta (v_\lambda)=v_\lambda \otimes v_\lambda$, for all $\lambda \in \Lambda$.
\end{notation}

\begin{theorem}\label{commutationcontrolincommultiplication} Let $\G$ be a group as in Notation \ref{semidirectt} and assume that $\Lambda$ is a group such that $\El(\Gamma)=\El(\Lambda)=\emm$. 
	Let $\Delta: \emm\rar \emm\bar \otimes \emm$ be the comultiplication along $\Lambda$ as in Notation \ref{semidirectt}. 
	 Then the following hold:
\begin{enumerate}
\item [3)] for all $j\in \overline{1,2}$ there is $i\in \overline{1,2}$ such that $\Delta(\El(N_i))\prec_{\emm\bar\otimes \emm} \emm\bar\otimes \El(N_j)$, and
\item [4)] 
\begin{enumerate} \item [a)] for all $j\in \overline{1,2}$ there is $i\in \overline{1,2}$ such that $\Delta (\El(Q_i))\prec_{\emm\bar\otimes \emm} \emm\bar\otimes \El(N_j)$ or 
\item [b)]$\Delta(\El(Q))\prec_{\emm\bar\otimes \emm} \emm\bar\otimes \El(Q)$; moreover in this case for every $j\in \overline{1,2}$ there is $i\in \overline{1,2}$ such that $\Delta (\El(Q_j))\prec_{\emm\bar\otimes \emm} \emm\bar\otimes \El(Q_i)$ 
\end{enumerate}   
\end{enumerate}

\end{theorem}

\begin{proof} Let $\tilde \emm=\El(\G_1\times \G_2)$. Since $\G<\G_1\times \G_2$ we notice the following inclusions $\Delta(\El(N_1)), \Delta
 (\El(N_2))\subset \emm\bar \otimes \emm = \El(\G\times \G)\subset \El(\G_1\times \G_2 \times \G_1\times \G_2)$.  Since $\G_i$ is hyperbolic relative to $Q$ then using Theorem \ref{controlprodpropt2} we have either 
 \begin{enumerate}
 \item [5)]there exists $i\in 1,2$ such that $\Delta(\El(N_i))\prec_{\tilde \emm \bten \tilde \emm} \emm\bar\otimes \El(\G_1)$, or  
 \item [6)]$\Delta(\El(N_1\times N_2))\prec_{\tilde \emm \bten \tilde \emm} \emm\bten \El(\G_1\times Q)$
\end{enumerate}  
Assume 5) holds. Since $\De(\El(N_i))\subset \emm\bten \El(\G)$ then by Lemma \ref{intertwiningintersection} there is a $h\in \G_1 \times \G_2 \times \G_1 \times \G_2 $ so that $\De(\El(N_i))\prec_{\tilde \emm \bten \tilde \emm} \El(\G\times (\G \cap h (\G_1 \times \G_2 \times \G_1) h^{-1}))=\El(\G \times (\G\cap \G_1 ))= \emm\bten (\El((N_1 \times N_2 )\rtimes {\rm d}(Q))\cap (N_1 \rtimes Q\times 1))= \emm\bten \El(N_1)$. Note that since $\Delta(\El(N_i))$ is regular in $\emm \tp \emm$, using Lemma ~\ref{intertwininglower1}, we get that $\Delta(\El(N_i))\prec_{ \emm \bten \emm} \emm\bar\otimes \El(\G_1)$, thereby establishing 3).

Assume 6) holds. Since $\De(\El(N_1\times N_2 ))\subset \El(\G\times \G )$ then by Lemma \ref{intertwiningintersection} there is $h\in \G_1 \times \G_2 \times \G_1 \times \G_2 $ such that 
$\De(\El(N_1\times N_2))\prec \El(\G\times (\G \cap h (\G_1 \times \G_2 \times \G_1\times Q) h^{-1}))=\El(\G \times (\G\cap (\G_1 \times h_4 Qh_4^{-1}) ))= \emm\bten \El((N_1 \times N_2 )\rtimes {\rm d}(Q))\cap (N_1 \rtimes Q\times h_4 Qh_4^{-1}))$. Since $h_4\in \G_2= N_2 \rtimes Q$ we can assume that $h_4\in N_2$. Notice that 
$((N_1 \times N_2 )\rtimes {\rm d}(Q))\cap (N_1 \rtimes Q\times h_4 Qh_4^{-1})= h_4((N_1 \times N_2 )\rtimes {\rm d}(Q))\cap (N_1 \rtimes Q\times Q )h_4^{-1}=  h_4((N_1 \times 1)\rtimes {\rm d}(Q))h_4^{-1}$ and hence $\De(\El(N_1\times N_2 ))\prec_{\tilde \emm\bten \tilde \emm} \emm\bten \El(N_1 \rtimes {\rm d}(Q))$. Moreover using Lemma \ref{intertwininglower2} we further have that $\De(\El(N_1\times N_2 ))\prec_{ \emm\bten \emm} \emm\bten \El(N_1 \rtimes {\rm d}(Q))$.

In conclusion, there exist a $\ast$-isomorphism on its image $\phi: p \De(\El(N_1\times N_2)) p\rar \mathcal B:= \phi(p\De(\El(N_1\times N_2)) p)\subseteq q \emm\bten \El(N_1 \rtimes {\rm d}(Q))$ and $0 \neq v\in q\emm \bten \emm p$ such that 
\begin{equation}\label{intertwiningni}
\phi(x)v=vx \text{ for all } x\in p \De(\El(N_1\times N_2)) p.
\end{equation}

Next assume that 3) doesn't hold. Thus proceeding as in the first part of the proof of Theorem \ref{controlprodpropt2}, we get  \begin{equation}\mathcal B\nprec_{\emm \bten (N_1\rtimes {\rm d}(Q))} \emm\bten \El(N_1)=: \emm_1.\end{equation}

Next we observe the following inclusions

\begin{equation}\begin{split}
 &\emm_1\rtimes_{1\otimes \sigma} {\rm d}(Q)=\emm \bten \El(N_1)\rtimes_{1\otimes \sigma} {\rm d}(Q)= \emm \bten \El(N_1 \rtimes_\sigma {\rm d}(Q)) \\ &\subset \emm \bten \El((N_1 \times N_2) \rtimes_\sigma {\rm d}(Q))= \emm \bten \El(N_1) \bten \El(N_2)\rtimes {\rm d}(Q)= \emm_1 \rtimes _{1\otimes \sigma} N_2 \rtimes {\rm d}(Q)
\end{split}\end{equation} 

Also since $Q$ is malnormal in $ N_2\rtimes Q$ it follows from Lemma \ref{malnormalcontrol} that $vv^* \in \emm\bten \El(N_1\rtimes {\rm d}(Q))$ and hence  $\mathcal Bvv^* \subset \emm\bten \El(N_1\rtimes {\rm d}(Q))$. Pick $u\in \mathcal {QN}_{p(\emm \bten \emm )p} (p\De(\El(N_1\times N_2))p)$ and using \eqref{intertwiningni} we see that there exist $n_1,n_2,...,n_s \in p(\emm\bten \emm) p$ satisfying

\begin{equation}\begin{split}
\mathcal Bvuv^*& = \mathcal Bvv^*v uv^*=v p(\De(\El(N_1\times N_2)))p v^*vnv^* = vp(\De(\El(N_1\times N_2)))pnv^*\subseteq \sum^s_{i=1} vn_i p(\De(\El(N_1\times N_2)))p v^*\\&= \sum^s_{i=1} vn_ip(\De(\El(N_1\times N_2)))p v^*vv^*= \sum^s_{i=1} vn_ipv^*v(\De(\El(N_1\times N_2)))p v^*=\sum^s_{i=1} vn_ipv^* \mathcal B vv^*=\sum^s_{i=1} vn_ipv^* \mathcal B.
\end{split}\end{equation}
 Then by Lemma \ref{malnormalcontrol} again we must have that $vuv^*\in \emm \bten \El(N_1 \rtimes {\rm d}(Q))$. Hence we have shown that 
 \begin{equation}\label{qnormcontainment1}
  v\mathcal {QN}_{p(\emm \bten \emm )p} (p\De(\El(N_1\times N_2))p)v^*\subseteq \emm\bten \El(N_1 \rtimes {\rm d}(Q)).
 \end{equation}
 
 Since $v^*v\in (p\De (\El(N_1\times N_2))p)'\cap p(\emm\bten \emm) p\subset \mathcal {QN}_{p(\emm\bten \emm )p} (p(\De(\El(N_1\times N_2)))p$ then  \eqref{qnormcontainment1} further implies that 
 \begin{equation}\label{qnormcontainment2}
  v\mathcal {QN}_{p(\emm\bten \emm )p} (p(\De(\El(N_1\times N_2)))p)''v^*\subseteq \emm\bten \El(N_1 \rtimes {\rm d}(Q)).
 \end{equation} 
Here for every inclusion of von Neumann algebras $\mathcal R\subseteq \mathcal T$ and projection $p\in \mathcal R$ we used the formula $\mathcal {QN}_{p\mathcal T p} (p\mathcal Rp)''= p\mathcal {QN}_{\mathcal T} (\mathcal R)''p$ \cite[Lemma 3.5 ]{Po03}. As   $vp\De(\emm)pv^* \subseteq v\mathcal {QN}_{p(\emm\bten \emm)p} (p(\De(\El(N_1\times N_2)))p)''v^*$ we conclude that $\De (\emm)\prec \El(N_1\rtimes Q)$ which contradicts the fact that $N_2$ is infinite. Thus 3) must always hold. 
\vskip 0.05in 
Next we derive 4). Again we notice that $\De(\El(Q_1))$, $\De(\El(Q_2))\subset \De(\emm)\subset \emm\bten \emm =\El(\G\times \G)\subset \El(\G_1\times \G_2 \times\G_1\times \G_2)$. Using Theorem \ref{controlprodpropt2} we must have that either 
\begin{enumerate}
\item[7)]\label{int5}  $\De(\El(Q_i))\prec_{\tilde \emm\bten \tilde \emm } \emm\bten \El(\G_1)$, or 
\item [8)]\label{int6} $\De(\El(Q))\prec_{\tilde \emm\bten \tilde \emm} \emm\bten \El(\G_1 \times Q)$.
\end{enumerate}  

Proceeding exactly as in the previous case, and using Lemma ~\ref{intertwininglower1}, we see that  7) implies $\De(\El(Q_i))\prec_{\emm\bten \emm} \emm\bten \El(N_1)$ which in turn gives  $4a)$. Also proceeding as in the previous case, and using Lemma ~\ref{intertwininglower2}, we see that 8) implies 
\begin{equation}\label{int7}\De(\El({\rm d}(Q))\prec_{ \emm\bten \emm} \emm\bten \El(N_1 \rtimes {\rm d}(Q)).  
\end{equation}

To show the part $4b)$ we will exploit \eqref{int7}. Notice that there exist nonzero projections $ r\in \De(\El(Q))$, $t\in \emm\bten \El(N_1\rtimes {\rm d}(Q))$, nonzero partial isometry $w\in r(\emm\bten \emm) t$ and $\ast$-isomorphism onto its image $\phi: r\De(\El(Q))r\rightarrow \mathcal C:= \phi(r\De(\El(Q))r)\subseteq t(\emm \bten \El(N_1\rtimes {\rm d}(Q)))t$ such that \begin{equation}\phi(x)w=wx \text{ for all }  x\in r\De(\El(Q))r.\end{equation}
Since $\El(Q)$ is a factor we can assume without loss of generality that $r=\De(r_1\otimes r_2)$ where $r_i\in \El(Q_i) $. Hence $\mathcal C= \phi(r \De(\El(Q))r)= \phi(\Delta (r_1 \El(Q_i)r_2))\bten r_2 \El(Q_2) r_2=: \mathcal C_1\vee \mathcal C_2$ where we denoted by $\mathcal C_i= \phi(\Delta (r_i \El(Q_i))r_i)\subseteq t(\emm\bten \El(N_1 \rtimes {\rm d}(Q)))t$. Notice that $\mathcal C_i$'s are commuting property (T) subfactors of $\emm \bten \El(N_1 \rtimes {\rm d}(Q))$. Since $N_i \rtimes Q$ is hyperbolic relative to $\{Q\}$ and seeing $\mathcal C_1\vee \mathcal C_2 \subseteq \emm \bten \El(N_i\rtimes {\rm d}(Q))\subset \El(\G_1\times \G_2 \times (N_1 \rtimes {\rm d}(Q)))$ then by applying Theorem \ref{controlprodpropt2} we have that there exits $i\in 1,2$ such that \begin{enumerate}
\item [9)] $\mathcal C_1 \prec_{\tilde \emm\bten \El(N_1 \rtimes {\rm d}(Q))} \El(\G_1\times \G_2)$ or 
\item [10)] $\mathcal C_1\vee \mathcal C_2 \prec_{\tilde \emm\bten \El(N_1 \rtimes {\rm d}(Q))} \El(\G_1\times \G_2\times {\rm d}(Q))$.
\end{enumerate}
Since $\mathcal C_1\subset \emm\bten \emm$ then 9) and Lemma \ref{intlowertensor} imply that $\mathcal C_1 \prec_{\emm \bten \emm} \emm\otimes 1$ which by \cite[Lemma 9.2]{Io11} further implies that $\mathcal C_1$ is atomic, which is a contradiction. Thus we must have 10). However since $\mathcal C_1\vee \mathcal C_2\subset \emm \bten \emm$ then 10) and Lemma \ref{intlowertensor} give that $\mathcal C_1\vee \mathcal C_2 \prec_{\emm \bten \emm} \emm\bten \El({\rm d}(Q))$ and composing this intertwining with $\phi$ (as done in the proof of the first case in Theorem \ref{controlprodpropt2}) we get that  $\De(\El(Q)) \prec_{\emm\bten \emm} \emm\bten \El({\rm d}(Q))$. Now we show the moreover part. So in particular the above intertwining shows that we can assume from the beginning that $\mathcal C=\mathcal C_1 \vee \mathcal C_2 \subset t(\emm\bten \El({\rm d}(Q))) t$. Since $Q_i$ are biexact, weakly amenable then by applying \cite[Theorem 1.4]{PV12} we must have that either $\mathcal C_1 \prec \emm\bten \El({\rm d}(Q_1))$ or $\mathcal C_2 \prec \emm\bten \El({\rm d}(Q_1))$ or $\mathcal C_1\vee \mathcal C_2$ is amenable relative to  $\emm\bten \El({\rm d}(Q_1))$ inside $\emm\bten \emm$. However since $\mathcal C_1 \vee \mathcal C_2$ has property (T) the last case above still entails that $\mathcal C_1\vee \mathcal C_2 \prec \emm\bten \El({\rm d}(Q_1))$ which completes the proof.  \end{proof}

  \begin{theorem}\label{toproductgroupcorners}Let $\G$ be a group as in Notation \ref{semidirectt} and assume that $\Lambda$ is a group such that $\El(\Gamma)=\El(\Lambda)=\emm$. Let $\Delta: \emm\rar \emm\bar \otimes \emm$ be the ``comultiplication along $\Lambda$'' as in the Notation~\ref{semidirectt}. Also assume for every $j\in 1,2$ there is $i\in 1,2$ such that either $\Delta (\El(Q_i))\prec_{\emm\bar\otimes \emm} \emm\bar\otimes \El(Q_j)$ or $\Delta (\El(Q_i))\prec_{\emm\bar\otimes \emm} \emm\bar\otimes \El(N_j)$. Then one can find subgroups $\Phi_1,\Phi_2 \leqslant \Phi\leqslant \La $ such that \begin{enumerate}
  \item $\Phi_1,\Phi_2$ are infinite, commuting, property (T), finite-by-icc groups;
  \item $[\Phi:\Phi_1\Phi_2]<\infty$ and $\mathcal{QN}^{(1)}_\La(\Phi)=\Phi$;
  \item there exist $\mu \in \mathcal U(\emm)$, $z\in \mathcal P(\mathcal Z(\El(\Phi)))$, $h= \mu z\mu^*\in \mathcal P (\El(Q))$ such that 
  \begin{equation} \mu \El(\Phi) z\mu^* = h \El(Q)h.
  \end{equation}
  \end{enumerate}
  
  \end{theorem}

\begin{proof}
	For the proof we use an approach based upon the methods developed in \cite{CdSS15,CI17,CU18}. For the reader's convenience we include all the details.
	
	Since the relative commutants $\El(Q_j)'\cap \emm$ and $\El(N_j)'\cap \emm$ are non-amenable then in both cases using \cite[Theorem 4.1]{DHI16} (see also \cite[Theorem 3.1]{Io11} and \cite[Theorem 3.3]{CdSS15}), one can find a subgroup $\Sigma <\Lambda$ with $C_\Lambda(\Sg)$ non-amenable such that $\El(Q_1)\prec_\emm \El(\Sg)$. 
	Thus there are $0\neq p\in \mathscr P(\El(Q_1))$, $0\neq f\in \mathscr P(\mathcal L(\Sg))$, a partial isometry $0\neq v\in f\emm p$ and a $\ast$-isomorphism onto its image $\phi:p \El(Q_1)p \rightarrow \B:=\phi(p \El(Q_1)p) \subseteq f \El(\Sg)f $ so that \begin{equation}\label{inteq1}
	\phi(x)v=vx, \text{ for all } x\in p\El(Q_1)p.
	\end{equation}
	Notice that $vv^*\in \cB'\cap f\M f$ and $v^*v\in (p \El(Q_1)p)'\cap p\emm p=\El(Q_2)p$. Then \eqref{inteq1}  implies that $\cB vv^* = v \El(Q_1)v^* = u_1 \El(Q_1) v^*v u^*_1$, where $u_1\in \sU(\cM)$ extends $v$. Passing to relative commutants we get $vv^*(\cB' \cap f\cM f)vv^*= u_1 v^*v( (p\El(Q_1) p)'\cap p\cM p )v^*v u^*_1=u_1 v^*v( p \El(Q_2)  )v^*v u^*_1 $. These relations further imply $vv^*(\cB \vee \cB'\cap f\cM f)vv^*=\cB vv^*\vee vv^*(\cB' \cap f\cM f)vv^*\subseteq u_1 \El(Q) u^*_1$. As $\El(Q)$ is a factor, there is a new $u_2 \in \mathscr U(\cM)$ with \begin{equation}(\cB \vee \cB'\cap f\cM f) z_2 \subseteq u_2 \El(Q) u^*_2.\end{equation} Here $z_2$ is the central support of $vv^*$ in $\cB \vee \cB'\cap f\cM f$ and hence $z_2\in \mathcal Z(\cB'\cap f\cM f )$ and $vv^*\leq z_2 \leq f$.

	\noindent Let $\Omega= C_\La (\Sigma)$ and notice that $\El(\Omega)z_2\subseteq ((fL(\Sigma )f)'\cap f \cM f )z_2\subseteq (\cB'\cap f \cM f)z_2\subseteq u_2 \El(Q)u_2^*$. Since $Q$ is malnormal in $\Gamma$ and $z_2 \in (L(\Omega)f)'\cap f\cM f$  we further have   $z_2 (\El(\Omega)f\vee ((\El(\Omega)f)'\cap f\cM f)) z_2\subseteq u_2 \El(Q)u_2^*$. Again since $\El(Q)$ is a factor there is $\eta\in \mathscr U(\cM)$ so that \begin{equation}\label{inteq2} (\El(\Omega)f\vee ((\El(\Omega)f)'\cap f\cM f)) z\subseteq \eta^* \El(Q)\eta,\end{equation} where $z$ is the central support of $z_2$ in $\El(\Omega)f\vee ((\El(\Omega)f)'\cap f\cM f)$. In particular, we have $vv^*\leq z_2 \leq z \leq f$. Now since  $f\El(\Sigma )f\subseteq  (\El(\Omega)f)'\cap f\cM f$ then by \eqref{inteq2} we get $(f \El(\Sigma )f\vee \El(\Omega)f)z\subseteq \eta^*\El(Q)\eta$ and hence 
	\begin{equation}\label{inteq3} \eta(\El(\Omega)f\vee f \El(\Sigma )f) z \eta^*\subseteq \El(Q).\end{equation}
	
	Since $vv^*\leq z\in (f \El(\Sigma)f)'\cap f\cM f$ and $\cB$ is a factor then the map $\phi': p \El(Q)p \rightarrow \eta  \cB z  \eta^* \subseteq f \El(\Sigma)fz$ given by $\phi'(x)=\eta  \phi(x)z \eta^*$ still defines a $\ast$-isomorphism that satisfies $\phi'(x) y=y x$, for any $x\in p \El(Q_1)p$, where $0\neq y= \eta zv$ is a partial isometry. Hence, $ \El(Q_1) \prec_{\cM}  u^*f \El(\Sigma)fzu$. Since $Q$ is malnormal in $\G$, it follows that $\El(Q_1)\prec_{\El(Q)} \eta f \El(\Sigma)fz\eta^*$.

	\noindent To this end, using \cite[Proposition 2.4]{CKP14} and its proof there are $0\neq a \in \mathscr P(\El(Q_1))$, $0\neq r=\eta qz\eta^* \in \eta f \El(\Sigma)fz\eta^*$ with $q\in \mathscr P(f \El(\Sigma)f)$, and a $\ast$-isomorphism onto its image $\psi : a\El(Q_1)a \rightarrow \D:= \psi(a\El(Q_1)a)\subseteq \eta q \El(\Sigma)qz\eta^*$ satisfying the following properties:
	\begin{enumerate}
		\item [4)]\label{finiteindex1} the inclusion $\D \vee (\D' \cap  \eta q \El(\Sigma)qz\eta^* ) \subseteq  \eta q \El(\Sigma)qz\eta^* $  has finite index;
		\item [5)]\label{intrel1} there is a partial isometry $0\neq w \in  \El(Q)$ such that $\psi(x) w =w x$ for all $x\in a \El(Q_1)a$. \end{enumerate}
	
	Now observe the algebras $\D$, $\D' \cap   \eta  q \El(\Sigma)q z \eta^*$ and $\eta  \El(\Omega)q z \eta^*$ are mutually commuting. Also the prior relations show that $\D$ and $\eta  \El(\Omega )q z\eta^*$ has no amenable direct summand. Since  $Q_1$ and $Q_2$ are bi-exact it follows that $\D' \cap  \eta  q \El(\Sigma)qz\eta^*$ must be purely atomic. Therefore, one can find $0\neq e \in \mathscr P(\mathcal Z(\D' \cap   u^* q \El(\Omega)q zu ))$ such that after cutting down by $q$ the containment in 4) and replacing $\D$ by $\D e$ one can assume that 
	\begin{enumerate}\item [4')] $\D \subseteq  \eta q \El(\Sigma)qz\eta^* $  is a finite index inclusion of non-amenable II$_1$ factors. \end{enumerate} Moreover, replacing $w$ by $ew$  and  $\psi(x)$ by $\psi(x) e$ in the intertwining in  5) still holds.

	Notice that 5) implies $ww^*\in \mathcal D'\cap r\El(Q)r$, $w^*w\in a\El(Q_1)a'\cap  a\El(Q)a=\mathbb C a\otimes \El(Q_2) $. Thus there exists $0\neq b\in \mathscr P(\El(Q_2))$ such that $w^*w= a\otimes b$. Pick $c\in \mathscr U(\El(Q))$ such that $w=c(a\otimes b)$ then 5) gives that 
	\begin{equation}\label{equality1}\mathcal Dww^*=w\El(Q_1)w^*=c(a\El(Q_1)a\otimes  \mathbb C b)c^*.
	\end{equation}
	Let $\Xi= QN_{\Lambda}(\Sigma)$. Then  using \eqref{equality1} and 4')
	above we see that \begin{equation}\label{equality3}
	c(a\otimes b) \El(Q) (a\otimes b) c^*= ww^* \eta q z \mathcal {QN}_{\El(\La)}(\El(\Sigma))''qz\eta^*ww^*= ww^* \eta q z \El(\Xi)qz\eta^*ww^*
	\end{equation}
	and also 
	
	\begin{equation}\label{equality2}
	\begin{split}
	c(\mathbb Ca \otimes b\El(Q_2)b)c^*&= (c(a\El(Q_1)a\otimes \mathbb C b)c^*)' \cap c(a\otimes b) \El(Q) (a\otimes b) c^*\\
	&= (\mathcal Dww^*)'\cap ww^* \eta qz \El(\Xi) qz\eta^* ww^*\\&= ww^*(\mathcal D'\cap \eta qz \El(\Xi) qz\eta^*) ww^*.
	\end{split}
	\end{equation}
	
	\noindent Using 4') and \cite[Lemma 3.1]{Po02} we also have that 
	\begin{eqnarray}\label{finiteindex3}
	\mathcal D\vee (\eta qz \El(\Sigma)zq \eta^*)'\cap \eta qz \El(\Xi)zq \eta^* \subseteq^f \mathcal D\vee \mathcal D'\cap \eta qz \El(\Xi)zq \eta^* \subseteq \eta qz \El(\Xi)zq \eta^*,
	\end{eqnarray}
	where the symbol $\subseteq^f$ above means inclusion of finite index.
	\vskip 0.03in
	\noindent Relation \eqref{equality1} also shows that
	\begin{equation}\label{inclusion1}
	\begin{split}
	\mathcal D\vee (\eta qz \El(\Sigma)zq \eta^*)'\cap \eta qz \El(\Xi)zq \eta^* &\subseteq^f \eta qz \El(\Sigma)zq \eta^*\vee (\eta qz \El(\Sigma)zq \eta^*)'\cap \eta qz \El(\Xi)zq \eta^* \\
	&\subseteq \eta qz \El(\Sigma (vC_\La(\Sigma)))zq \eta^*\\
	&\subseteq \eta qz \El(\Xi)zq \eta^*.
	\end{split}
	\end{equation}
	
	\noindent Here $vC_\La(\Sigma)=\{\lambda \in \Lambda \,:\, |\lambda^{\Sigma}|<\infty \}$ is the virtual centralizer of $\Sigma$ in $\La$.
	
	\noindent Let $\Phi= QN^{(1)}_{\Lambda}(\Xi)$. Using \eqref{equality3}  and the fact that $Q$ is malnormal in $\Gamma$ then the same argument  from \cite[Claim 5.2, page 26, lines 1-10]{CU18} shows that $\Xi\leqslant \Phi$ has finite index.
	
	\noindent Combining \eqref{equality2}, \eqref{equality1}  \eqref{equality3} we notice that
	\begin{equation}\label{equality4'}
	ww^* (\mathcal D\vee \mathcal D'\cap \eta qz \El(\Xi)zq \eta^*) ww^*=ww^*\eta qz \El(\Xi)zq \eta^* ww^*=ww^*\eta qz \El(\Phi)zq \eta^* ww^*.
	\end{equation}
	In particular, \eqref{equality4'} shows that $\eta qz \El(\Xi)zq \eta^*\prec_{\eta qz \El(\Xi)zq \eta^*} \mathcal D\vee \mathcal D'\cap \eta qz \El(\Xi)zq \eta^*$ and using the finite index condition in  \eqref{finiteindex3} we get $\eta qz \El(\Xi)zq \eta^*\prec_{\eta qz \El(\Xi)zq \eta^*} \mathcal D\vee (\eta qz \El(\Sigma)zq \eta^*)'\cap \eta qz \El(\Xi)zq \eta^*$. Thus, by \eqref{inclusion1} we further have $\eta qz \El(\Xi)zq \eta^*\prec_{\eta qz \El(\Xi)zq \eta^*} \eta qz \El(\Sigma( vC_\Lambda(\Sigma))) zq \eta^*$ and since $\Sigma (vC_\Lambda(\Sigma))\leqslant \Phi$ and $[\Phi:\Xi]<\infty$ then using \cite[Lemma 2.6]{CI17} we get that $[\Phi: \Sigma (vC_\Lambda(\Sigma))]<\infty$.
	
	Relation \eqref{equality3} also shows that \begin{equation}\label{equality4}c(a\otimes b) \El (Q)(a\otimes b) c^*= ww^* \eta qz \El(\Xi)zq \eta^* ww^*=ww^* \eta qz \El(\Phi)zq \eta^* ww^*.\end{equation} 
	
	As $Q$ has property (T) then by \cite[Lemma 2.13]{CI17} so is $\Phi$ and $\Xi$ and hence $\Sigma vC_\Lambda(\Sigma)$ as well. Let $\{\mathcal O_n\}_n$ be an enumeration of all the orbits in $\La$ under conjugation by $\Sigma$. Denote by $\Omega_n:=\langle \mathcal O_1,...,\mathcal O_n\rangle$. Clearly $\Omega_n\leqslant \Omega_{n+1}$ and $\Sigma$ normalizes $\Omega_n$ for all $n$. Notice that $\Omega_n\Sigma \leqslant \Omega_{n+1}\Sigma$ for all $n$ and in fact $\Omega_n \Sigma \nearrow \Sigma (vC_\Lambda(\Sigma))$. Since $\Sigma (vC_\Lambda(\Sigma))$ has property (T) there exists $n_0$ such that $\Omega_{n_0}\Sigma =\Sigma (vC_\Lambda(\Sigma))$. In particular, there is a finite index subgroup $\Sigma'\leqslant \Sigma$ such that $[\Sigma', \Omega_{n_0}]=1$ and hence  $\Sigma', \Omega_{n_0} \leqslant^f  \Sigma (vC_\Lambda(\Sigma))\leqslant^f \Phi$ are commuting subgroups. Moreover if $r_1$ is the central support of $ww^*$ in $\eta z L(\Phi)qz \eta^*$ then by \eqref{equality4} we also have that  $\eta_0 \El(Q)\eta_0^*\supseteq \eta qz \El(\Xi)qz \eta^* r_1$ for some unitary $\eta_0$. Now since the $Q_i$'s are biexact the same argument from \cite{CdSS15} shows that the finite conjugacy radical of $\Phi$ is finite. Hence $\Phi$ is a finite-by-icc group and this canonically implies that $\Phi_1:=\Sigma'$ and $\Phi_2:=\Omega_{l_0}$ are also finite-by-icc. As $\Phi$ has property (T) then so do the $\Phi_i$'s.
	Altogether,  the above arguments and \eqref{equality4} show that there exist subgroups $\Phi_1,\Phi_2 \leqslant \Phi< \La $ satisfying the following properties: \begin{enumerate}
		\item [1)] $\Phi_1,\Phi_2$ are infinite, commuting, property (T), finite-by-icc groups;
		\item [2)] $[\Phi:\Phi_1\Phi_2]<\infty$ and $QN^{(1)}_\La(\Phi)=\Phi$;
		\item [3)] there exist $\mu \in \mathscr U(\M)$, $d\in \mathscr P(\El(\Phi))$, $h= \mu d\mu^*\in \mathscr P (\El(Q))$ such that 
		\begin{equation}\label{equality5} \mu d\El(\Phi) d\mu^* = h \El(Q)h.
		\end{equation}
	\end{enumerate}
	In the last part of the proof we show that after replacing $d$ with its central support in $\El(Q)$, all the required relations in the statement still hold. Since $\El(Q)$ is a factor then using \eqref{equality5}  one can find $\xi \in \mathscr U(\emm)$ such that $\xi \El(\Phi )t \xi^*\subseteq \El(Q)$ where $t$ is the central support of $d$ in $\El(Q)$. Hence $\xi \El(\Phi ) t \xi^*\subseteq r_2\El(Q)r_2$, where $r_2=\xi t \xi^*$.  Fix $e_o\leqslant t$ and $f_o\leqslant d$ projections in the factor $\El(\Phi)t$ such that $\tau(f_o)\geqslant\tau(e_o)$. From \eqref{equality5} we have  $\mu f_o\El(\Phi) f_o\mu^* = l \El(Q)l$ and $\xi e_o\El(\Phi )e_o \xi^*\subseteq r_o \El(Q)r_o$ where $r_o= \xi e_o\xi^{\ast}$ and $l= \mu f_o \mu^*$. Let $\xi_o\in \El(Q)$ be a unitary such that $r_o\leqslant \xi_o l\xi_o^*$. Thus $\xi e_o\El(\Phi )e_o \xi^*\subseteq r_o \El(Q)r_o \subseteq \xi_o l \El(Q)l \xi^*_o= \xi_o\mu f_o\El(\Phi) f_o\mu^*\xi_o^*$ and hence \begin{equation}\label{onesidedcommensurator1}\mu^*\xi_o^*\xi e_o\El(\Phi )e_o \subseteq  f_o\El(\Phi) f_o\mu^*\xi_o^*\xi \subset \El(\Phi)\mu^*\xi_o^*\xi.\end{equation} 
	Next let $e_o+ p_1+p_2+...+p_s=t$ where $p_i\in \El(\Phi)t$ are mutually orthogonal projection such that $e_o$ is von Neumann equivalent (in $\El(\Phi)t$) to $p_i $  for all $i\in \overline{1,s-1}$ and $p_s$ is von Neumann subequivalent to $e_o$. Now let $u_i$ be unitaries in $\El(\Phi)t$ such that $u_i p_i u_i^*= e_o$ for all $i\in \overline{1,s-1}$ and  $u_s p_su_s^*= z'_o\leqslant e_o$. Combining this with \eqref{onesidedcommensurator1} we get $\mu^*\xi_o^*\xi e_o\El(\Phi )p_i  = \mu^*\xi_o^*\xi e_o\El(\Phi )u_i^* e_ou_i=\mu^*\xi_o^*\xi e_o\El(\Phi ) e_o u_i\subseteq   \El(\Phi)\mu^*\xi_o^*\xi u_i$ for all $i\in \overline{1,s-1}$. Similarly, we get $\mu^*\xi_o^*\xi e_o\El(\Phi )p_s  = \mu^*\xi_o^*\xi e_o\El(\Phi )u_s^* z'_ou_s=\mu^*\xi_o^*\xi e_o\El(\Phi ) z'_ ou_s\subseteq \mu^*\xi_o^*\xi e_o\El(\Phi ) e_ou_s\subset  \El(\Phi)\mu^*\xi_o^*\xi u_s$. Using these relations we conclude that \begin{equation*}\begin{split}\mu^*\xi_o^*\xi e_o\El(\Phi )& =\mu^*\xi_o^*\xi e_o\El(\Phi )t=\mu^*\xi_o^*\xi e_o\El(\Phi )(e_o+\sum^s_{i=1}p_i) \\ &\subseteq \mu^*\xi_o^*\xi e_o\El(\Phi )e_o+ \sum^s_{i=1} \mu^*\xi_o^*\xi e_o\El(\Phi ) p_i\\
	& \subseteq \El(\Phi )\mu^*\xi_o^*\xi+ \sum^s_{i=1} \El(\Phi )\mu^*\xi_o^*\xi u_i.
	\end{split}\end{equation*}  
	In particular, this relation shows that $\mu^*\xi_o^*\xi e_o\in \mathcal {QN}^{(1)}_{\El(\La)}(\El(\Phi))$ and since $\mathcal {QN}^{(1)}_{\El(\La)}(\El(\Phi))''=\El(\Phi)$ by 2) then we conclude that $\mu^*\xi_o^*\xi e_o\in \El(\Phi)$. Thus using this together with \eqref{onesidedcommensurator1} one can check that
	\begin{align*}
\xi e_o \El(\Phi) e_o \xi^*&= \xi e_o\xi^*\xi_o\mu(\mu^*\xi^*_o \xi e_o \El(\Phi) e_o \xi^*\xi_o\mu) \mu^*\xi^*_o\xi e_o\xi^*\\&=\xi e\xi^*\xi_o\mu f_o\El(\Phi)f_o \mu^* \xi^*_o \xi e \xi^*\\&=\xi e_o \xi^*\xi_o l \El(Q)l  \xi^*_o \xi e_o \xi^*= r_o \El(Q)r_o.
	\end{align*} 
	
	In conclusion we have proved that $\xi \El(\Phi) t \xi^*\subseteq r_2\El(Q)r_2$ and for all $e_o\leqslant t$ and $f_o\leqslant d$ projections in the factor $\El(\Phi)t$ such that $\tau(f_o)\geqslant\tau(e_o)$ we have $\xi e_o \El(\Phi) e_o \xi^*=r_o\El(Q)r_o$ where $r_o\leqslant r_2= \xi t\xi^*$. By Lemma \ref{equalcorners} this clearly implies that $\xi \El(\Phi)t\xi^*= r_2\El(Q)r_2$ which finishes the proof.
\end{proof}

\begin{lemma}\label{somefiniteintersection}Let $\G$ be a group as in Notation \ref{semidirectt} and assume that $\Lambda$ is a group such that $\El(\Gamma)=\El(\Lambda)=M$.  Also assume there exists a subgroup  $\Phi< \La $, a unitary $\mu \in \mathcal U(\emm)$ and projections  $z\in \mathcal Z(\El(\Phi))$, $r= \mu z\mu^*\in \El(Q)$ such that 
	\begin{equation}\label{equalcorner2} \mu \El(\Phi) z\mu^* = r \El(Q)r.
	\end{equation}
	For every $\lambda \in \La\setminus \Phi$ so that $\left |\Phi \cap \Phi^\lambda \right|=\infty$ we have $zu_\lam z=0$. In particular, there is $\lam_o\in \La\setminus \Phi$ so that $\left |\Phi \cap \Phi^{\lambda_o} \right|<\infty$.
	
\end{lemma}

\begin{proof}Notice that since $Q<\G=(N_1\times N_2)\rtimes Q$ is almost malnormal then we have the following property: for every sequence $\El(Q)\ni x_n \rar 0$ weakly and every $x,y\in M$ such that  $E_{\El(Q)}(x)=E_{\El(Q)}(y)=0$ we have \begin{equation}\label{malnormal1}
	\|E_{\El(Q)}(xx_ky)\|_2\rar 0,\text{ as } k\rar \infty.
	\end{equation}   
	Using basic approximations and the $\El(Q)$-bimodularity of the expectation we see that it suffices to check \eqref{malnormal1} only for elements of the form $x= u_{n}$ and $y= u_{m}$ where $n,m\in (N_1\times N_2)\setminus\{1\}$. Consider the Fourier decomposition $x_n = \sum_{h\in Q}\tau(x_k u_{h^{-1}}) u_h$ and notice that
	\begin{equation}\begin{split}
	\|E_{\El(Q)}(xx_k y)\|_2^2&=\| \sum_{h\in Q} \tau(x_k u_{h^{-1}}) \delta_{n h m, Q} u_{nhm}\|_2^2\\&=\| \sum_{h\in Q} \tau(x_k u_{h^{-1}}) \delta_{n\sigma_h(m) h, Q} u_{n\sigma_h(m)h}\|_2^2=\sum_{h\in Q, \sigma_h(m)=n^{-1}} |\tau(x_k u_{h^{-1}})|^2.\end{split}
	\end{equation} 
	Since the action $Q\ca N_i$ has finite stabilizers one can easily see that the set $\{h\in Q \,:\, \sigma_h(m)=n^{-1}\}$ is finite and since $x_n \rar 0$ weakly then $\sum_{h\in Q, \sigma_h(m)=n^{-1}} |\tau(x_k u_{h^{-1}})|^2\rar 0$ as $k\rar \infty$  which concludes the proof of \eqref{malnormal1}. 
	Using the conditional expectation formula for compression  we see that \eqref{malnormal1} implies that for every sequence $\El(Q)\ni x_n \rar 0$ weakly and every $x,y\in r\emm r$ so that  $E_{r\El(Q)r}(x)=E_{r\El(Q)r}(y)=0$ we have $\|E_{r\El(Q)r}(xx_ky)\|_2\rar 0$, as $k\rar \infty$. Thus using the formula \ref{equalcorner2} we get that for all $ \mu \El(\Phi)z \mu^* \ni x_n \rar 0$ weakly and every $x,y\in \mu z \emm z \mu^*$ so that  $E_{\mu \El(\Phi)z \mu^*}(x)=E_{\mu \El(\Phi)z \mu^*}(y)=0$ we have $\|E_{\mu \El(\Phi)z \mu^*}(xx_ky)\|_2\rar 0$, as $k\rar \infty$. This entails that for all $ \El(\Phi)z  \ni x_n \rar 0$ weakly and every $x,y\in z \emm z $ satisfying  $E_{\El(\Phi)z }(x)=E_{ \El(\Phi)z }(y)=0$ we have 
	\begin{equation}\label{malnormal2}\|E_{\El(\Phi)z }(xx_ky)\|_2\rar 0\text{, as }k\rar \infty.\end{equation} 
	Fix $\lam\in \La\setminus \Phi$ so that $|\Phi\cap \Phi^\lambda|=\infty$. Hence there are infinite sequences $\lam_k, \omega_n\in \La$ so that $\lam \omega_k\lam^{-1}=\lam_k$ for all integers $k$. Since $\lambda\in \La\setminus \Phi$ then $E_{\El(\Phi)}(u_\lambda z)= E_{\El(\Phi)z}(zu_{\lambda^{-1}})=0$. Also we have that $u_{\omega_k} z\rar 0$ weakly as $k\rar \infty$. Using these calculations we have that
	\begin{equation}\label{malnormal3}\begin{split}
	\|E_{\El(\Phi)} (z u_\lam z u_{\lam^{-1}} z )\|_2^2&= \| E_{\El(\Phi) }( u_\lam z u_{\lam^{-1}} z )\|^2_2=\| u_{\lam \omega_k \lam^{-1}} E_{\El(\Phi) }( u_\lam z u_{\lam^{-1}} z )  \|^2_2\\
	=& \| E_{\El(\Phi) }( u_{\lam \omega_k}  z u_{\lam^{-1}} z )  \|^2_2= \| E_{\El(\Phi) z}( z u_{\lam} z u_{\omega_k}  z u_{\lam^{-1}} z )  \|^2_2\rar 0\text{ as }k\rar \infty.\end{split}
	\end{equation}  
	Also using \eqref{malnormal3} the last quantity above converges to $0$ as $k\rar \infty$ and hence $E_{\El(\Phi)} (z u_\lam z u_{\lam^{-1}} z )=0$ which entails that $zu_\lam z=0$, as desired. For the remaining part notice first that since $[\G:Q]=\infty$ then \eqref{equalcorner2} implies that $[\La: \Phi]=\infty$. Assume by contradiction that for all $\lam \in \La\setminus \Phi$ we have $zu_\lam z=0$. As $[\La: \Phi]=\infty$ then for every positive integer $l$ one can construct inductively $\lam_i\in \La \setminus \Phi$ with  $i\in \overline{1, l}$ such that $\lam_i \lam_j^{-1}\in  \La \setminus \Phi$ for all $i>j$ such that $i,j\in \overline{1,l}$. But this implies that $0= z u_{\lam_i\lam_j^{-1}} z= z u_{\lam_i}u_{\lam_j^{-1}} z$  and hence $u_{\lam_i^{-1}} z u_{\lam_i}$ are mutually orthogonal projections when $i=\overline{1,l}$. This is obviously false when $l$ sufficiently large. \end{proof}

\begin{theorem}\label{parabolicQ} Assume the same conditions as in Theorem \ref{toproductgroupcorners}. Then one can find subgroups $\Phi_1,\Phi_2 \leqslant \Phi\leqslant \La $ so that \begin{enumerate}
		\item $\Phi_1,\Phi_2$ are infinite, icc, property (T) groups so that $\Phi=\Phi_1\times \Phi_2$;
		\item $\mathcal{QN}^{(1)}_\La(\Phi)=\Phi$;
		\item There exists $\mu \in \mathcal U(\emm)$ such that  $\mu \El(\Phi) \mu^* =  \El(Q)$.
	\end{enumerate}

\end{theorem}
\begin{proof} From Theorem \ref{toproductgroupcorners} there exist subgroups $\Phi_1,\Phi_2 \leqslant \Phi\leqslant \La $ such that \begin{enumerate}
		\item $\Phi_1,\Phi_2$ are, infinite, commuting, finite-by-icc, property (T) groups so that $[\Phi:\Phi_1\Phi_2]<\infty$;
		\item $\mathcal{QN}^{(1)}_\La(\Phi)=\Phi$;
		\item There exist $\mu \in \mathcal U(\emm)$ and  $z\in \mathcal P(\mathcal Z(\El(\Phi)))$ with  $h= \mu z\mu^*\in \mathcal P (\El(Q))$ satisfying 
		\begin{equation}\label{maxcorner1} \mu \El(\Phi) z\mu^* = h \El(Q)h.
		\end{equation}
	\end{enumerate}  
	
	Next we show that in \eqref{maxcorner1} we can pick $z\in \mathcal Z(\El(\Phi))$ maximal with the property that for every projection $t \in \mathcal Z(\El(\Phi)z^{\perp})$ we have \begin{equation}\label{nonitertwiningcomplement}L(\Phi_i)t \nprec_\emm \El(Q)\text{ for }i=1,2.\end{equation}

	To see this let $z\in \mathcal F$ be a maximal family of mutually orthogonal (minimal) projections $z_i \in \mathcal Z(\El(\Phi))$ such that $\El (\Phi)z_i\prec_\emm \El(Q)$. Note that since $\Phi$ has finite conjugacy radical it follows that $\mathcal F$ is actually finite. Next let $z\leqslant \sum{z_i}:=a \in \mathcal Z(\El(\Phi))$ and we briefly argue that $\El(\Phi)a\prec^s_\emm \El(Q)$. Indeed since $(\El(\Phi)a)'\cap a\emm a= a(\El(\Phi)'\cap \emm) a= \mathcal Z(\El(\Phi))a$ and the latter is finite dimensional then for every $r\in (\El(\Phi)a)'\cap a\emm a$ there is $z_i \in \mathcal F$ such that $rz_i=z_i\neq 0$. Since $\El (\Phi)z_i\prec_\emm \El(Q)$ and then $\El (\Phi)r\prec_\emm \El(Q)$ as desired. Thus applying Lemma \ref{maxcorner2}, after perturbing $\mu$ to a new unitary  we get $\mu \El(\Phi) a\mu^* = h_o \El(Q)h_o.$  Finally, we show \eqref{nonitertwiningcomplement}. Assume by contradiction there is $t_o\in \mathcal Z(\El(\Phi)z^{\perp})$ so that $\El(\Phi_i)t_o \prec_\emm \El(Q)$ for some $i=1,2$. Thus there exist projections $r\in \El(\Phi)t_o$, $q\in \El(Q)$, a partial isometry  $w\in \emm$ and a $\ast$-isomorphism on the image $\phi: r \El(\Phi)r \rar \mathcal B:=\phi(r\El(\Phi)r)\subseteq q\El(Q)q$ such that $\phi(x)w=wx$. Notice that $w^*w \in t_o (\El(\Phi_i)'\cap \emm)t_o$ and $ww^*\in \mathcal B'\cap q\emm q$. But since $Q<\G $ is malnormal it follows that  $\mathcal B'\cap q\emm q\subseteq q\El(Q)q$ and hence  $ww^*\in q\El(Q)q$. Using this in combination with previous relations we get that $wr\El(\Phi_i)rw^*=\mathcal Bww^*\subseteq \El(Q)$ and extending $w$ to a unitary $u$ we have that $u r \El(\Phi_i)ru^*\subseteq \El(Q)$. Since $\El(Q)$ is a factor we can further perturb the unitary $u$ so that $u \El(\Phi_i)r_ou^*\subseteq \El(Q)$ where $r\leqslant r_o \leqslant t_o$ is the central support of $r$ in $\El(\Phi_i)t_o$. Using malnormality of $Q$ again we further get $ r_o(\El(\Phi_i)\vee \El(\Phi_i)'\cap \emm)r_ou^*\subseteq \El(Q)$ and perturbing $u$ we can further assume that $ (\El(\Phi_i)\vee \El(\Phi_i)'\cap \emm) s_o u^*\subseteq \El(Q)$ where $r_o\leqslant s_o$ is the central support of $r_o$ in $\El(\Phi_i)\vee \El(\Phi_i)'\cap \emm)$. In particular, $u(\El(\Phi) s_o u^*\subseteq \El(Q)$ and hence $\El(\Phi) s_o\subseteq u^*\El(Q)u$. Since $r\leqslant r_o\leqslant s_o$ and $r\leqslant t_o$ the previous containment implies that there is a minimal projection $s'\in \El(\Phi)a^\perp$ so that $\El(\Phi)s'\prec \El(Q)$ which contradicts the maximality assumption on $\mathcal F$. Finally replacing $z$ with $a$ in our statement, etc our claim follows.
\vskip 0.05in
Next fix $t\in \mathcal Z(\El(\Phi)z^{\perp})$. Since $\El(\Phi_1)t$ and $\El(\Phi_2)t$ are commuting property (T) von Neumann algebras then using the same arguments as in the first part of the proof of Theorem \ref{commutationcontrolincommultiplication} there are two possibilities: either i) there exists $j\in 1,2$ such that  $\El(\Phi_j)t\prec_\emm \El(N_2)$ or ii) $\El (\Phi)t \prec_\emm \El(N_2\rtimes Q)$. Next we briefly argue ii) is impossible. Indeed, assuming ii), Theorem \ref{controlprodpropt1} for $n=1$ would imply the existence of $j \in 1,2$ so that $\El(\Phi_j)t \prec_\emm \El(Q)$ which obviously contradicts the choice of $z$. Thus we have i) and passing to the relative commutants we have that $\El(N_1)\prec \El(\Phi_j)t'\cap t\emm t= t (\El(\Phi_j )'\cap \emm) t$. Using the relationships between the $\Phi_j$'s we see that $t (\El(\Phi_j )'\cap \emm) t\subset t\El (\Phi_j )\vee \El(\Phi_j)'\cap \emm) t\subseteq t \El(\Phi_j (vC_\La(\Phi_j)))t\subseteq t \El(\Phi)t$. In conclusion, we have 
	\begin{equation}\label{nonintertiningcomplement2}\El(N_1)\prec_\emm  t \El(\Phi) t,\text{ for all }t \in \mathcal Z(\El(\Phi)z^{\perp}).
	\end{equation}
	Let $A=\{ \lam \in\La \,:\, |\Phi\cap \Phi^\lam|<\infty\}$ and $B=\{ \lam \in\La \,:\, |\Phi\cap \Phi^\lam|=\infty\}$. Note that $A\cup B=\La$ and $A\neq \emptyset$. Since $N_1$ is infinite then for every $\lam\in A$ we have that $\El(N_1)\nprec_\emm \El(\Phi\cap \Phi^\lam) z^\perp$. Thus using \eqref{nonintertiningcomplement2} together with the same argument from the proof of \cite[Theorem 6.16]{PV06}, working under $z^\perp$, we get $z^\perp E_{\El(\Phi)}(u_\lam z^\perp xz^\perp )=0$ for all $x\in \emm$. This further implies that $z^\perp u_\lam z^\perp=0$ for all $\lam\in A$ and hence $ u_\lam z^\perp u_{\lam^{-1}}\leqslant z$. 
	
	On the other hand by Lemma \ref{somefiniteintersection} for all $\lam\in B$ we get $z u_\lam z=0$ and hence $u_\lam z u_{\lam^{-1}}\leqslant z^\perp$. So if $B\neq\emptyset$ we obviously have equality in the previous two relations, i.e.\ $u_\lam z u_{\lam^{-1}}= z^\perp$ for all $\lam\in B$ and $u_\lam z^\perp u_{\lam^{-1}}= z$ for all $\lam\in A$. These further imply there exist $a_o\in A$ and $b_0\in B$ such that $A=a_0 C_\La(z^\perp)$ and $B= b_oC_\La(z)$; here $C_\La(z)\leqslant \La$ is the subgroup of all elements of $\La$ that commute with $z$ and similarly for $C_\La(z^\perp)$. Thus $\La= A\cup B = a_o C_\La(z^\perp)\cup b_o C_\La(z)$. Thus we can assume, without loss of generality, that $[\La:C_\La(z)]<\infty$. But since $\La$ is icc this implies that $z=1$. The rest of the statement follows.   \end{proof}

\begin{theorem} In the Theorem \ref{commutationcontrolincommultiplication} we cannot have case 4a).\end{theorem}

\begin{proof} Assume by contradiction that for all $j\in 1,2$ there is $i\in 1,2$ such that $\Delta (\El(Q_i))\prec_{\emm\bar\otimes \emm} \emm\bar\otimes \El(N_j)$. Using \cite[Theorem 4.1]{DHI16} and the property (T) on $N_j$ one can find a subgroup $\Sg<\La$ such that $\El(Q_i)\prec_\emm \El(\Sg)$ and $\El(N_j)\prec_\emm \El(C_\La(\Sg))$. Since $\mu \El(\Phi)\mu^*=\El(Q)$ and $Q_i$ are biexact then by the product rigidity results in \cite{CdSS15} one can assume that there is a unitary $u\in \El(Q)$ such that $u\El(Q_1)u^*=\El(\Phi_1)^t$ and $u\El(Q_2)u^*=\El(\Phi_2)^{1/t}$. Thus we get that $\El(\Phi_i)\prec
_\emm \El(\Sg)$  and hence $[\Phi_i: g\Sg g^{-1} \cap \Phi_i]<\infty$. So working with $g\Sg g^{-1}$ instead of $\Sg$ we can assume that  $[\Phi_i: \Sg  \cap \Phi_i]<\infty$. In particular $\Sg\cap \Phi_i$ is infinite and since $\Phi$ is almost malnormal in $\La$ it follows that $C_\La(\Sg\cap \Phi_i) < \Phi$. Thus we have that  $\El(N_j)\prec_\emm \El(C_\La(\Sg))\subseteq \El(C_\La(\Sg\cap \Phi_i))\subset \El(\Phi)=\mu^*\El(Q)\mu$ which is obviously a contradiction. \end{proof}

\begin{theorem}\label{commutationcontrolincommultiplication2} Let $\G$ be a group as in Notation \ref{semidirectt} and assume that $\Lambda$ is a group such that $\El(\Gamma)=\El(\Lambda)=\emm$. Let $\Delta: \emm \rar \emm \bar \otimes \emm$ be the comultiplication along $\Lambda$ as in Notation \ref{semidirectt}. Then the following hold:
	\begin{enumerate}
		\item [i)]  $\Delta (\El(N_1)), \Delta (\El(N_2)), \Delta(\El(N_1\times N_2))\prec^s_{\emm \bar\otimes \emm} \El(N_1\times N_2)\bar\otimes \El(N_1\times N_2)$, and
		\item [ii)] there is a unitary $u\in \emm \bten \emm$ such that $u\Delta(\El(Q))u^*\subseteq  \El(Q)\bar\otimes \El(Q)$.    
	\end{enumerate}
	
\end{theorem}

\begin{proof} First we show i). From Theorem \ref{commutationcontrolincommultiplication} we have that for all $j\in 1,2$ there is $j_i\in 1,2$ such that $\Delta(\El(N_{j_i}))\prec_{\emm \bar\otimes \emm} \emm\bar\otimes \El(N_j)$. Notice that since $\mathcal N_{\emm \bten \emm}\De(\El(N_i))''\supset \De(\emm)$ and $\De(\emm)'\cap \emm \bten \emm =\mathbb C 1$ then by \cite[Lemma 2.4 part (3)]{DHI16} we actually have  $\Delta(\El(N_{j_i}))\prec^s_{\emm\bar\otimes \emm} \emm\bar\otimes \El(N_j)$. Notice that for all $i\neq k$ we have $j_i\neq j_k$. Otherwise we would have $\De (\El(N_{j_i}))\prec^s_{\emm\bten\emm} \emm \bar\otimes \El(N_1)$ and $\De (\El(N_{j_i}))\prec^s_{\emm\bten\emm} \emm \bar\otimes \El(N_2)$ which by \cite[Lemma 2.8 (2)]{DHI16} would imply that $\De (\El(N_{j_i}))\prec^s_{\emm\bten\emm} \emm \bar\otimes \El(N_1\cap N_2)=\emm \otimes 1$ which is a contradiction. Furthermore using the same arguments as in \cite[Lemma 2.6]{Is16} we have that $\Delta(\El(N_1\times N_2))\prec^s_{\emm \bar\otimes \emm} \emm \bar\otimes \El(N_1\times N_2)$. Then working on the left side of the tensor we get that $\Delta(\El(N_1\times N_2))\prec^s_{\emm \bar\otimes \emm} \El(N_1\times N_2)\bar\otimes \El(N_1\times N_2)$.

	\vskip 0.05in
	
Finally, notice that part ii) is a direct consequence of Theorem \ref{parabolicQ}. 
\end{proof}

\subsection{Proof of Theorem \ref{semidirectprodreconstruction}} 

\begin{proof} We divide the proof into separate parts to improve the exposition. 
	\subsection*{Reconstruction of the Acting Group $Q$}
	To accomplish this we will use the notion of height for elements in group von Neumann algebras as introduced in \cite{IPV10,Io11}). From the previous theorem recall that $u\Delta(\El(Q))u^*\subseteq  \El(Q)\bar\otimes \El(Q)$. Let $\mathcal A= u \De(\El(N_1))u^*$. Next we claim that \begin{equation}\label{height3}
	h_{Q\times Q}(u\De(Q)u^*)>0.
	\end{equation}  
	For every $x,y \in \El(Q)\bten \El(Q)$ and every $a\in \mathcal A\bten \mathcal A$ supported on a finite set $F\subset N=N_1\times N_2$ we have that
	\begin{equation}\begin{split}\label{height1}
	\|E_{\mathcal A\bten \mathcal A}(xay)\|_2^2&= \|\sum_{q,l} \tau(x u_{q^{-1}})\tau(yu_l) E_{\mathcal A\bten \mathcal A}(u_qau_{l^{-1}}) \|_2^2\\
	&= \|\sum_{q,l} \tau(x u_{q^{-1}})\tau(yu_l) E_{\mathcal A\bten \mathcal A}(\sigma_q(a)u_{ql^{-1}}) \|_2^2\\
	&= \|\sum_{q} \tau(x u_{q^{-1}})\tau(yu_l) \sigma_q(a) \|_2^2\\
	&= \|\sum_{q\in Q,n\in N^2} \tau(x u_{q^{-1}})\tau(yu_l) \tau(au_{n^{-1}})u_{\sigma_q(n)} \|_2^2\\
	&= \sum_{r\in N^2} |\sum_{\sigma_q(n)=r}\tau(x u_{q^{-1}})\tau(yu_l) \tau(au_{n^{-1}})|^2\\
	&\leqslant h_{Q\times Q}^2(x)\sum_{r\in N^2} (\sum_{q\in Q \,:\,\sigma_{q^{-1}}(r^{-1})\in F}|\tau(yu_l)|| \tau(au_{\sigma_{q^{-1}}(r)}|)^2\\
	&\leqslant h_{Q\times Q}^2(x) \|y\|_2^2 \|a\|_2^2\max_{r\in N^2} |\{ q\in Q  \,:\,\sigma_{q^{-1}}(r^{-1})\in F\}|.
	\end{split}\end{equation}
	
	This estimate leads to the following property: for every finite sets  $K,S \subset Q$, every $a\in span \{\mathcal A\bten \mathcal A u_g \,:\,g\in K\}$ and all $\varepsilon>0$ there exist a scalar $C>0$ and a finite set $F \subset N^2$ such that for all $x,y \in \El(Q)\bten \El(Q)$ we have \begin{equation}\label{height2}
	\|P_{\sum_{s\in S} \mathcal A\bten \mathcal A u_s }(xay)\|_2^2\leqslant |K||S| C (h_{Q\times Q}^2(x) \|y\|_2^2 \|a\|_2^2\max_{r\in N^2} |\{ q\in Q  \,:\,\sigma_{q^{-1}}(r^{-1})\in F\}|)+ \varepsilon \|x\|_{\infty}
	\|y\|_\infty\end{equation}   
	
	Note this follows directly from  \eqref{height1} after we decompose the $a$ and the projection $P_{\sum_{s\in S} \mathcal A\bten \mathcal A u_s}$. 
	
	\vskip 0.05in 
	Next we use \eqref{height2} to prove our claim. Fix $\varepsilon>0$.  Since $\De(\mathcal A)\nprec \emm\otimes 1, \ 1\otimes \emm$ then by Theorem \ref{corner} one can find a finite subset $F_o \subset N^2\setminus ((N\times 1)\cup (1\times N))$ such that $a_{F_o} \in \mathcal A \bten \mathcal A$ is supported on $F_o$ and $\|a-a_{F_o}\|_2\leqslant \varepsilon$. Since  $\De(\mathcal A)\prec^s \mathcal A\bten \mathcal A$ there is a finite $S\subseteq Q\times Q$  such that
	\begin{equation}
	\|P_{\sum_{s\in S} \mathcal A\bten \mathcal A u_s }(a)-a\|_2\leqslant \varepsilon\text{ for all }a\in \De(\mathcal A).
	\end{equation} 
	
	Assume by contradiction \eqref{height3} doesn't hold. Thus there is a sequence $t_n\in Q$ such that $h_{Q\times Q}(t_n)=h_{Q\times Q}(u\De(u_{t_n})u^*)\rar 0$ as $n\rar \infty$. As $t_n$ normalizes $\De(\mathcal A)$ then one can see that 
	\begin{equation}
	\begin{split}1-\varepsilon=\| t_n a t_n^*\|_2^2-\varepsilon &\leqslant \|P_{\sum_{s\in S} \mathcal A\bten \mathcal A u_s }(t_na t_n^*)\|_2^2 \\ 
	&\leqslant  \| |P_{\sum_{s\in S} \mathcal A\bten \mathcal A u_s }(t_na t_n^*)\|^2_2+\varepsilon \\
	&\leqslant |F_o||S| C (h_{Q\times Q}^2(t_n) \|t_n\|_2^2 \|a_{F_o}\|_2^2\max_{r\in N^2} |\{ q\in Q  \,:\,\sigma_{q^{-1}}(r^{-1})\in F_o\}|)+ \varepsilon \|t_n\|^2_{\infty} \\&
	\leqslant |F_o||S| C (h_{Q\times Q}^2(t_n) \max_{r\neq 1} |Stab_Q(r)||F_o|)+ 2\varepsilon.
	\end{split}
	\end{equation} 
	
	Since the stabilizer sizes are uniformly bounded we get a contradiction if $\varepsilon>0$ is arbitrary small. 
	Now we notice that the height condition together with Theorem \ref{parabolicQ} and \cite[Lemmas 2.4,2.5]{CU18} already imply that $h_Q(\mu \Phi \mu^* )>0$ and by \cite[Theorem 3.1]{IPV10} there is a unitary $\mu_0\in \emm$ 
	such that $\mathbb T \mu_0 \Phi \mu_0^*=\mathbb T Q$.  
	
	\subsection*{Reconstruction of a Core Subgroup and its Product Feature}
	From Theorem \ref{commutationcontrolincommultiplication2} have that $\Delta(\El(N_1 \times N_2)) \prec^s_{\emm \tp \emm} \El(N_1 \times N_2) \tp \El(N_1 \times N_2)$. Proceeding exactly as in the proof of \cite[Claim 4.5]{CU18} we can show that $\Delta(\mathcal A) \subseteq \mathcal A \tp \mathcal A$, where $\mathcal A= u\El(N_1 \times N_2)u^{\ast}$. By Lemma \ref{comultsubgp}, there exists a subgroup $\Sg < \La$ such that $\mathcal A= \El(\Sg)$. The last part of the proof of \cite[Theorem 5.2]{CU18} shows that $\La= \Sg \rtimes \Phi$. In order to reconstruct the product feature of $\Sigma$, we need a couple more results.
	
	\begin{claim} For every $i=1,2$ there exists $j=1,2$ such that \begin{equation}\label{intonetensor}
		\Delta(\El(N_j))\prec^s \El (N_1\times N_2) \bten \El
		(N_i).
		\end{equation}
	\end{claim} 
	
	\noindent \textit{Proof of Claim.} We prove this only for $i=1$ as the other case is similar. We also notice that since $\mathcal N_{\emm\otimes \emm} (\Delta(\El(N_j)))''\supseteq \Delta (\emm)$ and $\Delta(\emm)'\cap \emm\bten \emm =\mathbb C1$ then to establish \eqref{intonetensor} we only need to show that $\Delta(\El(N_j))\prec \El (N_1\times N_2) \bten \El(N_i)$. From above we have $\Delta(\El(N_1\times N_2)\prec_{\emm\bar\otimes \emm} \El(N_1\times N_2)\bar\otimes \El(N_1\times N_2)$. Hence there exist nonzero projections $a_i\in \Delta (\El(N_i))$ and  $b\in  \El(N_1\times N_2)\bar\otimes \El(N_1\times N_2)$, a partial isometry $v\in \emm \bten \emm$ and an $\ast$-isomorphism on the image $\Psi: a_1\otimes a_2 \Delta(\El(N_1\times N_2)) a_1\otimes a_2 \rightarrow \Psi(a_1\otimes a_2 \Delta(\El(N_1\times N_2)) a_1\otimes a_2) := \mathcal R \subseteq b(\El(N_1\times N_2)\bar\otimes \El(N_1\times N_2)) b$ such that $\Psi(x)v=vx$ for all $x\in a_1\otimes a_2 \Delta(\El(N_1\times N_2)) a_1\otimes a_2$.   
	\vskip 0.04in 
	Denote by $\mathcal D_i:= \Psi (a_i (\Delta (\El(N_i)))a_i)\subseteq b\El(N_1\times N_2)\bar\otimes \El(N_1\times N_2) b$ and notice that $\mathcal D_1$ and $\mathcal D_2$ are commuting property (T) diffuse subfactors. Since the group $N_2$ is ($\mathbb F_\infty$)-by-(non-elementary hyperbolic group) then by \cite{CIK13,CK15} it follows that there is $j=1,2$ such that $\mathcal D_j \prec_{\mathcal \El(N_1\times N_2)\bten \El(N_1\times N_2)} \El(N_1\times N_2)\bten \El(N_1 \times \mathbb F_\infty)$. Since $\mathbb F_\infty$ has Haagerup's property and $\mathcal D_j$ has property (T) this further implies that $\mathcal D_j \prec_{\mathcal \El(N_1\times N_2)\bten \El(N_1\times N_2)} \El(N_1\times N_2)\bten \El(N_1)$. Composing this intertwining with $\Psi$ we get $\Delta(\El(N_j))\prec \El (N_1\times N_2) \bten \El(N_1)$, as desired.
	\vskip 0.06in
	Also, we note that $j_1 \neq j_2$. Otherwise we would have that $\Delta(\El(N_j)) \prec^s \El(N_1 \times N_2) \bten \El(N_1) \cap \El (N_2)= \El(N_1 \times N_2) \bten 1$, which obviously contradicts \cite[Proposition 7.2.1]{IPV10}.
	$\hfill\blacksquare$
	\vskip 0.07 in
	Let $\mathcal A= u\El(N_1))u^{\ast}$. Thus, we get that $\Delta(\mathcal A) \prec^s \El(N_1 \times N_2) \otimes \El(N_i)$ for some $i=1,2$. This implies that for every $\varepsilon >0$, there exists a finite set $S \subset u^{\ast} Q u$, containing $e$, such that $\|d-P_{S \times S}(d)\|_2 \leq \varepsilon$ for all $d \in \Delta (\mathcal A)$. However, $\Delta (\mathcal A)$ is invariant under the action of $u^{\ast}Q u$, and hence arguing exactly as in \cite[Claim 4.5]{CU18} we get that $\Delta(\mathcal A) \subset (\El(\Sg) \bten u\El(N_i)u^{\ast})$. We now separate the argument into two different cases:
	
	\textbf{Case I:} $i=1$.
	
	In this case, $\Delta(\mathcal A) \subseteq \El(\Sg) \tp \mathcal A$. Thus by Lemma \ref{comultsubgp} we get that there exists a subgroup $\Sg_0 < \Sg$ with $\mathcal A = \El(\Sg_0)$. Now, $\mathcal A' \cap \El(\Sg)= u\El(N_2)u^{\ast}$. Thus, $\El(\Sg_0)' \cap \El(\Sg) = u\El(N_2)u^{\ast}$. Note that $\Sg$ and $\Sg_0$ are both icc property (T) groups. This implies that $\El(\Sg_0)' \cap \El(\Sg)= \El(vC_{\Sg}(\Sg_0))$, where $vC_{\Sg}(\Sg_0)$ denotes the \textit{virtual centralizer} of $\Sg_0$ in $\Sg$. Proceeding as in \cite{CdSS17} we can show that $\Sg= \Sg_0 \times \Sg_1$.
	
	\textbf{Case II:} $i=2$.
	
	Let $\mathcal B= u\El(N_2)u^{\ast}$. In this case, $\Delta(\mathcal A) \subseteq \El(\Sg) \tp \mathcal B$. However, Lemma \ref{comultsubgp} then implies that $\mathcal A \subseteq \mathcal B$, which is absurd, as $\El(N_1)$ and $\El(N_2)$ are orthogonal algebras. Hence this case is impossible and we are done. \end{proof}



\noindent {\bf Remarks.} $1)$  There are several immediate consequences of the Theorem \ref{semidirectprodreconstruction}. For instance one can easily see the von Neumann algebras covered by this theorem are non-isomorphic with the ones arising from any irreducible lattice in higher rank Lie group. Indeed,  if $\La$ is any such lattice satisfying  $\El(\G) \cong \El(\La)$, then Theorem ~\ref{semidirectprodreconstruction} would imply that $\La$  must contain an infinite normal subgroup of infinite index which contradicts Margulis' normal subgroup theorem.
\vskip 0.06in
$2)$ While it well known there are uncountable many non-isomorphic group II$_1$ factors with property (T) \cite{Po07} little is known about producing concrete examples of such families. In fact the only currently known infinite families of pairwise non-isomorphic property (T) groups factors are $\{\mathcal L(G_n)\,|\, n\geq 2\}$ for $G_n$ uniform latices in $Sp(n,1)$ \cite{CH89} and $\{\mathcal L(G_1\times G_2\times \cdots\times G_k)\,|\, k\geq 1\}$ where $G_k$ is any icc property (T) hyperbolic group \cite{OP03}. Theorem \ref{semidirectprodreconstruction} makes new progress in this direction by providing a new explicit infinite family of icc property (T) groups which gives rise to pairwise non-isomorphic II$_1$ factors. For instance, in the statement one can simply $Q_i$ to vary in any infinite family of non-isomorphic uniform lattices in $Sp(n,1)$ for any $n\neq 2$. Unlike the other families ours consists of factors which are not solid, do not admit tensor decompositions \cite{CdSS17}, and do not have Cartan subalgebras, \cite{CIK13}.

\vskip 0.06in

$3)$ We notice that Theorem \ref{semidirectprodreconstruction} still holds if instead of $\G= (N_1\times N_2)\rtimes (Q_1\times Q_2)$ one considers any finite index subgroup of $\G$ of the form $\G_{s,r}= (N_1\times N_2)\rtimes (Q^{s}_1\times Q^{r}_2)\leqslant \G$, where $Q^{s}_1\leqslant Q_1$ and $Q^{r}_2\leqslant Q_2$ are arbitrary finite index subgroups. One can verify these groups still enjoy all the algebraic/geometric properties used in the proof of Theorem \ref{semidirectprodreconstruction} (including the fact that $N_1\rtimes Q^s_1$ is hyperbolic relative to $Q^s_1$
 and $N_1\rtimes Q^r_2$ is hyperbolic relative to $Q^r_2$) and hence all the von Neumann algebraic arguments in the proof of Theorem \ref{semidirectprodreconstruction}  apply verbatim. The details are left to the reader. 

\vskip 0.06in 

$4)$ The group factors considered in Theorem \ref{semidirectprodreconstruction} have trivial fundamental group by \cite[Theorem B]{CDHK20}

\section{Concrete Examples of Infinitely Many Pairwise Non-isomorphic Group II$_1$ Factors with Property (T)}

In this section we present several applications of our main techniques to the structural study of property (T) group factors. An earlier result of Popa  \cite{Po07} shows that the map $\G \mapsto \El(\G)$ is at most countable to one. Since there are uncountably many icc property (T) groups, this obviously implies the existence of uncountably many group property (T) factors which are pairwise non-isomorphic. However, currently there are still no explicit constructions of such families in the literature. In this section we make new progress in this direction by showing that the canonical fiber product of Belegradek-Osin Rips construction groups can be successfully used to provide possibly the first such examples (Corollary~\ref{unctlblet}). In addition, our methods also yields other interesting consequences. For instance, they can be used to provide an infinite series of finite index subfactors of a given property (T) II$_1$ factor that are pairwise non-isomorphic which is also a novelty in the area (Corollary \ref{inffindext}). This further gives infinitely many examples of icc, property (T) groups $\G_n$ measure equivalent to a fixed group $\G$, such that $\El(\G_n)$ are pairwise mutually nonsiomorphic. The first examples of group measure equivalent groups $\G$ and $\La$ giving rise to nonisomorphic group von Neumann algebras were given in \cite{CI}, thereby answering a question of D. Shlyakhtenko.  Note that the examples in \cite{CI} didn't have property (T).

The following theorem is the main von Neumann algebraic result of the section. Some of the arguments used in the proof are very similar to the ones used in the proof of Theorem \ref{semidirectprodreconstruction} and thus we shall just refer the reader to the previous section for these. However, we will include all the details on the new aspects of the proof.     

\begin{theorem} \label{nonisomsemidirect}
Let $Q_1$, $Q_2$, $P_1$, $P_2$ be icc, torsion free, residually finite property (T) groups. Let $Q=Q_1 \times Q_2$ and $P=P_1 \times P_2$. Assume that $N_1 \rtimes Q$, $N_2 \rtimes Q \in \mathcal Rip_T(Q)$ and $M_1 \rtimes P$, $M_2 \rtimes P \in \mathcal Rip_T(P)$. Assume that $\Theta: \El((N_1 \times N_2) \rtimes Q) \rightarrow \El((M_1 \times M_2) \rtimes P)$ is a $\ast$-isomorphism. 

Then one can find a $\ast$-isomorphism, $\Theta_i: \El(N_i) \rightarrow \El(M_i)$, a group isomorphism $\delta: Q \rightarrow P$, a multiplicative character $\eta: Q \rightarrow \mathbb T$, and a unitary $u \in \mathcal U( \El((M_1 \times M_2) \rtimes P))$ such that  for all $ \g \in Q$, $x_i \in N_i$ we have that$$ \Theta((x_1 \otimes x_2)u_{\g})= \eta(\g)u(\Theta_1(x_1) \otimes \Theta_2(x_2)v_{\delta(\g)})u^{\ast}. $$
\end{theorem}
\begin{proof} Let $\emm =\El((M_1 \times M_2) \rtimes P) $ , $\G_i= N_i \rtimes Q$ and let $\tilde \emm = \El(\G_1 \times \G_2)$. Note that $\Theta(\El(N_1))$ and $\Theta(\El(N_2))$ are commuting property (T) subfactors of $\El((M_1 \times M_2) \rtimes P)$. Hence by Theorem~\ref{controlprodpropt2} we have that either 
	\begin{enumerate}
		\item[1)] exists $  i \in \{1,2\}$ such that $\Theta(\El(N_i)) \prec_{\tilde \emm} \El(\G_1)$ or
		\item[2)] $\Theta(\El(N_1 \times N_2)) \prec_{\tilde \emm} \El(\G_1 \times P)$.
	\end{enumerate} 
Assume 1) holds. Then proceeding exactly as in the first part of proof of Theorem ~\ref{commutationcontrolincommultiplication} we have that $\Theta(\El(N_i)) \prec_{\tilde \emm} \El(M_1)$. As $\El(M_1)$ is regular in $\emm$, we conclude using Lemma~\ref{intertwininglower1} that $\Theta(\El(N_i)) \prec_{ \emm} \El(M_1)$.
 
Assume 2). Then the same argument as in the second part of the proof of Theorem~\ref{commutationcontrolincommultiplication} we have that $\Theta(\El(N_1 \times N_2)) \prec_{\tilde \emm} \El(M_1 \rtimes diag(P))$. Thus if $\Theta(\El(N_i)) \nprec \El(M_1)$ for all $i =1,2$, then the same argument as in the last part of Theorem~\ref{commutationcontrolincommultiplication} will lead to a contradiction. 

In conclusion, we have shown that for all $i=1,2$ there exists $j\in 1,2$ such that $\Theta(\El(N_j)) \prec_{\emm} \El(M_i)$. As $\Theta(\El(N_j))$ is regular in $\emm$, we actually have that $\Theta(\El(N_j)) \prec_{\emm}^s \El(M_i)$. Notice that in particular this forces different $i$'s to give rise to different $j$'s. Indeed, otherwise we would have that  $\Theta(\El(N_j)) \prec_{\emm}^s \El(M_1)$ and $\Theta(\El(N_j)) \prec_{\emm}^s \El(M_2)$. Then by \cite[Lemma 2.6]{DHI16}, this would imply that $\Theta(\El(N_j)) \prec_\emm \El(M_1) \cap \El(M_2)=\mc$, which is obviously a contradiction. Therefore we get that either
\begin{enumerate}
	\item[4a)] $\Theta(\El(N_1)) \prec_{\emm}^s \El(M_1)$ and $\Theta(\El(N_2)) \prec_{\emm}^s \El(M_2)$, or
	\item[4b)] $\Theta(\El(N_1)) \prec_{\emm}^s \El(M_2)$ and $\Theta(\El(N_2)) \prec_{\emm}^s \El(M_1)$.
\end{enumerate}
Note that both cases imply that $\Theta(\El(N_1)), \Theta(\El(N_2)) \prec_{\emm}^s \El(M_1 \times M_2)$. Using \cite[Lemma 2.6]{Is16}, we further get that 
\begin{equation}\label{sintert1}
\Theta(\El(N_1 \times N_2)) \prec_{\emm}^s \El(M_1 \times M_2).
\end{equation}

Proceeding in a similar manner, we also have the reverse intertwining $\El(M_1 \times M_2) \prec_{\emm}^s \Theta(\El(N_1 \times N_2))$. Since $\El(M_1 \times M_2)$, $\El(N_1 \times N_2)$ are irreducible, regular subfactors of $\emm$, by \cite[Lemma 8.4]{IPP05} one can find $u \in \euu(\emm)$ such that
\begin{equation}\label{uniteq}
	u \El(M_1 \times M_2)u^{\ast}= \Theta(\El(N_1 \times N_2)).
\end{equation}
Note that $\Theta(\El(Q_1)), \Theta(\El(Q_2))$ are commuting property (T) subfactors of $\El((M_1 \times M_2) \rtimes P)$. Proceeding exactly as in the first part of the proof, we conclude that either $\Theta(\El(Q_i)) \prec_{\tilde \emm} \El(\G_1)$ or $\Theta(\El(Q_1 \times Q_2)) \prec_{\tilde \emm} \El(\G_1 \rtimes P)$. As before, this further implies that either 
\begin{enumerate}
	\item [7)] $ \Theta(\El(Q_i)) \prec_{\emm} \El(M_1)$, or 
	\item[8)]  $\Theta(\El(Q_1 \times Q_2)) \prec_{\emm} \El(M_1 \rtimes diag(P))$.
	\end{enumerate}
	Assume 7). Since by \eqref{uniteq} we also have $\El(M_1) \prec_{\emm}^s \Theta(\El(N_1 \times N_2))$ and hence by \cite[Lemma 3.7]{Va07} we conclude that $\Theta(\El(Q_i)) \prec_{\emm} \Theta(\El(N_1 \times N_2))$. However, this implies that $Q_i$ is finite, which is a contradiction. \\
	Hence, we must have 8).
	Proceeding as in the end of proof of Theorem~\ref{commutationcontrolincommultiplication}, we conclude that $\Theta(\El(Q)) \prec_{\emm} \El(P)$. Thus there exists $\Psi: p\Theta(\El(Q))p \rightarrow \mathcal R:=\Psi(p\Theta(\El(Q))p) \subseteq q\El(P)q$ such that $\Psi(x)v=vx$ for all $x \in p\Theta(\El(Q))p$. Also note that $vv^* \in \mathcal R' \cap q \emm q$ and $v^*v \in (p\Theta(\El(Q))p)' p\emm p$. Since $\mathcal R \subseteq q\El(P)q$ is diffuse and $P \leq (M_1 \times M_2) \rtimes P$ is a malnormal subgroup, we have that $\mathcal{QN}_{q \emm q}(\mathcal R)'' \subseteq q \El(P)q$. Thus $vv^* \in q \El(P)q$ and hence $vp\Theta(\El(Q))pv^*=\mathcal Rvv^* \subseteq q\El(P)q$. Extending $v$ to a unitary $v_0$ in $\emm$ we have that $v_0p\Theta(\El(Q))pv_0^* \subseteq \El(P)$. As $\El(P)$ and $\El(Q)$ are factors, after perturbing $v_0$ to a new unitary we may assume that
	\begin{enumerate}
		\item[9)] $v_0\Theta(\El(Q))v_0^* \subseteq \El(P)$.
	\end{enumerate} 
In a similar manner we have that there exists $w_0 \in \euu(\emm)$
\begin{enumerate}
	\item[10)] $w_0\El(P)w_0^* \subseteq \Theta(\El(Q))$.
\end{enumerate} 
Conditions 9) and 10) together imply that $w_0\El(P)w_0^* \subseteq \Theta(\El(Q)) \subseteq v_0^* \El(P)v_0$. In particular, $v_0w_0\El(P)w_0^*v_0^* \subseteq \El(P)$. Since $P$ is malnormal in $(M_1 \times M_2) \rtimes P$ we have that $v_0w_0 \in \El(P)$ and hence $w_0 \El(P)w_0^*=v_0^* \El(P)v_0.$ Combining this with the above relations we get that
\begin{enumerate}
	\item[11)] $w_0\El(P)w_0^* = \Theta(\El(Q))$.
\end{enumerate}
Since the action $Q \ca(N_1 \times N_2)$ has trivial stabilizers, using conditions 11) and 6), arguing as in the proof of Theorem~\ref{semidirectprodreconstruction} we get that $h_{w_0\El(P)w_0^*}(\Theta(Q))>0$. By \cite[Theorem 3.3]{IPV10} we get that there exists $w_1 \in \euu(\emm)$, and isomorphism $\delta: Q \rightarrow P$ such that\\
$\Theta(u_g)= w_1 v_{\delta(g)}w_1^*$ for all $g \in Q$.\\
Finally, this together with relation 4), proceeding exactly as in the proof of Theorem~\ref{semidirectprodreconstruction} implies the desired conclusion.
\end{proof}

The previous theorem can be used to provide an infinite series of finite index subfactors of a given property (T) II$_1$ factor that are pairwise non-isomorphic. 
 
 \begin{corollary}\label{inffindext} \begin{enumerate}
 		\item[1)] Let $Q_1,Q_2$ be uniform lattices in $Sp(n,1)$ with $n\geq 2$ and let $Q:=Q_1 \times Q_2$. Also let $\cdots \leqslant Q_1^s\leqslant \cdots \leqslant Q_1^2\leqslant Q^1_1\leqslant Q_1$ be an infinite family of finite index subgroups and denote by $Q_s:= Q_1^s\times Q_2\leqslant Q$. Then consider $N_1\rtimes_{\sigma_1} Q,N_2\rtimes_{\sigma_2} Q\in {\mathcal Rip}_T(Q)$ and let $\G= (N_1\times N_2)\rtimes _{\sigma_1\times \sigma_2}  Q$. Inside $\G$ consider the finite index subgroups $\G_s:=(N_1\times N_2)\rtimes _{\sigma_1\times \sigma_2}  Q_s$. Then the family $\{ \mathcal L(\G_s) \,|\, s\in I \}$ consists of pairwise non-isomorphic finite index subfactors of $\mathcal L(\G)$. 
 		\item[2)] Let $\G,\G_n$ be as above. Then $\G_n$ is measure equivalent to $\G$ for all $n \in \mathbb N$, but $\El(\G_n)$ is not isomorphic to $\El(\G_m)$ for $n \neq m$. 
 	\end{enumerate}
 	
 \end{corollary}       

\begin{proof} 1) Assume that $\mathcal L(\G_s)\cong \mathcal L(\G_l)$. Notice that $Q_2$, $Q_1^s$, $Q_1^l$ are torsion free, residually finite property (T) groups. Thus applying Theorem \ref{nonisomsemidirect} we get in particular that $Q_s\cong Q_l$. However since $Q_2$, $Q^s_1$, and $Q_1^l$ are icc hyperbolic, this further implies that $Q^s_1\cong Q^l_1$. However by \cite{Pr76} or the co-hopfian property of one ended hyperbolic groups this implies that $s=l$ and the proof follows.\\
2) As $[\G:\G_n]< \infty$ , $\G_n$ is measure equivalent to $\G$, and hence $\G_n$ is measure equivalent to $\G_m$ for all $n,m \in \mathbb N$. The rest follows from part 1). 
\end{proof}
\vskip 0.06in

\noindent \textbf{Notation} Denote by $\mathcal {ST}$ denote the family of all icc, torsion free, residually finite property (T) groups.\\ 

For further use we record the following elementary result. Its proof is left to the reader.
\begin{proposition} \label{reqdgps}
	Fix $Q$ to be an icc, torsion free, residually finite, hyperbolic property (T) group. For instance, $Q$ can be chosen to be a uniform lattice in $Sp(n,1)$ for $n \geq 2$. Then the family $\mathcal{ST'}= \{G \times Q : G \in \mathcal{ST}\}$ consists of pairwise non-isomorphic groups. 
\end{proposition}
Finally, we present the main application of this section:

\begin{corollary} \label{unctlblet}
	Let $\{Q_{\iota}\}_{\iota \in \mathcal I}$ be an infinite family of pairwise nonisomorphic groups in $\mathcal{ST'}$. Consider the semidirect products $N_{\iota_1} \rtimes_{\sigma_1} Q_{\alpha}$, $N_{\iota_2} \rtimes_{\sigma_2} Q_{\iota} \in \mathcal Rip_T(Q_{\iota})$ for every $\iota \in \mathcal I$. Consider the canonical semidirect product $\G_\iota :=(N_{\iota_1} \times N_{\iota_2}) \rtimes_{\sigma_1\times \sigma_2} Q_{\iota}$ corresponding to the diagonal action $\sigma_1\times \sigma_2$. Then $\{\mathcal L(\G_{\iota})\,|\, \iota \in \mathcal I\}$ is an infinite family of pairwise nonisomorphic group $\rm II_1$ factors with property (T). 
\end{corollary}

\begin{proof}
	This follows directly from Theorem~\ref{nonisomsemidirect} and Proposition~\ref{reqdgps}
\end{proof}
The authors strongly believe the family $\mathcal {ST}$ consists of uncountably many pairwise nonisomorphic groups. In this scenario, Corollary~\ref{unctlblet} would provide an explicit family of uncountably many non-isomorphic property (T) group von Neumann algebras. However, we were unable to find in the literature a reference for whether $\mathcal {ST}$ contains uncountably many nonisomorphic groups. Therefore we leave the following as an open question. 

\begin{question} Find examples of uncountably many non-isomorphic icc property (T) groups $G$ that give non-stably isomorphic  II$_1$ factors $\mathcal L(G)$.
\end{question}


\section{Cartan-rigidity for von Neumann algebras of groups in $\mathcal {R}ip (Q)$}

In this last section we classify the Cartan subalgebras in II$_1$ factors associated with the groups in class $\mathcal {R}ip_{T} (Q)$ and their free ergodic pmp actions on probability spaces (see Theorem~\ref{uniquecartan1}, and Corollary~\ref{cor:nocartan}). Our proofs rely in an essential way on the methods introduced in \cite{PV12} and \cite{CIK13} as well as on the group theoretic Dehn filling discussed in Section \ref{Rip}.  For convenience we include detailed proofs.

First we establish the following general intertwining result regarding crossed product algebras arising from groups in $\mathcal Rip(Q)$.

\begin{theorem}\label{uniquecartan1} Let  $Q=Q_1\times Q_2$ where $Q_i$ are residually finite groups.  For every $i=1,2$ let  $\Gamma_i=N_i\rtimes_{\sigma_i} Q\in \mathcal Rip(Q)$ and denote by $\Gamma = (N_1\times N_2)\rtimes_\sigma Q$ the semidirect product associated with the diagonal action $\sigma=(\sigma_1,\sigma_2): Q \rightarrow Aut(N_1\times N_2)$. Let $\mathcal P$ be a von Neumann algebra together with an action $\Gamma\ca \mathcal P$ and denote by $\mathcal M = \mathcal P \rtimes \Gamma$. Let $p\in \M$ be a projection and let $\A\subset p\M p$ be a masa whose normalizer $\mathscr N_{p\M p}(\A)''\subseteq p \M p$ has finite index.  Then  $\mathcal A \prec_{\mathcal M} \mathcal P$.     
\end{theorem}

\begin{proof} Since $\Gamma_i =N_i\rtimes Q$ is hyperbolic relative to a residually finite group $Q$, then by Theorem \ref{dehnfilling} there exist a non-elementary hyperbolic group $H_i$, a subset $T_i\subseteq N_i$ with $|T_i|\geq 2$ and a normal subgroup $R_i\lhd Q$ of finite index such that we have a short exact sequence $$1\rightarrow \ast_{t\in T_i}R_i^t\hookrightarrow \Gamma_i\overset{\varepsilon_i}{\twoheadrightarrow} H_i\rightarrow 1. $$ 
	In particular there are infinite groups $K_1,K_2$ so that $\ast_{t\in T_i}R_i^t=K_1\ast K_2$.
	
	Denote by $\pi_i: \Gamma\twoheadrightarrow \Gamma_i $ the canonical projection given by $\pi_i((n_1,n_2) q)=n_i q$, for all $(n_1,n_2)q\in (N_1 \times N_2)\rtimes Q=\Gamma $. Then for every $i=1,2$ consider the epimorphism  $\rho_i= \varepsilon_i\circ \pi_i: \Gamma \rightarrow H_i$.  
	Following \cite[Section 3]{CIK13} consider the $\ast$-embedding $\Delta^{\rho_i}: \M \rightarrow \M \bar\otimes \El(H_i):=\tilde \M_i$ given by $\Delta^{\rho_i}(xu_g )= x u_g \otimes v_{\rho_i(g)}$ for all $x\in \M$, $g\in \Gamma$. Here $(u_g)_{g\in \Gamma}$ and $(v_{h})_{h\in H_i}$ are the canonical group unitaries in $\P\rtimes \Gamma$ and $\El(H_i)$, respectively.  As $\mathcal A$ is amenable, \cite[Theorem 1.4]{PV12} implies either a) $\Delta^{\rho_i}(\A)\prec_{\tilde \M_i} \M\otimes 1$ or b) the normalizer  $\Delta^{\rho_i}(\mathscr N_{p\M p}(\A)'')$ is amenable relative to $\M \otimes 1$ inside $\tilde \M_i$. Assume b) holds. As $\mathscr N_{p\M p}(\A)''\subseteq p \M p$ has finite index it follows that $\Delta^{\rho_i}(p\mathcal M p)$ is amenable relative to $\M \otimes 1$ inside $\tilde \M_i$. However, using \cite[Proposition 3.5]{CIK13} this further entails that $H_i$ is amenable, a contradiction. Thus a) must hold and using \cite[Proposition 3.4]{CIK13} we get that $\A \prec_\M \P\rtimes \ker(\rho_i)$. Let $\mathcal N=\P \rtimes \ker(\rho_i)$ and using \cite[Proposition 3.6]{CIK13} we can find a projection $0\neq q\in \mathcal N$, a masa $\B\subset q \mathcal N q$ with $\Q:=\mathscr N_{q \mathcal N q}(\B)''\subseteq q\mathcal N q$ has finite index. In addition one can find  projections $0\neq p_0 \in \A$, $0\neq q'_0\in \B'\cap p\M p$ and a unitary $u\in \M$ such that $u(\A p_0 )u^*=\B p_0$.
	
	To this end observe the restriction homomorphism $\pi_i: \ker(\rho_i)\rightarrow K_1\ast K_2$ is an epimorphism with $\ker(\pi_i)= N_{\hat i}$. As before, consider the $\ast$-embedding  $\Delta^{\pi_i}: \mathcal N \rightarrow \mathcal N \bar\otimes \El(K_1\ast K_2)$ given by given by $\Delta^{\pi_i}(xu_g )= x u_g \otimes v_{\pi_i(g)}$ for all $x\in \P$, $g\in \ker(\rho_i)$. Denote by $\tilde {\mathcal N}_i:= \mathcal N \bar\otimes \El(\ker(\rho_i))$. Also fix $0\neq z\in \mathcal Z (\Q'\cap q\mathcal N q)$. Since $\Delta^{\pi_i}(\B z)\subset \mathcal N \bar \otimes \El(K_1\ast K_2)$ is amenable then using $\cite{Io12,Va13}$ one of the following must hold: c) $\Delta^{\pi_i}(\Q z)$ is amenable relative to $\mathcal N \otimes 1$ inside $\tilde {\mathcal N}_i$; d) $ \Delta^{\pi_i}(\Q z)  \prec_{\tilde {\mathcal N}_i} \mathcal N \bar\otimes \El(K_j)$ for some $j=1,2$; e) $\Delta^{\pi_i}(\mathcal B z) \prec_{\tilde {\mathcal N}_i} \mathcal N \otimes 1$.
	
	Assume c) holds. As $\Q\subseteq q\mathcal N q$ is finite index so is $\Q z\subseteq z\mathcal N z$ and \cite[Lemma 2.4]{CIK13} implies that $z\mathcal N z \prec^s \Q z$ and using \cite[Proposition 2.3 (3)]{OP07} we get that $\Delta^{\pi_i}(z\mathcal N z)$ is amenable relative $\mathcal N \otimes 1$ inside $\tilde {\mathcal N}_i$. Thus \cite[Proposition 3.5]{CIK13} implies that $K_1\ast K_2$ is amenable, a contradiction.
	Assume d) holds. By \cite[Proposition 3.4]{CIK13} we have that $ \Q z  \prec \mathcal P \rtimes (\pi_i )^{-1}(K_j)$ and using \cite[Lemma 2.4 (3)]{DHI16} one can find a projection $0\neq r\in \mathcal Z(\Q z'\cap z\mathcal N z)$ such that $\Q r \prec^s \P \rtimes  (\pi_i )^{-1}(K_j)$. Since $\Q z \subseteq z\mathcal N z$ is finite index then so is $\Q r \subseteq r \mathcal N r$ and thus $r\mathcal N r \prec_{\mathcal N} \Q r$. Therefore using \cite[Lemma 2.4(1)]{DHI16} (or \cite[Remark 3.7]{Va07}) we conclude that $\mathcal N\prec  \mathcal P \rtimes (\pi_i )^{-1}(K_j)$. However this implies that $\pi^{-1}(K_j) \leqslant \ker(\rho_i)$ is finite index, a contradiction. Hence e) must hold and using \cite[Proposition 3.4]{CIK13} we further get that $\mathcal B z \prec_{\mathcal N} \P\rtimes N_{\hat i}$. Since this holds for all $z$ we conclude that $\mathcal B \prec^s_{\mathcal N} \P\rtimes N_{\hat i}$. This combined with the prior paragraph clearly implies that $\A \prec \P \rtimes N_{\hat i}$. 
	
	Since all the arguments above still work and the same conclusion holds if one replaces $\A$ by $\A a$ for any projection $0\neq a\in \A$ one actually has $\A \prec^s_\M \P \rtimes N_{\hat i}$. Since this holds for all $i=1,2$, using \cite[Lemma 2.8(2)]{DHI16} one concludes that $\A \prec_\M \P$, as desired. \end{proof}

\begin{corollary} \label{cor:nocartan} Let $\Gamma$ be a group as in the previous theorem and let $\Gamma \curvearrowright X$ be a free ergodic pmp action on a probability space. Then the following hold:
	\begin{enumerate}
		\item The crossed product $L^\infty(X)\rtimes \Gamma$ has unique Cartan subalgebra;
		\item The group von Neumann algebra $\El(\Gamma)$ has no Cartan subalgebra.
	\end{enumerate}
	
\end{corollary}

\begin{proof} 1. Let $\mathcal A \subset L^\infty(X)\rtimes \Gamma=:\mathcal M$ be a Cartan subalgebra. By Theorem \ref{uniquecartan1} we have that $\mathcal A \prec_{\mathcal M}L^\infty(X)$ and since $L^\infty(X)\subseteq \mathcal M$ is Cartan then \cite[Theorem]{Po01}  gives the conclusion.
	2. If $\mathcal A \subset \El(\Gamma)$ is a Cartan subalgebra then Theorem \ref{uniquecartan1} implies that $\mathcal A \prec \mathbb C 1$ which contradicts that $\mathcal A$ is diffuse.\end{proof}

\section*{Acknowledgments}
 The authors are grateful to Prof.\ Jesse Peterson for many useful suggestions. The authors would also like to thank Prof.\ Mikhail Ershov for many helpful comments, and especially for pointing us to the reference \cite{Pr76}.  Part of this work was completed while the third author was visiting the University of Iowa as part of the student exchange program under the NSF grants FRG-DMS-1854194 and FRG-DMS-1853989. He is grateful to the mathematics department there for their hospitality. 
 
 Last but not least the authors are grateful to the two anonymous referees for their numerous comments and suggestions which simplified some of the proofs and greatly improved the readability and the overall exposition of this paper.

\noindent
\textsc{Department of Mathematics, The University of Iowa, 14 MacLean Hall, Iowa City, IA 52242, U.S.A.}\\
\email {ionut-chifan@uiowa.edu} \\
\email{sayan-das@uiowa.edu}\\
\textsc{Department of Mathematics, Vanderbilt University, 1326 Stevenson Center, Nashville, TN 37240, U.S.A.}\\
\email{krishnendu.khan@vanderbilt.edu}

\end{document}